\documentclass[11pt,dvipsnames]{amsart}
\usepackage{amsmath,amssymb,amsfonts}
\usepackage[mathscr]{eucal}
\usepackage[margin=1in]{geometry}
\usepackage{mathtools}
\usepackage{quiver}
\usepackage{comment}

\usepackage{tikz, tikz-cd}
\usetikzlibrary{matrix,arrows,decorations.pathmorphing,cd,patterns,calc,backgrounds,patterns}
\tikzset{dummy/.style= {circle,fill,draw,inner sep=0pt,minimum size=1.2mm}}
\tikzset{vertex/.style={fill, circle, minimum size=.1cm, inner sep=0pt}}

\usepackage[pagebackref, colorlinks, citecolor=Black, linkcolor=Black, urlcolor=Black]{hyperref}
\hypersetup{linktoc=all,} 

\usepackage[textwidth=25mm, textsize=tiny]{todonotes}
\usepackage[final]{microtype}

\usepackage{enumerate}

\usepackage{appendix}

\numberwithin{equation}{section} 
\numberwithin{figure}{section}
\usepackage[nameinlink,capitalise,noabbrev]{cleveref}

\newcommand{\newrefformat}[2]{}

\usepackage{stmaryrd}

\newcommand\restr[2]{{
  \left.\kern-\nulldelimiterspace 
  #1 
  \vphantom{\big|} 
  \right|_{#2} 
  }}

\crefname{lemma}{Lemma}{Lemmas}
\crefname{theorem}{Theorem}{Theorems}
\crefname{definition}{Definition}{Definitions}
\crefname{proposition}{Proposition}{Propositions}
\crefname{remark}{Remark}{Remarks}
\crefname{corollary}{Corollary}{Corollaries}
\crefname{equation}{Equation}{Equations}
\crefname{construction}{Construction}{Constructions}
\crefname{ex}{Example}{Examples}
\crefname{appsec}{Appendix}{Appendices}
\crefname{subsection}{Subsection}{Subsections}

\theoremstyle{plain}
\newtheorem{theorem}[equation]{Theorem}
\newtheorem{corollary}[equation]{Corollary}
\newtheorem{proposition}[equation]{Proposition}
\newtheorem{lemma}[equation]{Lemma}

\newtheorem{introtheorem}{Theorem}
 
\crefname{introtheorem}{Theorem}{Theorems}

\theoremstyle{definition}
\newtheorem{definition}[equation]{Definition}
\newtheorem{example}[equation]{Example}
\newtheorem{remark}[equation]{Remark}

\author[M. E. Calle]{Maxine E. Calle}             
\email{callem@sas.upenn.edu}
\address{Department of Mathematics,
         University of Pennsylvania,
         Philadelphia, PA, 19104,
         USA}

\author[D. Chan]{David Chan}
\email{chandav2@msu.edu}
\address{Department of Mathematics,
         Michigan State University,
         East Lansing, MI, 48824
         USA}

\author[A. Mejia]{Andres Mejia}
\email{am3399@cam.ac.uk}
\address{Centre for Mathematical Sciences, 
Wilberforce Road, Cambridge, 
CB3 0WA, United Kingdom}

\date{\today}



\newcommand{\ZZ}{\mathbb{Z}}

\newcommand{\abs}[1]{\left\lvert#1\right\rvert}

\newcommand{\cat}[1]{\mathscr{#1}}

\newcommand{\Loop}{\Omega}

\newcommand{\op}{\operatorname{op}}

\newcommand{\id}{\operatorname{id}}

\newcommand{\Alg}{\operatorname{Alg}}
\newcommand{\Map}{\operatorname{Map}}
\newcommand{\Sing}{\operatorname{Sing}}

\newcommand{\cof}{\rightarrowtail}
\newcommand{\Mod}{\mathcal{M}od}
\newcommand{\Set}{\mathrm{Set}}
\newcommand{\FP}{\mathrm{FP}}
\newcommand{\Perf}{\mathrm{Perf}}
\newcommand{\fd}{\mathrm{fd}}
\newcommand{\Ab}{\mathrm{Ab}}

\newcommand{\Cat}{\mathrm{Cat}}
\newcommand{\Gpd}{\mathrm{Gpd}}
\newcommand{\Fun}{\mathrm{Fun}}
\newcommand{\Span}{\mathrm{Span}}
\newcommand{\Top}{\mathrm{Top}}
\newcommand{\Sp}{\mathrm{Sp}}
\newcommand{\Wald}{\mathrm{Wald}}

\newcommand{\ex}{\mathrm{ex}}

\newcommand{\eff}{\mathrm{eff}}

\newcommand{\Coeff}{\mathrm{Coeff}}
\newcommand{\tr}{\mathrm{tr}}
\newcommand{\res}{\mathrm{res}}

\newcommand{\Mack}{\mathrm{Mack}}
\newcommand{\ps}{\mathrm{ps}}
\newcommand{\Proj}{\mathrm{Proj}}

\begin{document}
\begin{abstract}
    We introduce a version of algebraic $K$-theory for coefficient systems of rings which is valued in genuine $G$-spectra for a finite group $G$. We use this construction to build a genuine $G$-spectrum $K_G(\ZZ[\underline{\pi_1(X)}])$ associated to a $G$-space $X$, which provides a home for equivariant versions of classical invariants like the Wall finiteness obstruction and Whitehead torsion. We provide a comparison between our $K$-theory spectrum and the equivariant $A$-theory of Malkiewich--Merling via a genuine equivariant linearization map.
\end{abstract}

\title[A linearization map for $A_G(X)$]{A linearization map for\\genuine equivariant algebraic $K$-theory}

\maketitle

\setcounter{tocdepth}{1}
\tableofcontents
\section{Introduction}

The algebraic $K$-theory of a ring $R$ is a spectrum $K(R)$ whose homotopy groups encode information about the structure of $R$. We can also use $K$-theory to obtain information about a topological space $X$ by studying $K(\ZZ[\pi_1(X)])$. The $K$-theory of $\ZZ[\pi_1(X)]$ encodes important geometric invariants, like the Euler characteristic, Wall finiteness obstruction, and Whitehead torsion \cite{wall,whiteheadTorsion}. 

When the space $X$ comes with extra symmetries, in the form of a group action by a finite group $G$, we want to adapt these invariants to capture these symmetries. In this setting, the group ring $\ZZ[\pi_1(X)]$ inherits a $G$-action and so its $K$-theory spectrum $K(\ZZ[\pi_1(X)])$ is a spectrum with $G$-action, i.e.\ a \textit{na\"ive $G$-spectrum}. However, this na\"ive $G$-spectrum is not refined enough to fully capture the $G$-homotopy type of $X$. There are equivariant analogues of the $K$-groups of $\ZZ[\pi_1(X)]$ that are known to contain versions of Wall finiteness \cite{Baglivo, andrzejewski:1986} and Whitehead torsion \cite{illman, araki/kawakubo:scob}. In particular, L\"uck constructs a spectrum whose low-degree homotopy groups (and quotients thereof) house these invariants \cite{Luck}.

In this paper, we show that these invariants live in the $G$-fixed points of a \textit{genuine $G$-spectrum} $K_G(\ZZ[\underline{\pi_1(X)}])$, which is a genuinely equivariant refinement of $K(\ZZ[\pi_1(X)])$. We use the framework of \textit{spectral Mackey functors} to model genuine $G$-spectra, following Guillou--May \cite{GuillouMay:2011}. 
Since $\ZZ[\pi_1(X)]$ has a $G$-action, one way we could obtain a genuine $K$-theory spectrum is via work of Merling \cite{merling:2017}. For a ring $R$ with $G$-action, Merling constructs a genuine $G$-spectrum $K_{\theta}(R)$. When $|G|$ is invertible in $R$, $K_{\theta}(R)$ can be understood through its fixed points
\[
    K_{\theta}(R)^H\simeq K(R^H_{\theta}[W_GH]),
\]where $R^H_{\theta}[W_GH]$ is the \emph{twisted group ring} (see \cref{defn: twisted group ring}) constructed from the $H$-fixed points $R^H$ through the action of the Weyl group $W_GH = N_G(H)/H$. While $K_{\theta}(\ZZ[\pi_1(X)])$ captures the data of the fixed points of the ring $\ZZ[\pi_1(X)]$, it fails to accurately describe the data of the fixed points of the space $X$ due to the fact that $(\ZZ[\pi_1(X)])^H\neq \ZZ[\pi_1(X^H)]$. 

Instead, to construct an equivariant version of $K({\ZZ[\pi_1(X)]})$, we follow the philosophy of Elmendorf \cite{elmendorf}, which says that the equivariant homotopy theory of a $G$-space $X$ can be understood by studying the collection of fixed-point spaces $\{X^H\}_{H\leq G}$ along with \textit{restriction maps} $X^H\to X^K$ whenever $K\leq H$. In particular, this collection is a space-valued presheaf on the \emph{orbit category} $\mathcal{O}_G$, which is the full subcategory of $G$-sets spanned by the objects $G/H$ for $H\leq G$. Such a presheaf is sometimes called a \emph{coefficient system of spaces}. 

In general, if $\cat{C}$ is any category then a $\cat{C}$-valued presheaf on $\mathcal{O}_G$ is called a coefficient system in $\cat{C}$. In this paper, we build an equivariant $K$-theory machine that takes as input coefficient systems in the category of rings. The output of this machine is a genuine $G$-spectrum, whose fixed-point spectra admit a highly desirable splitting. 

\begin{introtheorem}[\cref{KTheoryExists,K theory splitting}]\label{intro K theory splitting}
    Let $G$ be a finite group. For any coefficient system of rings $\underline{S}\colon \mathcal{O}^{\rm op}_{G}\to \mathrm{Ring}$ there is a genuine $G$-spectrum $K_G(\underline{S})$ whose $G$-fixed points split
    \[
        K_G(\underline{S})^G \simeq \prod\limits_{(H)\leq G} K(\underline{S}^H_{\theta})
    \]
    where $K(\underline{S}^H_{\theta})$ is the ordinary algebraic $K$-theory of the twisted group ring $\underline{S}^H_{\theta} = \underline{S}(G/H)_\theta[WH]$.
\end{introtheorem}

Our construction of genuine $G$-spectra makes use of some technical result which reframes a result of \cite{malkiewich/merling:2016} in terms of $2$-categorical input.  The main results are \cref{MMProp4.6} and \cref{Theorem: construction examples and maps} which describe exactly what data needs to be provided to construct examples of spectral Mackey functors and morphisms between them.  We expect that these results should have applications beyond the present work.

An important example of a coefficient system of rings in this paper is the following. Given a $G$-space $X$ for a finite group $G$, we can form a coefficient system of rings $\ZZ[\underline{\pi_1(X)}]\colon \mathcal{O}_G^{\rm op}\to {\rm Ring}$ given by \[
\ZZ[\underline{\pi_1(X)}](G/H):=\mathbb Z[\pi_1(X^H)].
\] We then obtain a genuine $G$-spectrum $K_G(\ZZ[\underline{\pi_1(X)}])$ and when $X^H$ is connected for all $H$ the $G$-fixed points split as\[
K_G(\ZZ[\underline{\pi_1(X)}])^G \simeq \prod\limits_{(H)\leq G} K(\ZZ[\pi_1(X^H_{hW_GH})]). 
\] 
In this case, we recover a result of L\"uck \cite[\S 10]{Luck}, although we note that \cref{intro K theory splitting} does not seem to be comparable to L\"uck's work in general. Via this identification with L\"uck's constructions, we immediately deduce that the genuine $G$-spectrum $K_G(\ZZ[\underline{\pi_1(X)}])$ provides a home for equivariant analogues of the Euler characteristic, the Wall finiteness obstruction, and Whitehead torsion.

The spectrum $K(\ZZ[\pi_1(X)])$ is closely related to another spectrum $A(X)$ called the \textit{Waldhausen $A$-theory} of $X$, whose homotopy groups describe the stable geometry of $X$. For example, the stable parametrized $h$-cobordism theorem relates the $A$-theory of a smooth manifold $M$ to a space of stable $h$-cobordisms which contains information about the diffeomorphism group of $M$ \cite{waldhausen:1983,WaldhausenJahrenRognes}. The spectrum $A(X)$ comes equipped with a \emph{linearization map} \[\ell\colon A(X)\to K(\ZZ[\pi_1(X)]),\] which is known to be $2$-connected \cite{waldhausen:1983}. Consequently, the Wall finiteness obstruction and Whitehead torsion (which live in quotients of $K_0(\ZZ[\pi_1(X)])$ and $K_1(\ZZ[\pi_1(X)])$, respectively) can be lifted to elements of the homotopy groups of $A(X)$. The linearization, together with the Dundas--Goodwillie--McCarthy theorem, has been utilized to great effect for computations \cite{Dundas,Klein-Rognes}.

In \cref{sec:Equivariant Linearization}, we construct a version of the linearization map which relates $K_G(\ZZ[\underline{\pi_1(X)}])$ to the \textit{genuine equivariant $A$-theory spectrum} of $X$ constructed in \cite{malkiewich/merling:2016}. Given a $G$-space $X$ (for $G$ a finite group), Malkiewich--Merling construct a genuine $G$-spectrum $A_G(X)$ which fits into an equivariant stable parametrized $h$-cobordism theorem \cite{malkiewich/merling:2022}. An important property of their construction, proved in \cite{badzioch/dorabiala:2016}, is that the fixed points of $A_G(X)$ split as\begin{equation}\label{A theory splitting}
    A_G(X)^G \simeq \prod_{(H) \leq G} A(X^{H}_{hW_GH}),
\end{equation}
which is reminiscent of the splitting of \cref{intro K theory splitting}. A consequence of our work is that this genuine $G$-spectrum encodes the expected geometric invariants of the $G$-space $X$.

\begin{introtheorem}[{\cref{corollary: wall,corollary: whitehead}}]\label{intro wall}
    For a $G$-connected $G$-space $X$, the equivariant Euler characteristic, Wall finiteness obstruction, and Whitehead torsion are realized as elements in the homotopy groups of $A_G(X)$.
\end{introtheorem}

These equivariant invariants were originally constructed in work of Andzejewki, Baglivo, L\"uck, and Illman \cite{andrzejewski:1986,Baglivo,Luck,illman}, and our work provides, to our knowledge, the first connection between these invariants and the genuine equivariant $A$-theory. To prove \cref{intro wall}, we construct a map of genuine $G$-spectra $L\colon A_G(X)\to K_G(\ZZ[\underline{\pi_1(X)}])$ and respects the splittings on fixed points.

\begin{introtheorem}[\cref{linearization is 2 connected}] \label{intro linearization theorem}
    For any based $G$-space $X$ with $X^H$ connected for all $H \leq G$, there is an equivariant linearization map
    \[
        L\colon A_G(X)\to K_G(\underline{\ZZ[\pi_1X]})
    \]
    whose underlying map of spectra is the classical linearization map.  Moreover, this map is equivariantly $2$-connected, meaning it induces isomorphisms on $\pi_0^H$ and $\pi_1^H$ for all $H\leq G$.
\end{introtheorem}

The following theorem shows that the splittings of \cref{A theory splitting} and \cref{intro K theory splitting} are compatible, allowing us to deduce facts about equivariant linearization from the non-equivariant situation (such as the connectivity claim of \cref{intro linearization theorem}).

\begin{introtheorem}[\cref{linearization splitting}]\label{intro theorem L compatible}
    Let $X$ be a $G$-space.  For any subgroup $H\leq G$ there is a homotopy commutative diagram of spectra:
    \[
        \begin{tikzcd}
            A_G(X)^G \ar{d} \ar["L^G"]{r} & K_G(\underline{\ZZ[\pi_1X]})^G \ar{d}\\
            A(X^H_{hWH}) \ar[swap, "\ell"]{r} & K(\ZZ[\pi_1(X^H_{hWH})])
        \end{tikzcd}
    \]
    where the unlabeled vertical maps are projection maps coming from the splittings and $\ell$ is the classical linearization map for the space $X^H_{hWH}$.
\end{introtheorem}

\subsection{Outline}
In \cref{sec:bg on coeff sys}, we review necessary background on coefficient systems needed for \cref{sec:KT of CS}, where we construct the $K$-theory of a coefficient ring $\underline{S}$ as a genuine $G$-spectrum $K_G(\underline{S})$ and prove \cref{intro K theory splitting} (see \cref{subsec:KT split}).
In \cref{sec:Equivariant Linearization}, we first recall the construction of equivariant $A$-theory (see \cref{sec:bckgd on A theory}) and then define and analyze the equivariant linearization map, proving \cref{intro linearization theorem} and \cref{intro theorem L compatible}. Finally, in \cref{subsec: geometric invariants}, we apply earlier results of \cref{sec:Equivariant Linearization} to study the geometric invariants of $G$-spaces, proving \cref{intro wall}. There are two appendices.   The first contains a technical proof of a theorem used in \cref{sec:KT of CS}. The second contains a number of technical results which deal with the construction of morphisms of genuine $G$-spectra using the framework of spectral Mackey functors.

\subsection{Notation and Conventions}

Throughout the paper we fix a finite group $G$ and use $H$ to denote a subgroup of $G$. For any $g\in G$, we denote the conjugate subgroup $g^{-1}Hg$ by $H^g$.  
A coefficient ring will be denoted by $\underline{S}$ and an arbitrary coefficient system by $\underline{M}$.  We denote the evaluation of any coefficient system $\underline{M}$ at the orbit $G/H$ by either $M(G/H)$ or $\underline{M}^H$.  We will often abbreviate the twisted group ring $\underline{S}^H_{\theta}[WH]$ to simply $\underline{S}^H_{\theta}$.

Unless otherwise stated, when we refer to the $K$-theory of a Waldhausen category we will always mean the $K$-theory symmetric spectrum obtained by via $S_{\bullet}'$-construction as in \cite{BM:08}.  This is a varaint of Waldhausen $S_{\bullet}$-construction which is better suited to our purposes; see \cref{subsec: appendix S dot prime} for a review of this construction.

\subsection{Acknowledgments} We would like to thank Anna Marie Bohmann, Anish Chedalavada, Liam Keenan, Cary Malkiewich, Mona Merling, and Maximilien P\'eroux for helpful conversations, as well as the referee for their comments which greatly improved the paper.  We would also like to thank Bohmann and Merling for providing financial support which facilitated this collaboration. The first- and third-named authors were partially supported by NSF grant DGE-1845298. The second-named author was partially supported by NSF grants DMS-2104300 and DMS-2135960.

\section{Background on coefficient systems}\label{sec:bg on coeff sys}

In this section we review some background on coefficient systems and coefficient rings.  We begin by showing that the category of modules over coefficient rings is an abelian cateogry with enough projectives.  We then recall the notions of induction, restriction, and conjugation of coefficient systems and determine some of their basic properties needed for the construction of the equviariant $K$-theory spectrum of a coefficient ring in \cref{sec:KT of CS}.  Let $\mathcal{O}_G$ denote the orbit category of $G$, i.e.\ the full subcategory of finite $G$-sets on the objects $G/H$ for subgroups $H\leq G$.

\begin{definition}
    A \textit{coefficient system} is a functor $\underline{M}\colon \mathcal{O}^{op}_G\to {\rm Ab}$ from the orbit category to the category of abelian groups. We denote the evaluation of a coefficient system $\underline{M}$ at orbit $G/H$ by $\underline{M}^H$.
    
    A morphism of coefficient systems is a natural transformation of functors, and the category of $G$-coefficient systems is denoted $\mathrm{Coeff}^G$.
\end{definition}

When $H\leq K$ there is a canonical map $G/H\to G/K$ in $\mathcal{O}_G$ which sends $eH$ to $eK$.  If $\underline M$ is a coefficient system we call the induced map $R^K_H\colon \underline{M}^K\to \underline{M}^H$ the \emph{restriction} from $K$ to $H$.  When $H=K$, the endomorphism group $\mathcal{O}_G(G/H,G/H)\cong N_G(H)/H =W_GH$ is called the \emph{Weyl group} of $H$ in $G$. There is a natural left action of $W_GH$ on $\underline{M}^H$.

\begin{definition}
    We call a coefficient system $\underline{S}$ a \textit{coefficient system of rings} or a \textit{coefficient ring} if $\underline{S}$ is valued in the category of rings, i.e. $\underline{S}^H\in {\rm Ring}$ for all $H\leq G$ and the restriction maps and Weyl group actions are ring homomorphisms.
\end{definition}

\begin{example}\label{free coefficient systems}
    Let $X$ be a finite $G$-set and define a coefficient system $A_X$ by 
    \[
        A_X^H = \ZZ\{X^H\},
    \]
    the free abelian group on the $H$-fixed points of $X$. For any $X$ there are natural bijections 
    \[
        X^H\cong \mathrm{Set}^G(G/H,X)
    \] 
    which supply the restriction maps for $A_X$. Explicitly, if $f\colon G/H\to G/K$ is a map in $\mathrm{Set}^G$, there is an induced pullback map $f^*\colon A_X^K\to A_X^H$. The $W_GH$-action on $A_X^H$ is through the action of $WH$ on the fixed point set $X^H$.  This makes $A_X$ a functor on the opposite of the orbit category.  
\end{example}

\begin{example}\label{ex:constant Z coeff system}
     When $X=G/G$, we have $A_{G/G}^H  \cong \ZZ$ for all $H$ and all restriction maps are the identity; this is sometimes called the constant $\ZZ$ coefficient system. 
\end{example}

\begin{example}\label{ex:FP(R)}
    Let $R$ be a ring with $G$-action through ring automorphisms. The \emph{fixed point} coefficient ring of $R$ is the coefficient ring defined by \[\FP(R)^H = R^H,\] the $H$-fixed points of $R$.  The restriction maps are given by inclusion of fixed points and the Weyl group actions are given by the natural action on fixed points.
\end{example}

The category of coefficient systems has a monoidal product $\square$ given by taking levelwise tensor products.  That is, \[(\underline{M}\square \underline{N})^H = \underline{M}^H\otimes \underline{N}^H.\]  Note that the unit is the constant $\ZZ$ coefficient system $A_{G/G}$ from \cref{ex:constant Z coeff system}.  It is a straightforward exercise to see that coefficient rings are exactly the monoids in this monoidal category.  With this observation, we can make sense of modules over a coefficient ring in the usual way and study their algebraic theory.  We denote the category of (left) modules over $\underline{S}$ by $\mathrm{Mod}_{\underline{S}}$.

We now give a more concrete characterization of $\underline{S}$-modules, which requires the following definition.

\begin{definition}\label{defn: twisted group ring}
    Let $R$ be a ring with $G$-action and denote the action of $g\in G$ on $r\in R$ by $r^g$.  The twisted group ring $R_{\theta}[G]$ of $R$ is the same abelian group as the usual group ring $R[G]$ but with multiplication 
    \[
        (r_1g_1)(r_2g_2) = (r_1r_2^{g_1})g_1g_2.
    \]
    We will often abbreviate $R_{\theta}[G]$ to simply $R_{\theta}$ when the group $G$ is clear.
\end{definition}

\begin{lemma}\label{twisted group ring modules}
    Let $R$ be a ring with $G$-action.  A module over $R_{\theta}[G]$ is equivalent to an $R$-module $M$ together with an action of $G$ on the abelian group of $M$ subject to the condition $g(rm) = r^g(gm)$ for all $g\in G$, $m\in M$ and $r\in R$.
\end{lemma}
\begin{proof}
    The inclusion $R\to R_{\theta}[G]$ is a ring map and thus every $R_{\theta}[G]$-module is an $R$-module by restriction of scalars.  If $M$ is an $R_{\theta}[G]$-module, the actions by elements $g\in R_{\theta}[G]$ determine a $G$-action on $M$ that satisfies the stated condition.  Conversely, if $M$ is an $R$-module with appropriate $G$-action then we define $(rg)\cdot m = r(gm)$ and it is readily checked that this defines an $R_{\theta}[G]$-module structure on $M$.
\end{proof}

\begin{example}
    Let $F\subset L$ be a finite Galois extension with Galois group $G$.  Then for every $F$-vector space $V$, the extended $L$-module $L\otimes_F V$ is a module over $L_{\theta}[G]$.  The starting point of Galois descent is that the converse is also true: every $L_{\theta}[G]$-module is isomorphic to $L\otimes_F V$ for some $V$.
\end{example}

The twisted group ring construction allows us to give a levelwise condition on coefficient systems to check if they are modules over a coefficient ring.

\begin{lemma}\label{modules over coeff rings}
    Let $\underline{S}$ be a coefficient ring.  A coefficient system $\underline{M}$ is a module over $\underline{S}$ if and only if for all $H\leq G$ the abelian group $\underline{M}^H$ is a module over the twisted group ring $\underline{S}^H_{\theta}[W_GH]$ such that the action maps commute with the restrictions in the evident way.
\end{lemma}

\begin{proof}
    If $\underline{M}$ is an $\underline{S}$-module then by definition we have maps $\mu_H\colon \underline{S}^H\otimes \underline{M}^H\to \underline{M}^H$ for all $H$ which gives $\underline{M}^H$ a module structure over $\underline{S}^H$ for all $H$.  Furthermore, since $\underline{M}$ is a coefficient system, $\underline{M}^H$ comes with a $W_GH$-action and so we just need to check that this action satisfies the condition of \cref{twisted group ring modules}.  This is immediate from the fact that the maps $\mu_H$ assemble into a map of coefficient systems, and hence are $W_GH$-equivariant. Thus $\underline{M}^H$ is a module over the twisted group rings $\underline{S}^H_{\theta}[WH]$.  The converse is similar.
\end{proof}

For the rest of the paper we abbreviate $\underline{S}^H_{\theta}[W_GH]$ to simply $\underline{S}^H_{\theta}$.
We highlight a few important classes of $\underline{S}$-modules.
\begin{definition}
For a $G$-set $X$, we write $\underline{S}_{X}$ for the coefficient system $\underline{S}\square A_X$, where $A_X$ is the coefficient system from \cref{free coefficient systems}; this is always a left $\underline{S}$-module.  We say
\begin{enumerate}
    \item an $\underline{S}$-module is \emph{(finite) free} if it is isomorphic to $\underline{S}_{X}$ for some (finite) $G$-set $X$, 
    \item an $\underline{S}$-module $M$ is \emph{finitely generated} if it admits a surjection $\underline{S}_{X}\twoheadrightarrow \underline{M}$ where $X$ is some finite $G$-set,
    \item an $\underline{S}$-module is \emph{projective} if it is a direct summand of a free module.
\end{enumerate}
\end{definition}
\begin{remark}
    Being a free (or projective) $\underline{S}$-module is not a levelwise condition.  That is, $\underline{M}^H$ could be a free $\underline{S}^H$-module for all $H\leq G$ but $\underline{M}$ might not be a free $\underline{S}$-module.
    
    For example, let $G = C_2$ be the cyclic group with $2$-elements and let $\underline{S} = A_{G/G}$ be the unit. A $\underline{S}$-module is therefore just a coefficient system. Define $\underline{M}$ by $\underline{M}^{C_2} = \ZZ$ and $\underline{M}^e = 0$.  The evident surjection $\underline{S}\to\underline{M}$ does not split, implying that $\underline{M}$ is neither projective nor free. This example also shows that it does not suffice for $\underline{M}^H$ to be a free module over $\underline{S}^H_{\theta}$ for all $H$. 
\end{remark}

The following lemmas justify the use of the terminology ``free'' and ``projective.''

\begin{lemma}
    For any $\underline{S}$-module $\underline{M}$ and any $H\leq G$ there is an isomorphism of abelian groups $\mathrm{Mod}_{\underline{S}}(\underline{S}_{G/H},\underline{M})\cong \underline{M}^H$.
\end{lemma}
\begin{proof}
    We have a string of isomorphisms $\mathrm{Mod}_{\underline{S}}(\underline{S}_{G/H},\underline{M})\cong \mathrm{Coeff}^G(A_X,\underline{M})\cong \underline{M}^H$ where the first uses the usual extension-restriction of scalars adjunction and the second is the Yoneda lemma.
\end{proof}

The proof of the next lemma is identical to the usual proof for projective modules over a ring.

\begin{lemma}
    A finitely generated $\underline{S}$-module $\underline{M}$ is projective if and only if the functor $\mathrm{Mod}_{\underline{S}}(\underline{M},-)$ is right exact.
\end{lemma}

The category $\mathrm{Mod}_S$ is an abelian category, and the subcategory of projectives is an exact subcategory in the sense of \cite[\S II.7]{Kbook}.

\begin{corollary}
    The abelian category $\mathrm{Mod}_{S}$ has enough projectives.  The subcategory of projective objects is an exact category with admissible monics and epics given by split morphisms which are levelwise injections and surjections, respectively.
\end{corollary}

The category of finitely generated projective $\underline{S}$-modules is denoted $\mathrm{Proj}_{\underline{S}}$.

\subsection{Induction and restriction}\label{subsec:induction restriction}

For $H\leq K\leq G$, there is a functor $\rho\colon \mathcal{O}_H\to \mathcal{O}_K$ defined on objects by $\rho(H/J) = K/J$. On morphisms, $\rho(H/J\xrightarrow{f} H/L)$ is defined to be the unique map so that the square
\[
    \begin{tikzcd}[column sep = large]
        K\times_H H/J \ar["K\times_Hf"]{r} \ar[swap, "\cong"]{d} & K\times_H H/L \ar["\cong"]{d} \\
        K/J\ar[swap, "\rho(f)"]{r} & K/L
    \end{tikzcd}
\]
commutes.  The left vertical isomorphism is defined by $[k,hJ]\mapsto (kh)J$, the right isomorphism is similar.  This definition also gives us a functor $\rho^{op}\colon \mathcal{O}_H^{op}\to \mathcal{O}_K^{op}$.  
\begin{definition}\label{definition: induction and restriction}
    For $H\leq K$, the \emph{restriction} functor \[R^K_H\colon \mathrm{Coeff}^K\to \mathrm{Coeff}^H\] is the functor $(\rho^{op})^*$. The \emph{induction} functor \[I^K_H\colon \mathrm{Coeff}^H\to \mathrm{Coeff}^K\] is defined by left Kan extension $\rho^{op}_!$.
\end{definition}

The following lemma, known as \emph{Frobenius reciprocity}, shows that the induction and restriction functors interact well with the box product. 
\begin{lemma}\label{Frobenius reciprocity}
    For any $K$-coefficient system $\underline{M}$ and $H$-coefficient system $\underline{N}$ there is a natural isomorphism
    \[
        \underline{M}\square I^K_H(\underline{N}) \xrightarrow{\sim} I^K_H(R^K_H(\underline{M})\square \underline{N})
    \]
    of $K$-coefficient systems.
\end{lemma}
\begin{proof}
    Since induction, restriction, and the box product all commute with colimits, it suffices to check the case where $\underline{M}=A_X$ and $N=A_Y$ for some $G$-sets $X$ and $Y$.  In this case, the right side is isomorphic to $A_{K\times_H(Res^K_H(X)\times Y)}$, where $Res^K_H(X)$ is $X$, considered as an $H$-set.  The left side is $A_{X\times (K\times_H Y)}$.  It is enough to show that the two representing $K$-sets are isomorphic.  The map
    \[
        X\times (K\times_H Y)\to K\times_H(Res^K_H(X)\times Y)
    \]
    given by $(x,[k,y])\mapsto [k,(k^{-1}x,y)]$ is $K$-equivariant and easily checked to be a bijection, completing the proof.
\end{proof}

Since $R^K_H$ is strong monoidal, it restricts to a functor \[R^K_H\colon \mathrm{Mod}_{R^G_K(\underline{S})}\to \mathrm{Mod}_{R^G_H(\underline{S})}\] for any $G$-coefficient ring $\underline{S}$.  Even though $I^K_H$ is not monoidal, Frobenius reciprocity implies that $I^K_H$ also gives a functor on module categories.  Indeed, if $\underline{M}$ is a module over $R^G_H(\underline{S})$ then we have a composite
\[
    R^G_K(\underline{S})\square I^K_H(\underline{M})\xrightarrow{\sim} I^K_H(R^G_H(\underline{S})\square \underline{M})\to I^K_H(\underline{M})
\]
where the last map is $I^K_H$ applied to the action map of $R^G_H(\underline{S})$ on $\underline{M}$.  This composite defines a $R^G_K(\underline{S})$-module structure on $I^K_H(\underline{M})$.  Since both $I^K_H$ and $R^K_H$ preserve free objects and direct sums, we also get induction and restriction functors on categories of projective objects.  We summarize this discussion in the following proposition.

\begin{proposition}\label{induced functors on projectives}
    The coefficient system induction and restriction functors $I^K_H$ and $R^K_H$ restrict to adjunctions 
    \[
    \begin{tikzcd}[column sep =large]
        \mathrm{Mod}_{R^G_K(\underline{S})} \ar[shift left=2mm, "I^K_H"]{r} & \mathrm{Mod}_{R^G_H(\underline{S})} \ar[shift left=2mm,"R^K_H"]{l}
    \end{tikzcd}
\]
and 
\[
    \begin{tikzcd}[column sep =large]
    \mathrm{Proj}_{R^G_K(\underline{S})} \ar[shift left=2mm, "I^K_H"]{r} & \mathrm{Proj}_{R^G_H(\underline{S})} \ar[shift left=2mm,"R^K_H"]{l}
    \end{tikzcd}
\]
where the top arrow denotes the left adjoint.

\end{proposition}

It is straightforward to see how to compute the restriction functor.  We can also compute the induction functors fairly efficiently.

\begin{lemma}\label{induction on coefficient systems}
    For any $H\leq K\leq G$, any $H$-coefficient system $\underline{M}$, and any $J\leq K$, there is an isomorphism
    \[
        I^K_H(\underline{M})^J\cong \bigoplus\limits_{x\in (K/H)^J} \underline{M}^{J^{x}}
    \]
    where $J^x=x^{-1}Jx$ is the conjugate subgroup.
\end{lemma}
\begin{proof}
    Every coefficient system is the directed colimit of its finitely generated sub-coefficient systems so it suffices to prove the result in the finitely generated case, since induction is a left adjoint and hence commutes with colimits.  Moreover, since every finitely generated coefficient system is a quotient of $A_X$ for some finite $X$, we can reduce further to proving the claim for $A_X$.  Finally, if $X\cong B\amalg C$ then we have $A_X\cong A_{B}\oplus A_C$ and thus we can reduce to the case $\underline{M} = A_{H/L}$ for some $L\leq H$.  

    In this case, the interplay of representable functors and left Kan extensions give an isomorphism 
    \[
        I_H^K(A_{H/L})\cong A_{K/L}
    \]
    and, unwinding the definitions, it remains to establish a natural bijection
    \[
        (K/L)^J\cong \coprod\limits_{x\in (K/H)^J} (H/L)^{J^x}
    \]
    for every $J\leq K$. 

    First note that if $x\in K$ such that $xH\in (K/H)^J$, then for all $j\in J$ we have $x^{-1}jx \in H$ and thus we have a map 
\begin{equation}\label{fixed points of orbits}
    (K/H)^J\to \{ xH\in K/H\mid J^x\subset H\}.
\end{equation}
It is straightforward to check that this map is a bijection.  Similarly, for any $xH\in (K/H)^J$ we have a bijection
\begin{equation}\label{fixed points of orbits 2}
    (H/L)^{J^x} \leftrightarrow \{yL\in H/L\mid J^{xy}\subset L\}
\end{equation}

Let $f\colon K/L\to K/H$ be the canonical quotient which induces a map $F\colon (K/L)^J\to (K/H)^{J}$. Thus we can express
    \[
        (K/L)^J = \coprod\limits_{xH\in (K/H)^{J}} F^{-1}(xH) 
        = \coprod\limits_{xH\in (K/H)^{J}} \{ xy_iL \mid J^{xy_i}\subset L\}
    \]
    where $y_1,\dots, y_n$ are a transversal of the cosets $H/L$.  The claimed bijection follows from this and the bijection \eqref{fixed points of orbits 2}.
\end{proof}
\begin{remark}\label{remark: induction preserves upward vanishing}
    Note that if $M$ is an $H$-coefficient system such that $M(H/L)=0$ for all $L\neq e$ then the lemma implies that $I^K_H(M)(K/P)=0$ for $P\neq e$.  
\end{remark}
\begin{corollary}\label{IndResExact}
    Induction and restriction are both exact functors between abelian categories.
\end{corollary}

We will later need a special case of \cref{induction on coefficient systems} for the fixed point coefficient system $\FP(R)^H = R^H$ where $R$ is a ring with $G$-action. 

\begin{corollary}\label{induction is extension of scalars}
    Let $H\leq K\leq G$ and let $e\leq G$ denote the trivial group.  For any module $\underline{M}$ over $R^G_H(\FP(R))$, there is an isomorphism of $R_{\theta}[K]$-modules \[I^K_H(\underline{M})^e\cong R_{\theta}[K]\otimes_{R_{\theta}[H]}\underline{M}^e.\]
\end{corollary}
\begin{proof}
    Let $\gamma_i,\dots \gamma_n$ denote a collection of representatives for left cosets $K/H$.  By \cref{induction on coefficient systems}
    \[
        I^K_H(\underline{M})^e \cong \bigoplus\limits_{i=1}^n \underline{M}^e
    \]
    where the $K$-action on the right permutes the summands via the action of $H$ on the $\gamma_i$ and then acts on the the copies of $\underline{M}^e$ by the residual elements of $H$ which arise from this action.  
    
    On the other hand, $\gamma_1,\dots, \gamma_n$ form a basis for $R_{\theta}[K]$ as a right $R_{\theta}[H]$-module.  It follows that 
    \[
        R_{\theta}[K]\otimes_{R_{\theta}[H]} N\cong \bigoplus\limits_{i=1}^n N 
    \]
    for any left $R_{\theta}[H]$-module $N$, with $H$-action given exactly as described previously, which implies the result. 
\end{proof}

In addition to induction and restriction, there is a third functor relating categories of coefficient systems.  If $g\in G$ and $H\leq G$ then there are evident isomorphisms of categories $\mathrm{Conj}_H^g\colon \mathcal{O}^{op}_{H}\to \mathcal{O}^{op}_{H^g}$ which send an orbit $H/L$ to $H/L^g$.  We define the \emph{conjugation} functor
\begin{definition}\label{definition: conjugation functors}
    The \emph{conjugation} functor 
    \[
     c_g\colon \mathrm{Coeff}^H\to \mathrm{Coeff}^{H^g}
    \]
    is the restriction along $\mathrm{Conj}_{H^g}^{g^{-1}}
\colon \mathcal{O}^{\op}_{H^g}\to \mathcal{O}^{\op}_{H}$.  
\end{definition}
The following properties are immediately checked.
\begin{lemma}\label{conjugation lemmas}
    For any $H\leq K$ and $g,h\in G$,
    \begin{enumerate}
        \item  there are equalities $c_gc_h = c_{hg}$,
        \item  there are equalities $c_gI^K_H = I^{K^g}_{H^g}c_g$ and $c_gR^K_H = R^{K^g}_{H^g}c_g$,
        \item if $h\in H$ then $c_h\colon \mathrm{Coeff}^H\to\mathrm{Coeff}^H$ is the identity functor.
    \end{enumerate}
\end{lemma}

We stress that the above lemma uses equalities and not natural isomorphisms of functors.

We end this section with a compatibility formula for induction and restriction.  
\begin{proposition}\label{Ind Res double coset}
    For any $J,H\leq K$, there is a natural isomorphism of functors
    \[
        R^K_HI^K_J\cong \bigoplus\limits_{\gamma\in J\backslash K/H} I^H_{H\cap J^{\gamma}}R^{J^{\gamma}}_{H\cap J^{\gamma}} c_{\gamma} 
    \]
    where the $\gamma$ run over a a set of representatives of the double cosets $J\backslash K/H$.
\end{proposition}
\begin{proof}
    All of restriction, induction, and conjugation commute with colimits so it suffices to check this result on the free coefficient systems $A_X$ for a finite $J$-set $X$.  In this case, the three functors can be computed by applying the corresponding functors to the $G$-set $X$ and then taking the corresponding free object.  The result is now immediate from the classical Mackey double coset formula.
\end{proof}

\section{A \textit{K}-theory construction for coefficient systems}\label{sec:KT of CS}

The algebraic $K$-theory of a ring $R$ is classically defined using the category of finitely generated projective modules over $R$. We will similarly define $K$-theory of a coefficient ring $\underline{S}$ using the category $\mathcal{P}^G(\underline{S})$ of finitely-generated projective modules over $\underline{S}$. Indeed, the category $\mathcal{P}^G(\underline{S})$ is a Waldhausen category with cofibrations given by monics (levelwise injections) and weak equivalences given by isomorphisms, so we may apply Waldhausen's $K$-theory construction \cite{waldhausen:1983}.  

One way to define a $K$-theory spectrum for $\underline{S}$ is to simply take the $K$-theory of $\mathcal{P}^G(\underline{S})$. Although this may be an interesting construction, it is not a genuine $G$-spectrum. We use the induction and restriction functors from \cref{subsec:induction restriction} to provide the extra structure needed to produce a genuine $G$-spectrum. 

The model of genuine $G$-spectra we use in this paper are \emph{spectral Mackey functors} which, loosely, can be described as systems of ordinary spectra parametrized the subgroups of $G$ and connected by transfer and restriction maps.  Work of Guillou--May shows that the category of spectral Mackey functors is Quillen equivalent to the category of genuine $G$-spectra \cite{GuillouMay:2011}; this result can be viewed as a stable version of Elmendorf's theorem.  An $\infty$-categorical approach to spectral Mackey functors has been developed and studied by Barwick and co-authors \cite{BarwickOne,Barwick2}.  Explicit comparisons between these $\infty$-categorical models and the classical genuine equivariant stable homotopy category can be found in \cite{Nardin, CMNN}.  

The spectral Mackey functor approach to genuine $G$-spectra is particularly well-suited to constructions involving $K$-theory.  Following the approach of Bohmann--Osorno \cite{BohmannOsorno:2015} and Malkiewich--Merling \cite{malkiewich/merling:2016}, one can construct examples of genuine $G$-spectra by first building a \emph{categorical Mackey functor} $\mathcal{M}$. It will suffice to think of this as a collection of categories $\mathcal{M}^H$ indexed on the subgroups of $G$ and connected by restrictions and transfers.  If the constituent categories $\mathcal{M}^H$ are sufficiently well-structured, i.e.\ they are all Waldhausen or permutative, then one can apply a multiplicative form of $K$-theory (see \cite{GMMO-final}) levelwise to obtain a spectral Mackey functor.  

This framework reduces the production of the genuine $K$-theory spectrum of a coefficient ring into two parts: first, we find input categories for the machine that produce the desired spectrum after applying $K$-theory; second, we need to check that these categories fit together coherently. We will take on the former task in this section and delay the full details of the latter to \cref{sec:tech pf app}.

For a coefficient ring $\underline{S}$, we write $\mathcal{P}^H_{\underline{S}}$ for the category of projective modules over the $H$-coefficient ring $R^G_H(\underline{S})$.  These categories are exact and \cref{induced functors on projectives,IndResExact} tell us that induction and restriction are exact functors 
\[
    \begin{tikzcd}[column sep =large]
        \mathcal P^H_{\underline{S}} \ar[shift left=2mm, "I^K_H"]{r} & \mathcal P^K_{\underline{S}} \ar[shift left=2mm,"R^K_H"]{l}
    \end{tikzcd}
\]
for any $H\leq K$.  Since exact functors extend to give maps on $K$-theory spectra, we obtain maps of spectra 
\[
\begin{tikzcd}[column sep = large]
        K(\mathcal P^H_{\underline{S}}) \ar[shift left=2mm, "K(I^K_H)"]{r} & K(\mathcal P^K_{\underline{S}}) \ar[shift left=2mm,"K(R^K_H)"]{l}
    \end{tikzcd}
\]

This data is the essence of the following theorem.
\begin{theorem}\label{KTheoryExists} Let $\underline{S}$ be a ring $G$-coefficient system. There is a genuine $G$-spectrum $K_G(\underline{S})$, called the \emph{equivariant algebraic $K$-theory} of $\underline{S}$  with \[K_G(\underline{S})^H\simeq K(\mathcal{P}^H_{\underline{S}}).\]  After passing through these equivalences, the transfers, restrictions, and conjugations of $K_G(\underline{S})$ agree with $K(I^H_K)$, $K(R^H_K)$ and $K(c_g)$ respectively.
\end{theorem}

\begin{remark}
    The construction of \cref{KTheoryExists} can be adapted with essentially no changes to produce $K$-theory spectra from Green functors or equivariant ring spectra. Although we do not pursue these constructions in this paper, we believe them to be of independent interest.
\end{remark}

The proof of \cref{KTheoryExists} is rather involved, and requires some careful book-keeping of functors. In the remainder of this section we give the proof with the exception of a few technical details which are deferred to \cref{sec:tech pf app}.

Before proceeding we need some notation. There is a strict $2$-category $\mathbb{B}^G$ whose objects are the subgroups of $G$.  The morphism category $\mathbb{B}^G(H,K)$ is a category whose objects consist, essentially, of formal combinations of transfers, restrictions, and conjugations.  This category is made explicit in \cref{sec:tech pf app}.

\begin{proposition}[{c.f.\ \cite[Proposition 4.6]{malkiewich/merling:2016}}] \label{MMProp4.6}
    Suppose $F\colon \mathbb{B}^G\to {\rm Wald}$ is a strict $2$-functor, where the codomain is the $2$-category of Waldhausen categories, exact functors, and natural transformations.  Additionally, suppose this $2$-functor satisfies the additional condition:
    \begin{itemize}
    \item[($\star$)] for any $A\in F(H)$ and any $S,T\in \mathbb{B}^G(H,K)$ the canonical map \[F(S)(A)\vee F(T)(A)\to F(S\amalg T)(A)\] is an isomorphism and $F(\emptyset)(A) \cong 0_K$ where $\vee$ is the coproduct in $F(K)$ and $0_K\in F(K)$ is the zero object. 
    \end{itemize}
    Then there is a spectral Mackey functor $K_G(F)$ with \[K_G(F)^H \simeq K(F(H))\] and transfers, restrictions, and conjugations given by applying $K$ to the categorical transfers, restrictions, and conjugations induced by $F$. 
\end{proposition}

Since we have not yet defined the category $\mathbb{B}^G$ we cannot include a proof here.  The proof of the proposition is deferred to \cref{sec:tech pf app}.

We would like to say that the assignment $H\mapsto \mathcal P^H_{\underline{S}}$ is the object function of a strict $2$-functor and then apply to \cref{MMProp4.6} to construct a spectral Mackey functor.  Unfortunately, this does not easily assemble into a strict $2$-functor since the composition of induction functors is not strictly associative.  Instead, the induction functors are only associative up to coherent natural isomorphisms, i.e.\ $I^H_KI^K_J\cong I^H_J$ and both of the two ways of reducing the composition of three inductions to a single induction are the same.  Maps between $2$-categories that respect composition up to natural isomorphism are known as \emph{pseudofunctors}, and it turns out this data is sufficient for our purposes.
\begin{proposition}\label{pseudo is enough}
    Suppose $F\colon \mathbb{B}^G\to {\rm Wald}$ is a pseudofunctor which satisfies condition $(\star)$.  Then there is a spectral Mackey functor $K_G(F)$ with homotopy equivalences of spectra \[K_G(F)^H \simeq K(F(H))\] and, after passing through these equivalences, the transfers, restrictions, and conjugations are given by applying $K$ to to the categorical transfers, restrictions, and conjugations induced by $F$. 
\end{proposition}
\begin{proof}
    By forgetting structure, the pseudofunctor $F$ determines a pseudofunctor $F\colon \mathbb{B}^G\to {\rm Cat}$ to the $2$-category of categories, functors, and natural transformations.  It is well known (see \cite{Power:1989}) that all such pseudofunctors are pseudonaturally equivalent to some strict $2$-functor $\widehat{F}\colon \mathbb{B}^G\to {\rm Cat}$.  That is, there are equivalences of categories $\widehat{F}(H)\simeq F(H)$ and these equivalences commute, up to natural isomorphisms, with the functors induced by all $1$-cells in $\mathbb{B}^G$.  We can pullback the Waldhausen structures on $F(H)$ along these equivalences to lift $\widehat{F}$ to a strict $2$ functor into ${\rm Wald}$.  We claim that $\widehat{F}$ continues to satisfy $(\star)$.  Granting this claim, we are done by applying \cref{MMProp4.6}.
    
    To prove the claim, fix an $S,T\in S_{H,K}$ and $A\in \widehat{F}(H)$.  Then we have
    \[
        \widehat{F}(S\amalg T)(A)\cong F(S\amalg T)(A)\xleftarrow{\cong} F(S)(A)\vee F(T)(A)\cong     \widehat{F}(S)(A)\vee \widehat{F}(T)(A)
    \]
    where the middle isomorphism uses the fact that $F$ satisfies $(\star)$.  The key here is that being a zero object or a coproduct is not a Waldhausen property and so strictification procedure will not affect this.  That $\widehat{F}(\emptyset)(A)\cong 0_K$ is similar.
\end{proof}

The equivariant $A$-theory of \cite{malkiewich/merling:2016} is recovered from \cref{pseudo is enough}.

\begin{example} \label{A theory as a waldhausen Mackey}
    Let $F\colon\mathbb{B}^G\to \mathrm{Wald}$ be the $2$-functor which, on objects, sends $H$ to the Waldhausen category $R^H_{\fd}(X)$ of finitely dominated $H$-retractive spaces over a $G$-space $X$.  Then the resulting spectral Mackey functor is the equivariant $A$-theory $A_G(X)$ from \cite{malkiewich/merling:2016}, discussed in more detail in \cref{sec:bckgd on A theory}. Indeed, tracing through the arguments in \cite[Proposition 4.11]{malkiewich/merling:2016}, one sees that Malkiewich--Merling first construct pseudofunctorial transfers and restriction and then perform the standard strictification procedure.
\end{example}

The next proposition, together with \cref{pseudo is enough}, proves \cref{KTheoryExists}.

\begin{proposition}\label{spectral Mackey functor}
    For every coefficient ring $\underline{S}$ the assignment  $H\mapsto \mathcal{P}^H_{\underline{S}}$, together with the induction, restriction, and conjugation functors assembles into a pseudofunctor $\mathbb{B}^G\to {\rm Wald}$ satisfying condition $(\star)$.
\end{proposition}

The fact that induction, restriction, and conjugation are compatible in a pseudofunctorial way is the data of \cref{Ind Res double coset} and \cref{conjugation lemmas}.  We leave the full details of the proof of \cref{spectral Mackey functor} to \cref{corollary: proof of proposition from main body} as they are quite technical and not essential to understanding the rest of the paper.

\subsection{A splitting on $K$-theory}\label{subsec:KT split}

In this subsection, we will prove the following theorem which is analogous to the $A$-theory splitting from \cite{badzioch/dorabiala:2016} and a splitting on $K$-theory from \cite{Luck}.

\begin{theorem}\label{K theory splitting}
    For any coefficient ring $\underline{S}$ there is a splitting of $K$-theory spectra
    \[
       K_G(\underline{S})^G\simeq K(\mathrm{Proj}_{\underline{S}})\simeq \prod\limits_{(H)\leq G}  K(\underline{S}_{\theta}^H)
    \]
    where the product is taken over conjugacy classes of subgroups of $G$ and $\underline{S}_{\theta}^H$ is the twisted group ring.
\end{theorem}

Before proving this theorem, we highlight some important examples.

\begin{example}\label{example: K theory of constant Z}
    If $\underline{S} = A_{G/G} = \underline{\ZZ}$ is the constant $\ZZ$ coefficient ring, then $\underline{S}^H = \ZZ$ has trivial $WH$ action for every $H$.  The splitting becomes
    \[
        K_G(A_{G/G})^G \simeq \prod\limits_{(H)\leq G} K(\ZZ [WH]).
    \]
    Since $A_{G/G}$ is the unit for coefficient rings, this is the $K$-theory of finitely generated projective coefficient systems.  This result can be seen as an analog of a splitting in the $K$-theory of Mackey functors from certain groups which is due to Greenlees \cite{greenlees:1992}.

\end{example}

\begin{example} \label{homotopy orbits twisted group ring}
    Let $X$ be a based $G$-space with $X^H$ connected for all $H$.  Define $\underline{S}=\underline{\ZZ[\pi_1X]}$ by
    \[
        \underline{\ZZ[\pi_1X]}^H = \ZZ[\pi_1(X^H)]
    \]
     with restrictions induced by the inclusions of fixed points $X^H\to X^K$ for $K\leq H$. The $WH$-actions are induced by the actions of $WH$ on $X^H$.  In this setting, the splitting becomes
     \[
        K_G(\underline{\ZZ[\pi_1X]})^G \simeq \prod\limits_{(H)\leq G} K(\ZZ[\pi_1(X^H)]_{\theta}[WH])
     \]
     This splitting will play a key role in our description of linearization in \cref{sec:Equivariant Linearization}.
\end{example}
\begin{example}
     Let us give a concrete example of $K_G(\underline{\ZZ[\pi_1X]})^G$. Let $X=S^2$ with action by the group $C_2$ by reflection over the plane containing the equator.  Then $X^{C_2} = S^1$ which gives 
     \[
        K_G(\underline{\ZZ[\pi_1X]})^G\simeq K(\ZZ[t^{\pm}])\times K(\ZZ[C_2])
     \]
     where the first term comes from identifying the group ring $\ZZ[\ZZ]$ with the ring of Laurent polynomials.  The $K$-theory of the ring of Laurent polynomials admits another splitting by the fundamental theorem of $K$-theory so we have
     \[
         K_G(\underline{\ZZ[\pi_1X]})^G\simeq K(\ZZ)\times \Sigma K(\ZZ)\times K(\ZZ[C_2]).
     \]
\end{example}

The remainder of this subsection is devoted to the proof of \cref{K theory splitting}, using a similar method as in \cite{badzioch/dorabiala:2016}.
Let $\{(e)=(H_0),(H_1),\dots, (H_n)=(G)\}$ be the set of conjugacy classes of subgroups of $G$ where $(H_j)\leq (H_i)$ implies $j\leq i$. Here $(H_j)\leq (H_i)$ means $H_j$ is subconjugate to $H_i$ in $G$ (with $(H_j)<(H_i)$ if $H_j$ is subconjugate but not conjugate to $H_i$).

Let $\underline{S}$ be a coefficient ring and let $\mathcal{P}$ be the category of finitely generated projective $\underline{S}$ modules. For $1\leq i\leq n$, let $\mathcal{P}_{\leq i}$ be the subcategory of finitely generated projective modules which vanish on $G/H_j$ for all $j>i$.  We have a sequence of inclusions of subcategories
\[
    0 = \mathcal{P}_{\leq 0} \to \dots \to \mathcal{P}_{\leq n} = \mathcal{P}
\]
and we note that each subcategory is closed under admissible extensions, admissible quotients, and admissible subobjects. If these categories were all abelian, these inclusions would induce fiber sequences on $K$-theory 
\[
    K(\mathcal{P}_{\leq (i-1)})\to K(\mathcal{P}_{\leq i}) \to K(\mathcal{P}_{\leq i}/\mathcal{P}_{\leq(i-1)})
\]
where $\mathcal{P}_{\leq i}/\mathcal{P}_{\leq(i-1)}$ is a quotient category. However, since the categories of projective modules are only guaranteed to be exact, we will need to use slightly different methods. In particular, we need to replace the notion of a quotient category with a new category we call $\mathcal{P}_{i}$.
\begin{definition}
    A $\underline{S}$-module $\underline{M}$ is \textit{$H_i$-generated} if the smallest submodule of $\underline{M}$ containing all of $\underline{M}^{H_i}$ is all of $\underline{M}$.  We denote the category of $H_i$-generated projective modules by $\mathcal{P}_i$.
\end{definition}

\begin{proposition}\label{projective splitting sequence}
    For every $P\in \mathcal{P}_{\leq i}$, there exists a split exact sequence of $\underline{S}$-modules
    \[
        P_i\to P\to P/P_i
    \]
    where $P_i$ is an $H_i$-generated object in $\mathcal{P}_{\leq i}$ and $P/P_i$ is in $\mathcal{P}_{\leq (i-1)}$.  Moreover, this sequence is unique up to isomorphism.
\end{proposition}
\begin{proof}
    Since $P/P_i$ must vanish on $G/H_i$ and $P_i$ must be $H_i$-generated, the only choice is to take $P_i\subset P$ to be the smallest submodule which contains all of $P^{H_i}$.  It follows that there is an exact sequence \[
        P_i\to P\to P/P_i
    \] and that this sequence is unique up to isomorphism with respect to the stated properties.  It remains to show this sequence splits.

    Pick a finite $G$-set $X$ and a surjection $r\colon \underline{S}_X\twoheadrightarrow P$.  Since $P$ vanishes on $G/H_j$ for all $j>i$, we may assume that the isotropy of $X$ is bounded above by $H_i$.  That is, $X$ is isomorphic to a finite union of transitive $G$-sets of the form $G/H_j$ for $j\leq i$.  Write $X = X_i\amalg X'$, where $X_i$ is the subset of elements with stabilizer conjugate to $H_i$ and $X'$ has isotropy strictly bounded by $i$.  This gives a splitting $\underline{S}_{X}\cong \underline{S}_{X_i}\oplus \underline{S}_{X'}$.  Since $\underline{S}_{X'}(G/H_i) =0$, it follows that $r$ restricted to $\underline{S}_{X_i}$ surjects onto $P(G/H_i)$.  The Yoneda lemma implies that maps out of $\underline{S}_{X_i}$ are determined entirely by where they send certain elements at level $G/H_i$, and it follows that the image $r(\underline{S}_{X_i})$ is exactly $P_i$.  
    
    Consider the following commutative square:
    \[
        \begin{tikzcd}
            \underline{S}_{X} \ar[two heads, "r"]{r} \ar[two heads]{d} & P\ar[two heads]{d}\\
            \underline{S}_{X}/\underline{S}_{X_i} \ar[two heads, "\overline{r}"]{r} & P/P_i
        \end{tikzcd}
    \]
    where the vertical maps are quotients and the map $\overline{r}$ is the induced map on the quotient, which exists because $r(\underline{S}_{X_i}) = P_i$.  Since $P$ is projective, there is a map $s\colon P\to \underline{S}_X$ such that $r\circ s = {\rm id}_{P}$. Since morphisms out of $P_i$ are determined by where they send elements in $P_i^{H_i}$, and $\underline{S}_{X'}^{H_i}=0$, we see that $s(P_i)\subset \underline{S}_{X_i}$.  It follows that there is an induced map $\overline{s}\colon P/P_i\to \underline{S}_X/\underline{S}_{X_i}$ which splits $\overline{r}$ and so $P/P_i$ is a direct summand of $\underline{S}_X/\underline{S}_{X_i}\cong \underline{S}_{X'}$ and is therefore projective.  The projection $P\twoheadrightarrow P/P_i$ then splits and we are done.
\end{proof}

\begin{corollary}    For any coefficient ring $\underline{S}$ there is a splitting of $K$-theory spectra
    \[
        K(\mathcal{P})\simeq \prod\limits_{i=1}^n  K(\mathcal{P}_i)
    \]
\end{corollary}
\begin{proof}
    Let $\mathcal{E}_i$ be the Waldhausen category whose objects are split exact sequences
    \[
        Q\to P\to P/Q
    \]
    where $P$ is in $\mathcal{P}_{\leq i}$, $Q$ is $H_i$-generated and $P/Q$ is in $\mathcal{P}_{\leq (i-1)}$. The Waldhausen structure is that of the extension category (see \cite[\S II.9.3]{Kbook}). By the Waldhausen additivity theorem \cite[Proposition 1.3.2(1)]{waldhausen:1983}, the functor 
    \[
        \mathcal{E}_i\to \mathcal{P}_i\times \mathcal{P}_{\leq (i-1)}
    \]
    which projects onto the left and right terms of the exact sequences induces an equivalence on $K$-theory.  We will show the functor $\mathcal{P}_{\leq i}\to \mathcal{E}_i$ which sends $P$ to the sequence 
    \[
        \Psi(P) = [P_i\to P\to P/P_i]
    \]
    of \cref{projective splitting sequence} is an exact equivalence.  With this in hand, our splitting follows from induction since we have
    \[
        K(\mathcal{P}_{\leq i})\simeq K(\mathcal{E}_i)\simeq K(\mathcal{P}_i)\times K(\mathcal{P}_{\leq(i-1)}).
    \]

    Since everything in sight is projective, we see that every object in $\mathcal{E}_i$ is equivalent to one of the form 
    \[
        Q\to Q\oplus Q'\to Q'
    \]
    where the maps are the projection and inclusions.  This sequence is evidently isomorphic to $\Psi(Q\oplus Q')$ and so $\Psi$ is essentially surjective.  Moreover, every map between sequences is entirely determined by what happens in the middle of the sequence and so fullness, faithfulness, and exactness of $\Psi$ all follow.
\end{proof}

To finish the proof of \cref{K theory splitting}, we need to understand the categories $\mathcal{P}_i$ in terms of more recognizable algebraic objects. We will show that these categories are equivalent to the categories of projective modules over the twisted group rings  $\underline{S}^{H_i}_{\theta}$.

Evaluation at $G/H_i$ gives us a functor from $\underline{S}$-modules to $\underline{S}^{H_i}_{\theta}$-modules and our first order of business is to show that evaluation sends objects in $\mathcal{P}_i$ to projective modules. 

\begin{lemma}\label{eval is projective}
    If $P$ is in $\mathcal{P}_{\geq i}$ then $P^{H_i}$ is a finitely generated projective $\underline{S}^{H_i}_{\theta}$-module.
\end{lemma}
\begin{proof}
    By \cref{projective splitting sequence} there is a short exact sequence
    \[
        P_i\to P\to P/P_i
    \]
    where $P_i\in \mathcal{P}_i$ and $P/P_i\in \mathcal{P}_{\geq (i-1)}$.  It follows that $P_i^{H_i} = P^{H_i}$ and so we can reduce to the case that $P\in \mathcal{P}_i$ is $H_i$-generated.

    Since $P^{H_i}$ is finitely generated, there is a surjection $f\colon (\underline{S}^{H_i}_{\theta})^n\to P^{H_i}$ for some $n$.  By the Yoneda lemma, this corresponds to a unique morphism $F\colon \underline{S}_{n(G/H_i)}\to P$ defined by $F^{H_i}=f$.  This map is surjective at level $G/H_i$, and is therefore surjective since $P$ is $H_i$-generated.  Since $P$ is projective the map $F$ admits a section, which gives a section of $f$ at level $G/H_i$. Thus $P^{H_i}$ is a direct summand of $(\underline{S}^{H_i}_{\theta})^n$ and is therefore projective.
\end{proof}

We now show that the exact functor 
\[
ev_{G/H_i}\colon \mathcal{P}_{i}\to {\rm Proj}_{\underline{S}^{H_i}_{\theta}}
\]
is an equivalence. This map being full and faithful is the assertion that for any $P$ and $T$ in $\mathcal{P}_{i}$ the functor $ev_{G/H_i}$ induces isomorphisms
\[
    \mathcal{P}_i(P,T)\cong {\rm Proj}_{\underline{S}^{H_i}_{\theta}}(P^{H_i},T^{H_i})
\]
and in fact we can prove a stronger result.  

\begin{lemma}\label{Phi functors}
    Let $\underline{S}$ be a coefficient ring. For every projective $\underline{S}^{H_i}_{\theta}$-module $Q$ there is an object $\Phi_i(Q)$ in $\mathcal{P}_i$ such that there are isomorphisms of abelian groups
    \[
        {\rm Mod}_{\underline{S}}(\Phi_i(Q),M) \cong {\rm Mod}_{\underline{S}^{H_i}_{\theta}}(Q,M(G/H_i))
    \]
    for any $\underline{S}$-module $M$.
\end{lemma}
\begin{proof}
    Pick a projective $S_{H_i}$-module $N$ such that $(S_{H_i})^n\cong Q\oplus N$ for some $n$.  For any $\underline{S}$-module $M$ there are natural isomorphisms of abelian groups
    
    \begin{align}
        {\rm Mod}_{\underline{S}}(A_{n(G/H_i)},M) & \cong {\rm Mod}_{\underline{S}^{H_i}_{\theta}}((\underline{S}^{H_i}_{\theta})^n, M^{H_i})\nonumber\\
         & \cong {\rm Mod}_{\underline{S}^{H_i}_{\theta}}(Q, M^{H_i})\oplus {\rm Mod}_{\underline{S}^{H_i}_{\theta}}(N, M^{H_i}) \label{free mapping splitting}
    \end{align}

    Picking $M=A_{n(G/H_i)}$, the identity on $M$ corresponds to the sum of a maps $(i_Q,0)$ and $(0,i_N)$ where $i_Q$ is the inclusion of $Q$ into $(\underline{S}^{H_i}_{\theta})^n$ along the isomorphism $(S_{H_i})^n\cong Q\oplus N$.  The map $(i_Q,0)$ is an idempotent endomorphism of $A_{n(G/H_i)}$ which produces a splitting 
    \[
        A_{n(G/H_i)}\cong \Phi_i(Q)\oplus \Phi_i(N)
    \]
    where $\Phi_i(Q) = \rm{Im}(\rm{id_Q},0)$.  That $\Phi_i(Q)$ satisfies the universal property of the statement follows from \cref{free mapping splitting}.  To see that $\Phi_i(Q)$ is $H_i$-generated, note that it is the homomorphic image of $ A_{n(G/H_i)}$ which is also $H_i$-generated. To see that $\Phi_i(Q)^{H_i}=Q$, note that at level $G/H_i$ the map $(\rm{id_Q},0)$ is simply given by the projection of $(\underline{S}^{H_i}_{\theta})^n$ onto $Q$.
\end{proof}

\begin{remark}
    For any $P$ in $\mathcal{P}_{\leq i}$, note that the proof above shows that $P_i\cong \Phi_i(P^{H_i})$. Moreover, for every projective $\underline{S}^{H_i}_{\theta}$-module $Q$, we have $\Phi_i(Q)^{H_i}=Q$ which incidentally implies that $ev_{G/H_i}$ is essentially surjective.
\end{remark}

We want to show that the assignment $Q\mapsto \Phi_i(Q)$ is actually a functor.  If we can prove this fact then the lemma tells us that $\Phi_i$ is an inverse to $ev_{G/H_i}$.
To this end, note that for any fixed $Q$ the coefficient system $\Phi_i(Q)$ represents the functor $R^Q\colon Mod_{\underline{S}}\to {\rm Ab}$ given by
    \[
        R^Q(M) = {\rm Mod}_{S_{H_i}}(Q,M^{H_i}).
    \]
These assemble into a functor 
\[
        R^{(-)}\colon {\rm Proj}_{\underline{S}^{H_i}_{\theta}}^{op}\to  {\rm Fun}(Mod_{\underline{S}}, {\rm Ab}),
\] whose image lands in the subcategory of represented functors.  By the enriched Yoneda lemma, there must be a functor
\[
    \Phi_i^{op}\colon \rm{Proj}^{op}_{S_{H_i}}\to \rm{Fun}(\rm{Mod}_{\underline{S}},\rm{Ab})
\]
so that  $R^{(-)} \cong Y\circ \Phi_i^{op}$, where $Y$ is the Yoneda embedding.  Then the functor $\Phi_i = (\Phi_i^{op})^{op}$ is a quasi-inverse to evaluation at $G/H_i$, proving our claim that $\mathcal{P}_i$ is categorically equivalent to $\mathrm{Proj}_{\underline{S}^{H_i}_{\theta}}$.

\begin{corollary}\label{cor:Kfib seq in mod}
    There is a split fiber sequence of spectra \[
    K(\mathcal P_{\leq i-1}) \to K(\mathcal{P}_{\leq i}) \xrightarrow{K(ev_{G/H_i})} K(\underline{S}_{\theta}^{H_i}).
    \]
\end{corollary}\begin{proof}
    As shown in the proof of \cref{K theory splitting}, there is a split fiber sequence\[
    K(\mathcal P_{\leq i-1})\to K(\mathcal P_{\leq i}) \to K(\mathcal P_i).
    \] \cref{Phi functors} then shows that $ev_{G/H_i}\colon \mathcal P_i\to {\rm Proj}_{\underline{S}_\theta^{H_i}}$ induces an equivalence after $K$-theory which is split by $\Phi_i$. 
\end{proof}

\begin{remark}
    We could also consider the $K$-theory of the category $\mathcal M$ of finitely generated $\underline{S}$-modules instead of $\mathcal P$. Defining $\mathcal M_{\leq i}$ analogous to $\mathcal{P}_{\leq i}$, one can show that the quotient category $\mathcal M_{\leq i}/\mathcal M_{\leq (i-1)}$ is equivalent to the category $\mathcal M_i$ of finitely generated modules over $\underline{S}^{H_i}_{\theta}$. In this case, one can then show there is a split fiber sequence\[
K(\mathcal{M}_{\leq (i-1)})\to K(\mathcal{M}_{\leq i}) \to K(\mathcal M_{\leq i}/\mathcal M_{\leq (i-1)})
\] which induces a splitting of spectra
    \[
        K(\mathcal{M})\simeq \prod\limits_{i=1}^n K(\mathcal M_{i}).
    \] The spectrum $K(\mathcal M)$ is analogous to the $G$-theory of a ring.
\end{remark}

\subsection{Perfect complexes}\label{subsec:perf}
In this subsection, we give an alternate construction of the $K$-theory of a coefficient ring using perfect complexes.  This approach has several advantages coming from the fact that the Waldhausen structure on perfect complexes admits a cylinder functor.  For our purposes, this construction is a more natural home for the categorical version of the linearization map considered in the next section.

\begin{definition}
    A chain complex over a ring is \emph{perfect} if it is quasi-isomorphic to a bounded complex of finitely generated projective modules.  The category $\rm{Perf}_R$ of perfect chain complexes over $R$ forms a Waldhausen category whose cofibrations are degree-wise monics and whose weak equivalences are quasi-isomorphisms.  
\end{definition}
Since the restriction and induction functors on $\underline{S}$-modules are an exact adjunction which preserve projective objects we obtain an adjunction on categories of perfect complexes. 
\begin{lemma}\label{lemma: adjunction on perfect complexes}
    The restriction and induction of coefficient systems induces an adjunction
    \[
    \begin{tikzcd}[column sep =large]
    \mathrm{Perf}_{R^G_K(\underline{S})} \ar[shift left=2mm, "I^K_H"]{r} & \mathrm{Perf}_{R^G_H(\underline{S})} \ar[shift left=2mm,"R^K_H"]{l}
    \end{tikzcd}
\]
for any coefficient ring $\underline{S}$.
\end{lemma}

The Gillet-Waldhausen theorem (see \cite[1.11.7]{thomason/trobaugh:1990}) says that the $K$-theory of $\rm{Perf}_R$ is equivalent to the $K$-theory of $\rm{Proj}_R$ and thus the two constructions give rise to the same $K$-groups.  The same proof gives us the analogous result in this setting.

\begin{proposition}
    Let $\underline{S}$ be a coefficient ring.  The functor $\rm{Proj}_{\underline{S}}\to \rm{Perf}_{\underline{S}}$ which sends a finitely generated projective module $P$ to the chain complex which is $P$ in degree zero and 0 elsewhere induces a homotopy equivalence on $K$-theory.
\end{proposition}

Recall that for $1\leq i\leq n$, we defined $\mathcal{P}_{\leq i}$ to be the subcategory of finitely generated projective $\underline{S}$-modules which vanish on $G/H_j$ for all $j>i$.  We will call a perfect complex of $\underline{S}$-modules \emph{$H_i$-bounded} if it is quasi-isomorphic to a bounded complex consisting of modules in $\mathcal{P}_{\leq i}$; that is, if all the constituent modules are $H_i$-bounded and projective.  We denote the category of $H_i$-bounded perfect complexes by $\rm{Perf}_{\leq i}$.  
\begin{lemma}
    If $M_{*}$ an $H_i$-bounded perfect complex of $\underline{S}$-modules then $M_*(G/H_i)$ is a perfect complex of $\underline{S}_{\theta}^{H_i}$-modules.
\end{lemma}
\begin{proof}
    Pick an $H_i$-bounded projective complex $P_*$ and a quasi-isomorphism $P_*\to M_*$.  By \cref{eval is projective}, $P_*^{H_i}$ is a bounded complex of projective $\underline{S}_{\theta}^{H_i}$-modules.  Since evaluation at any $G/H_i$ is exact it preserves quasi-isomorphisms and the result follows.
\end{proof}

The lemma gives us well defined functors
\[
    ev_{G/H_i}\colon \mathrm{Perf}_{\leq i}\to \mathrm{Perf}_{\underline{S}_{\theta}^{H_i}}
\]
for all $i$. Although is not clear how to define $\Phi_i$ on perfect complexes, the splitting exists up to homotopy after applying $K$-theory.

\begin{proposition}
    For any $1\leq i\leq n$ there is a homotopy equivalence of spectra
    \[
        K(\mathrm{Perf}_{\leq i})\simeq K(\mathrm{Perf}_{\leq (i-1)})\times K(\Perf_{\underline{S}^{H_i}_{\theta}}).
    \]
    where the projection to the right summand is induced by evaluation at $G/H_i$.
\end{proposition}
\begin{proof}
Observe that there is a commutative diagram\[
    \begin{tikzcd}
        \mathcal P_{\leq i-1} \ar[r] \ar[d] & \mathcal{P}_{\leq i} \ar[r, "ev_{G/H_i}"] \ar[d] & {\rm Proj}_{\underline{S}_\theta^{H_i}} \ar[d] \\
        \Perf_{\leq i-1} \ar[r] & \Perf_{\leq i} \ar[r, swap, "ev_{G/H_i}"] & \Perf_{\underline{S}_\theta^{H_i}}
    \end{tikzcd}
    \] where where the vertical functors send a module $M$ to the chain complex which is $M$ concentrated in degree zero. In particular, by the Gillet--Waldhausen Theorem, all the vertical maps induce equivalences after $K$-theory; note that $\mathcal P_{\leq i}$ is closed under extensions as being in  this category is a levelwise vanishing condition. Since the top row becomes a split fiber sequence after $K$-theory by \cref{cor:Kfib seq in mod}, so does the bottom row. 
\end{proof}

Crucially, this result allows us to interpret the projection maps in the splitting as coming from evaluation of perfect $\underline{S}$-complexes instead of just evaluation of projective modules.  This observation becomes important in our analysis of the linearization map after the splitting.

\begin{remark}
    There is an alternative proof of this result that uses the Waldhausen Localization Theorem \cite[Theorem V.2.1]{Kbook} on $\Perf_{\leq i}$ with $v$ being the usual quasi-isomorphisms and $w$ being those morphisms which are quasi-isomorphisms after evaluating at $G/H_i$. One then identifies the subcategory of $w$-acyclic objects of $\mathcal P_{\leq i}$ with  $\mathcal P_{\leq i-1}$ and $K(\Perf_{\leq i}, w)\simeq K(\Perf_{\underline{S}_\theta^{H_i}})\simeq K({\underline{S}_\theta^{H_i}})$. 
    As a corollary of this approach, we obtain a different proof of the splitting of \cref{K theory splitting} via the commutative diagram in the proof above. 
\end{remark}

\subsection{A comparison with the $K$-theory of $G$-rings}\label{subsec:Mona comparison}
    
Throughout this subsection, fix a ring $R$ with an action by a finite group $G$ through ring automorphisms. Merling \cite{merling:2017} constructs a genuine $G$-spectrum $K_{\theta}(R)$ associated to any ring with $G$-action. We will show how the $K$-theory of coefficient systems relates to Merling's construction when $\abs{G}$ is invertible in $R$. For example, $R$ may be a number field with Galois group $G$.

With the invertibility hypotheses on $|G|$, it was shown in \cite[4.0.1]{Brazelton} that Merling's $K$-theory can be described by
    \[
       K_{\theta}(R)^H \simeq K(R_{\theta}[H])
    \]
with the transfers given by extending scalars along the inclusions $R_{\theta}[H]\to R_{\theta}[K]$ and restrictions induced by restriction of scalars.  
    
    Let $\FP(R)$ be the fixed point coefficient system from \cref{ex:FP(R)} defined by $\FP(R)^H = R^H$. Using \cref{K theory splitting}, we have
    \[
        K_G(\FP(R))^H \simeq \prod\limits_{(I)\leq H} K(R^I_{\theta}[W_HI]).
    \]
    Note that taking $I=e$ we always have a summand given by $K(R_{\theta}[H])\simeq K_{\theta}(R)^H$.  We now show that this summand arises from a splitting on the spectrum level, i.e. $K_{\theta}(R)$ is a direct summand of $K_G(R)$.
    
    \begin{proposition}
        Let $R$ be a $G$-ring with $|G|$ invertible in $R$.  There is a map of $G$-spectra, in the equivariant stable homotopy category, $K_{\theta}(R)\to K_G(R)$ such that the induced map $\pi^H_n(K_{\theta}(R))\to \pi^H_{n}(K_G(R))$ is the inclusion of a direct summand for all $H\leq G$ and all $n$.
    \end{proposition}   
    \begin{proof}
        For any $H\leq G$, we have
        \[
            K_G(R)^H\simeq K(\mathrm{Proj}(\FP_H(R)))
        \]
        where $\FP_H(R)$ is the restriction of $\FP(R)$ to an $H$-coefficient system. The desired map of $G$-spectra is induced by the functors $
            \Phi_H\colon \mathrm{Mod}(R_{\theta}[H])\to \mathrm{Mod}(\FP_H(R))$ given by \[
           \Phi(N)\colon H/K \mapsto \left\{\begin{array}{cc}
               N & K=e, \\
               0 & K\neq e.
           \end{array}\right.
            \] We first show there is a functor        \[
            F\colon \mathrm{Mod}(\FP_H(R))\to \mathrm{Mod}(R_{\theta}[H])
        \] which splits $\Phi$ and we then show that $\Phi$ indeed induces a map of $G$-spectra.
            
Define the functor $F$ by
        \[
            F(M) = M(H/e)\left/\sum\limits_{e<K\leq H} R^K_e(M(H/K))\right..
        \]
        Note that $F(M)$ is a module over $R_{\theta}[H]$ because we quotiented by an $H$-fixed sub $R$-module of $M(H/e)$. There is an adjunction $F\dashv \Phi$ and because $\Phi$ is evidently exact we may conclude that $F$ preserves projectives. The adjunction further shows that $\Phi$ agrees with the functor $\Phi_1$ of \cref{Phi functors} on projective modules and so $\Phi$ also preserves projective objects.  Since $F$ splits $\Phi$, the induced maps $\pi^H_n(K_{\theta}(R))\to \pi^H_{n}(K_G(R))$ are all inclusions of direct summands. 
        
        It remains to show that the functors $\Phi_H$ induce a map of $G$-spectra. To do so, we apply a technical result, \cref{Theorem: construction examples and maps}(3) from \cref{subsec: appendix constructing examples}.  In the notation of \cref{Theorem: construction examples and maps}, we have $R(G/H) = \Proj(\FP_H(R))$, $P(G/H) = \Proj(R_{\theta}[H])$, and $L=\Phi$. Note that the $\Phi_H$ commutes on the nose with the restriction and conjugation functors so $\Phi\colon R\Rightarrow P$ is indeed a pseudonatural transformation.  Thus it suffices to check that for any canonical quotient $q\colon G/H\to G/K$ that the ``mate diagram''
\[\begin{tikzcd}
	{\Proj(\FP_H(R))} & {\Proj(\FP_H(R))} && {\Proj(R_{\theta}[H])} & {\Proj(R_{\theta}[H])} \\
	{\Proj(\FP_K(R))} & {\Proj(\FP_K(R))} && {\Proj(R_{\theta}[K])} & {\Proj(R_{\theta}[K])}
	\arrow["{I^K_H}"', from=1-1, to=2-1]
	\arrow[""{name=0, anchor=center, inner sep=0}, equals, from=1-2, to=1-1]
	\arrow["{\Phi_H}"', from=1-4, to=1-2]
	\arrow[""{name=1, anchor=center, inner sep=0}, equals, from=1-5, to=1-4]
	\arrow["{R_{\theta}[K]\otimes_{-}}", from=1-5, to=2-5]
	\arrow["{R^K_H}"', from=2-2, to=1-2]
	\arrow[""{name=2, anchor=center, inner sep=0}, equals, from=2-2, to=2-1]
	\arrow["{\mathrm{res}}", from=2-4, to=1-4]
	\arrow["{\Phi_K}", from=2-4, to=2-2]
	\arrow[""{name=3, anchor=center, inner sep=0}, equals, from=2-5, to=2-4]
	\arrow["\epsilon"', shorten <=6pt, shorten >=9pt, Rightarrow, from=0, to=2]
	\arrow["\eta", shorten <=6pt, shorten >=9pt, Rightarrow, from=1, to=3]
\end{tikzcd}\]
is inhabited by an invertible $2$-cell.  Here, $\epsilon$ and $\eta$ are the counit and unit of the two adjunctions.  If $M$ is any $R_{\theta}[H]$-module, then $I^K_H(\Phi_H(M))$ vanishes on all orbits $K/P$ for  $P\neq e$ by \cref{remark: induction preserves upward vanishing}.  Since $\Phi_K(R_{\theta}[K]\otimes_{R_{\theta}[H]}M)$ also has this property, by definition, it suffices to compute what happens at the $K/e$-level.  Using \cref{induction is extension of scalars}, we see that the map of interest is equivalent to
\[
    R_{\theta}[K]\otimes_{R_{\theta}[H]}M\xrightarrow{R_{\theta}[K]\otimes_{R_{\theta}[H]}\eta} R_{\theta}[K]\otimes_{R_{\theta}[H]}R_{\theta}[K]\otimes_{R_{\theta}[H]}M\xrightarrow{\mu\otimes_{R_{\theta}[H]}M} R_{\theta}[K]\otimes_{R_{\theta}[H]}M
\]
which is the identity by the triangle identities for extension and restriction of scalars.
    \end{proof} 

        

    \subsection{Reduced $K$-theory}\label{subsec: reduced K theory}
    In this subsection, we construct the \textit{reduced} $K$-theory of a coefficient ring. 
    Classically, the reduced $K$-theory is defined as the cofiber \[
    \widetilde{K}(R) := {\rm cofib}(K(\ZZ)\to K(R))
    \] of the map induced by the unique ring map $\ZZ\to R$. In $\widetilde{K}_0(R)$, this has the effect of taking the quotient of $K_0(R)$ by all the free modules; the class of a projective module $P$ is zero in $\widetilde{K}_0(R)$ if and only if $P$ is stably free. 

    We would like to proceed analogously in the equivariant setting. However, unlike in the non-equivariant setting, the initial coefficient ring $\underline{\mathbb{Z}}$ (the constant coefficient ring on $\ZZ$) has projective modules which are not free as $\underline{\ZZ}$-modules. Consequently take the cofiber of the map $K_G(\underline{\ZZ})\to K_G(\underline{S})$ induced by the unique coefficient ring map $\underline{\ZZ}\to \underline{S}$ does not have the desired effect of killing free modules. 
    
    Instead, we replace $\underline{\mathbb{Z}}$ with the subcategory of free $\underline{S}$-modules.  For any coefficient ring $\underline{S}$ let $\mathrm{Free}_{\underline{S}}$ denote the full subcategory of $\mathrm{Mod}_{\underline{S}}$ containing the finitely generated free modules.  Since $\mathrm{Free}_{R^G_H(\underline{S})}\subset \mathrm{Proj}_{R^G_H(\underline{S})}$, the assignment
    \[
        G/H\mapsto \mathrm{Free}_{R^G_H(\underline{S})}
    \]
is the object function of a pseudofunctor $\mathbb{B}^G\to \mathrm{Wald}$ and applying \cref{pseudo is enough} we obtain a genuine $G$-spectrum $K^{\rm free}_G(\underline{S})$. The inclusion of free modules into projective modules always induces a map of spectra $K^{\rm free}_G(\underline{S})\to K_G(\underline{S})$.  

\begin{definition}
    Let $\underline{S}$ be a coefficient ring.  The \emph{reduced algebraic $K$-theory} of $\underline{S}$, denoted $\widetilde{K}_G(\underline{S})$, is the cofiber of the map $K_G^{\mathrm{free}}(\underline{S})\to K_G(\underline{S})$.
\end{definition}

We detail what happens on $\pi_0$.  Since the free modules over an $H$-Mackey functor are in bijection with finite $H$-sets, it is straightforward to check that $\pi_0^{H}(K^{\mathrm{free}}_G(\underline{S}))$ is isomorphic to the Burnside ring $\Omega(H)$ of finite $H$-sets with addition given by disjoint union and multiplication given by cartesian product.  Thus for every subgroup $H\leq G$ there is a short exact sequence
\[
   0\to  \Omega(H)\to \pi_0^H(K_G(\underline{S}))\to \widetilde{K}_G(\underline{S})\to 0.
\]
In specific examples we can identify the copy of $\Omega(H)$ more explicitly.
\begin{example}\label{example: reduced K of constant Z}
    Recall from \cref{example: K theory of constant Z} that there is a splitting
    \[
        \pi^G_0(K_G(\underline{\mathbb{Z}}))\cong \prod\limits_{(H)\leq G} K_0(\mathbb{Z}[W_GH])
    \]
    where the product is over conjugacy classes of subgroups $H\leq G$ and $W_GH = N_G(H)/H$ is the Weyl group.  For all $H$ the augmentation $\mathbb{Z}[W_GH]\to \mathbb{Z}$ induces a splitting on $K$-theory which gives us
    \[
        K_0(\mathbb{Z}[W_GH])\cong K_0(\mathbb{Z})\times \widetilde{K}_0(\mathbb{Z}[W_GH])\cong \mathbb{Z}\times \widetilde{K}_0(\mathbb{Z}[W_GH])
    \]
    and putting this all together we have a splitting
    \[
        \pi^G_0(K_G(\underline{\mathbb{Z}}))\cong \left(\prod\limits_{(H)\leq G} \mathbb{Z}\right)\oplus\left(\prod\limits_{(H)\leq G} \widetilde{K}_0(\mathbb{Z}[W_GH])\right)\cong \Omega(G)\oplus \left(\prod\limits_{(H)\leq G} \widetilde{K}_0(\mathbb{Z}[W_GH])\right)
    \]
    where the second isomorphism comes from the observation that the Burnside ring $\Omega(G)$ is the free abelian group on the generators $G/H$, one for each conjugacy class of subgroup $H\leq G$.  
    
    This example shows explicitly that not every projective module over $\underline{\mathbb{Z}}$ is stably free.  For instance, when $G = C_{23}$ is a cyclic group of order $23$ then the group $\widetilde{K}_0(\mathbb{Z}[W_G1]) = \widetilde{K}_0(\mathbb{Z}[C_{23}])$.  By a theorem of Rim \cite{Rim} this group is isomorphic to the ideal class group of $\mathbb{Z}[\xi_{23}]$, where $\xi_{23}$ is a $23$rd root of unity, which is a cyclic of order $3$.
\end{example}

As the last example shows, the inclusion $\Omega(G)\to \pi_0^G(K_G(\underline{\mathbb{Z}}))$ splits. Since the image of the map $\pi_0^G(K_G(\underline{\mathbb{Z}}))\to \pi_0^G(K_G(\underline{S}))$ contains the class of all free $\underline{S}$ modules, the inclusion $\Omega(G)\to \pi_0^G(K_G(\underline{S}))$ factors through $\pi_0^G(K_G(\underline{\mathbb{Z}}))$. Thus whenever the inclusion $K_G(\underline{\mathbb{Z}})\to K_G(\underline{S})$ splits we will have a $\Omega(G)$ as a direct summand of $\pi_0^G(K_G(\underline{S}))$.  The kernel of the projection $\pi_0^G(K_G(\underline{S}))\to \Omega(G)$ is the reduced $K_0$ group.

\begin{example}\label{example: reduced K zero of a space}
    Let $X$ be a $G$-space which is $G$-connected.  Let $\underline{S} = \mathbb{Z}[\underline{\pi_1(X)}]$ be the coefficient ring of \cref{homotopy orbits twisted group ring} defined by 
    \[
        \underline{S}^H = \mathbb{Z}[\pi_1(X^H)].
    \]
    The augmentation maps $\mathbb{Z}[\pi_1(X^H)]\to \mathbb{Z}$ assemble into a map of coefficient systems $\underline{S}\to \underline{\mathbb{Z}}$ which splits the unit.   Thus $K_G(\underline{\mathbb{Z}})$ is a retract of $K_G(\underline{S})$ and the observations of the last paragraph imply there is a direct sum decomposition
    \[
      \pi_0^G(K_G(\underline{S}))\cong \Omega(G)\oplus \pi_0^G\widetilde{K}_G(\underline{S})  .
    \]
\end{example}

\section{Equivariant \textit{A}-theory and linearization}\label{sec:Equivariant Linearization}

In \cite{waldhausen:1983}, Waldhausen defines the $A$-theory of a space as the $K$-theory of the category of finitely dominated retractive spaces over $X$.

\begin{definition}
Let $R(X)$ be the category of \emph{retractive spaces over $X$}. The objects are spaces $Y$ together with maps\[
X \xrightarrow{i_Y} Y \xrightarrow{r_Y} X
\] which compose to the identity. Morphisms are maps $Y\to Y'$ rel $X$. A space over $X$ is \emph{finite} if it is obtained from $X$ by attaching finitely many cells.  A space over $X$ is \emph{finitely dominated} if it is a homotopy retract of a finite space over $X$.  We will write $R_{\fd}(X)$ for the full subcategory of finitely dominated retractive spaces over $X$.  

The category $R_{\fd}(X)$ is a Waldhausen category \cite{waldhausen:1983}. The Waldhausen structure is given by defining $wR_{\fd}(X)$ to be weak homotopy equivalences and $coR_{\fd}(X)$ to be those maps with the fiberwise homotopy extension property. The $A$-theory of $X$ is defined as
\[
    A(X) = K(R_{\fd}(X)).
    \]
\end{definition} 
The linearization map $\ell:A(X) \to K(\mathbb Z[\pi_1X])$ is a $2$-connected map relating the $A$-theory of a space $X$ to the $K$-theory of the group ring on $\pi_1(X)$, and is induced by sending $Y \in R_{\fd}(X)$  to the relative chain complex $C_\bullet(\tilde{Y},\tilde{X}) \in \mathrm{Perf}(\mathbb Z[\pi_1 (X)])$ \cite{waldhausen:1983,Klein-Malkiewich}.  The linearization plays an important role in computations related to $A$-theory; see \cite{Klein-Rognes} and \cite{Dundas} for discussion. 

In this section, we discuss how this story generalizes to the equivariant setting. We begin by reviewing the definition of equivariant $A$-theory from \cite{malkiewich/merling:2016} before constructing a version of the linearization map whose target is the genuine $G$-spectrum $K_G(\ZZ[\underline{\pi_1(X)}])$ from \cref{homotopy orbits twisted group ring}. We then show that this map is $2$-connected and recovers expected geometric invariants like the equivariant Wall finiteness obstruction and Whitehead torsion, as defined in \cite{Luck}. 

\subsection{Background on equivariant $A$-theory}\label{sec:bckgd on A theory}
If $X$ is a $G$-space, $R(X)$ inherits a $G$-action via exact functors, making it a Waldhausen $G$-category. In \cite{malkiewich/merling:2016}, Malkiewich--Merling first construct the \textit{coarse} equivariant $A$-theory of a $G$-space $X$ from the categorical homotopy fixed points of $R(X)$. This construction is called a coarse theory because the weak equivalences are just determined by the underlying space $X$, and do not take fixed points into account. In order to ensure the theory detects genuine equivalences, Malkiewich--Merling build a genuine equivariant $A$-theory spectrum from the categories of $H$-equivariant retractive spaces over $X$ for $H\leq G$.

\begin{definition}
For $H\leq G$, let $R^H_{\fd}(X)$ denote the Waldhausen category with\begin{itemize}
    \item objects: $H$-equivariant finitely dominated retractive spaces over $X$, i.e. $H$-spaces $Y$ with equivariant maps $i_Y$, $r_Y$ which are equivariant homotopy retracts of finite $H$-CW complexes rel $X$;
    \item morphisms: $H$-equivariant maps of retractive spaces;
    \item cofibrations/weak equivalences: $H$-equivariant maps which are cofibrations/weak equivalences of $H$-spaces; that is, maps $Y\to Z$ so that $Y^J\to Z^J$ is a cofibration/weak equivalence for all $J\leq H$.
\end{itemize}
\end{definition}

These Waldhausen categories come equipped with restriction, conjugation, and transfer maps which are deduced from categorical formulas but admit a nice geometric description.

\begin{theorem}[{\cite[Proposition 4.14]{malkiewich/merling:2016}}]\label{thm:fmlas for res cong transf}
Let $H\leq K$ and let $q\colon G/H\to G/K$ be the canonical quotient. Then\begin{itemize}
    \item[(i)] The restriction \begin{align*}
        R_H^K\colon R_{\fd}^K(X) &\to R_{\fd}^H(X)
    \end{align*} sends $(Y, i_Y, r_Y)$ to itself, with the $K$-action restricted to a $H$-action.
    \item[(ii)] Suppose $H'\leq G$ is conjugate to $H$, so $H'=gHg^{-1}$ for some $g\in G$. Then conjugation is given by\begin{align*}
        c_g\colon R_{\fd}^{H'}(X) &\to R_{\fd}^{H}(X)\\
        (Y, i_Y, r_Y) &\mapsto (Y, i_Y\circ g, g^{-1}\circ r_Y)
    \end{align*} where $h\in H$ acts on $Y$ by $g^{-1}hg$.
    \item[(iii)] The transfer $T_H^K\colon R_{\fd}^H(X)\to R_{\fd}^K(X)$ sends $(Y, i_Y, r_Y)$ to the pushout \[
    \begin{tikzcd}
    K\times_H X \ar[r] \ar[d] & X\ar[d] \\
    K\times_H Y \ar[r] & T_H^K(Y)
    \end{tikzcd}.
    \]\end{itemize}
\end{theorem}

In \cite[\S 4]{malkiewich/merling:2016}, Malkiewich--Merling show the data of $K(R_{\fd}^H(X))$ and the structure maps assembles into a spectral Mackey functor, and hence corresponds to a genuine $G$-spectrum.

\begin{definition}
The \textit{genuine equivariant $A$-theory} of a $G$-space $X$ is the spectral Mackey functor $A_G(X) := K_G(R_{\fd}(X))$ where\[
A_G(X)(G/H) = A_G(X)^H = K(R_{\fd}^H(X)),
\] with restriction and transfer induced by the formulas from \cref{thm:fmlas for res cong transf}.
\end{definition}

\begin{remark}
    Malkiewich--Merling's definition of equivariant $A$-theory uses categories of homotopy finite, as opposed to finitely dominated, spaces over a $G$-space $X$.  The difference is analogous to the choice of whether to define the $K$-theory of a ring in terms of the category of finitely generated free, as opposed to projective, modules.  Just as in the algebraic setting the only difference in the associated $K$-groups occurs at level $0$. 
\end{remark}
    
Non-equivariantly, $A$-theory of a connected space $X$ defined using homotopy finite retractive spaces always has $A_0(X)\cong \mathbb{Z}$ with the isomorphism given by an Euler characteristic.  On the other hand, defining $A$-theory using the category of finitely dominated retractive spaces over $X$ yields the more interesting group $A_0(X)\cong K_0(\mathbb{Z}[\pi_1(X)])$ with the isomorphism given by the linearization map.  Since the goal of this paper is to develop the linearization map in the equivariant setting we are naturally led to the definition of $A$-theory in terms of finitely dominated retractive spaces.

We now turn to the construction of equivariant linearization. For the remainder of this section we will let $X$ be a $G$-space that is equivariantly connected, meaning that $X^H$ is connected for all $H \leq G$; we will extend our constructions to more general $G$-spaces in \cref{subsec:non cntd linearization}. There is a split fibration sequence
\begin{align}\label{first sequence}
    X \to X_{hG} \to BG
\end{align}
that induces a short exact sequence 
\begin{align*}
    1 \to \pi_1(X) \to \tilde{G} \xrightarrow{\pi} G \to 1,
\end{align*}
where $\tilde{G}=\pi_1(X_{hG})$.  Moreover, the universal cover $\tilde{X}$ receives an action from $\tilde{G}$, making $C_\bullet(\tilde{X})$ a chain complex of $\mathbb Z[\tilde{G}]$-modules. The $G$-action on $\pi_1(X)$ extends to an action on $\mathbb Z[\pi_1(X)]$ through ring automorphisms. 
\begin{lemma}\label{lemma: homotopy orbits pi one is pi_1 twisted}
    There is an isomorphism of rings
    \[
        \ZZ[\pi_1(X_{hG})]\cong \ZZ[\pi_1(X)]_{\theta}[G]
    \]
    where the latter ring is the twisted group ring.
\end{lemma}
\begin{proof}
    The splitting of \eqref{first sequence} implies that $\tilde{G} \cong \pi_1(X) \rtimes G$ where the $G$-action on $\pi_1(X)$ is induced by the action on $X$.  Thus it suffices to prove the isomorphism of rings \[\mathbb Z[\pi_1(X) \rtimes G] \cong \mathbb Z[\pi_1(X)]_{\theta}[G].\]
    Note that the underlying abelian group of either side is isomorphic to the free abelian group on the set $\pi_1(X)\times G$.  In either case, the multiplication is determined by the rule 
    \[
        (\alpha,g)\cdot (\beta,h) = (\alpha(g\beta),gh)
    \]
    which implies the desired isomorphism.
\end{proof}

The next corollary is immediate from the last lemma and \cref{homotopy orbits twisted group ring}.

\begin{corollary}
    Let $X$ be a $G$-connected $G$-space and let $\underline{S}  = \mathbb{Z}[\underline{\pi_1(X)}]$ be the coefficient system from \cref{homotopy orbits twisted group ring}.  There is a splitting on $G$-fixed points
    \[
        K_G(\underline{S})^G\simeq \prod\limits_{(H)\leq G} K(\mathbb{Z}[\pi_1(X^H_{hWH})])
    \]
    where the product is over conjugacy classes of subgroups $H\leq G$.
\end{corollary}
\begin{remark}\label{remark: identify reduced K theory of spaces}
    The augmentation map $\ZZ[\pi_1(X^H_{hWH})]\to \ZZ$ induces a splitting on $K$-theory spectra
    \[
        K(\ZZ[\pi_1(X^H_{hWH})])\simeq K(\ZZ)\times \widetilde{K}(\ZZ[\pi_1(X^H_{hWH})]).
    \]
    Applying this splitting for each conjugacy class of subgroup $H\leq G$, and comparing with \cref{example: reduced K zero of a space} we identify the zeroth level of the equivariant $K$-theory of $\ZZ[\underline{\pi_1(X)}]$ as
    \[
        \pi_0^G(K(\ZZ[\underline{\pi_1(X)}]))\cong \Omega(G)\oplus \left(\widetilde{K}(\ZZ[\pi_1(X^H_{hWH})])\right)
    \]
    where $\Omega(G)$ is the Burnside ring of $G$.  
\end{remark}

In view of the corollary, there is a candidate for the linearization map on $G$-fixed points given by the non-equivariant linearizations applied factorwise 
\[
    A_G(X)^G \simeq \prod_{(H) \leq G} A(X^{H}_{hWH}) \to \prod_{(H) \leq G} K(\mathbb Z[\pi_1(X^H_{hWH})]\simeq K_G(\ZZ[\underline{\pi_1(X)}])^G.
\]

Although this map is evidently $2$-connected, it is a bit ad hoc. From a theoretical point of view, we would like an equivariant linearization map to come from a map of genuine $G$-spectra instead of just a map on fixed points.
The main objective of this section is to show that this factor-wise description lifts to a map of genuine $G$-spectra.

 \subsection{Lifting the linearization}
 Let $X$ be a based $G$-space.  Recall from \cref{A theory as a waldhausen Mackey} that equivariant $A$-theory is obtained from a pseudofunctor $\underline{R}(X)\colon \mathbb{B}^G\to \mathrm{Wald}$ which sends a subgroup $H\leq G$ to the category $R^H_{\fd}(X)$ of homotopy finite $H$-retractive spaces over $X$. 

Recall the coefficient system $\ZZ[\underline{\pi_1(X)}]$ from \cref{homotopy orbits twisted group ring}, with \[\ZZ[\underline{\pi_1(X)}](G/H)=\ZZ[{\pi_1(X^H)}]\] and restrictions induced by $X^H\to X^K$ for $K\leq H$. To construct the equivariant linearization, we will construct a pseudonatural transformation from $\underline{R}(X)$ to the pseudofunctor $\mathcal{P}\colon \mathbb{B}^G\to \mathrm{Wald}$ from \cref{spectral Mackey functor} which defines the the algebraic $K$ theory of $\ZZ[\underline{\pi_1(X)}]$.  This amounts to constructing functors
\[
    L^H\colon R^H_{\fd}(X)\to \mathrm{Perf}(R^G_H(\ZZ[\underline{\pi_1(X)}]))
\]
for all $H$ which induce a map on $K$-theory spectra and then checking that these functors are sufficiently compatible with the transfers and restrictions.  We will then show that the map $L^G$ recovers the factorwise description of the linearization map described at the end of last section.

 We begin by constructing the functor $L^G$; the functor $L^H$ for $H\leq G$ are constructed analogously.  For a space $X$ we write $\widetilde{X}$ for the universal cover of $X$.  For a retractive space $Y$ over $X$ we write $\widetilde{Y}$ for the pullback of $\widetilde{X}$ along the retraction $Y\to X$.  This is always a retractive space over $\widetilde{X}$. 

 \begin{definition}
     Let $X$ be a based $G$-space and suppose that $Y\in R^G_{\fd}(X)$. Let $L^G(Y)$ denote the chain complex of coefficient systems
     \[
        G/H\mapsto C_{*}(\widetilde{Y^H},\widetilde{X^H}).
     \]
 \end{definition}

 To see that $L^G(Y)$ is actually a chain complex of coefficient systems we need to construct the restriction maps. Note that for any $H\leq K\leq G$ the inclusions 
     \[
        X^K\to X^H \quad \mathrm{and}\quad Y^K\to Y^H
     \]
     induce maps on universal covers which, in turn, give chain maps
     \[
        R^K_H\colon C_{*}(\widetilde{Y^K},\widetilde{X^K})\to C_{*}(\widetilde{Y^H},\widetilde{X^H}) 
     \]
     which serve as restrictions.  The construction of Weyl group actions is essentially identical, giving us a chain complex of coefficient systems as claimed.
 
  \begin{proposition}\label{chains are perfect}
      The coefficient system $L^G(Y)$ is a chain complex of modules over $\ZZ[\underline{\pi_1(X)}]$.
 \end{proposition}
 \begin{proof}
  We want to show that the coefficient system \[G/H\mapsto C_{n}(\widetilde{Y^H},\widetilde{X^H}) \]
     is actually a $\ZZ[\underline{\pi_1(X)}]$-module for every $n$. By \cref{modules over coeff rings}, this statement is equivalent to each chain group $C_{n}(\widetilde{Y^H},\widetilde{X^H})$  being a module over the twisted group rings 
     \[
        \ZZ[\underline{\pi_1(X)}](G/H)_{\theta}[WH] \cong \ZZ[{\pi_1 (X^H)}]_{\theta}[WH].
     \]

    By \cref{twisted group ring modules}, a system $\ZZ[\pi_1(X^H)]_{\theta}[WH]$ is the same as a system of module over $\ZZ[\pi_1(X^H)]$ with semi-linear $WH$-action, such that the actions of $\pi_1*(X^H)$ are compatible with the restrictions. By standard arguments, $\pi_1(X^H)$ acts naturally on the pair $(\widetilde{Y^H},\widetilde{X^H})$ and checking this action is semilinear is straightforward. Additionally, in the standard model of the universal cover (for instance, in \cite{hatcher}) the universal cover construction is functorial in the action of the fundamental group action, in the sense that for any based continuous map $f\colon A\to B$ the diagram
    \[
        \begin{tikzcd}[column sep = large]
                \pi_1(A)\times \widetilde{A} \arrow[r,"\pi_1(f)\times \widetilde{f}"]\arrow[d]  & \pi_1(B)\times \widetilde{B} \ar[d]\\
    \widetilde{A}\ar[r,swap, "\widetilde{f}"] & \widetilde{B}
        \end{tikzcd}
    \]
    commutes, where the vertical arrows are the fundamental group actions.  It follows that the action $\ZZ[\pi_1(X^H)]$ on $C_{n}(\widetilde{Y^H},\widetilde{X^H})$ is compatible with the restriction maps, since these are induced by inclusions of based spaces.
 \end{proof}

Next, we show that $L^G(Y)$ is actually a perfect complex.

\begin{lemma}\label{L of finite complex}
    If $Y$ is a retractive space over $X$ which is also a finite $G$-CW complex relative to $X$ then the $\ZZ[\underline{\pi_1(X)}]$-complex $L^G(Y)$ is a perfect complex.
\end{lemma}
\begin{proof}
    By induction, it suffices to prove this when $Y$ is obtained from $X$ by attaching a single $G$-CW cell; let us assume it is of the form
    \[
        e = G/K\times \Delta^n
    \]
    for some $K\leq G$ and some $n$ and let $\varphi\colon \partial e\to X$ be the attaching map.  For any $H\leq G$ we have
    \[
        Y^H = e^H\cup_{\varphi} X^H
    \]
    and thus $Y^H$ is obtained from $X^H$ by attaching $|(G/K)^H|$ $n$-cells.  This gives an isomorphism
    \[
        C^{CW}_n(\widetilde{Y^H},\widetilde{X^H})\cong \ZZ[\pi_1(X^H)]\otimes \ZZ[(G/K)^H]
    \]
    with the Weyl group action coming from the action on both terms in the tensor product.  By definition, $\ZZ[(G/K)^H]$ is precisely the value of the free coefficient system $A_{G/K}$ generated by $G/K$ at level $G/H$.  The upshot is that we have an isomorphism
    \[
         L^G(Y)\cong \ZZ[\underline{\pi_1(X)}]\square A_{G/K}
    \]
    which is a finitely generated free $\ZZ[\underline{\pi_1(X)}]$-module.
\end{proof}

\begin{corollary}\label{L lands in perfect}
    If $Y$ is a homotopy retract of a finite $G$-CW complex relative to $X$ then $L^G(Y)$ is a perfect complex over $\ZZ[\underline{\pi_1(X)}]$.
\end{corollary}
\begin{proof}
    Since $Y$ is a homotopy retract of a finite CW complex, the previous lemma implies that $L^G(Y)$ is a chain homotopy retract of a perfect complex.  The result then follows from a bit of homological algebra, following \cite[Lemma 1.9 and Remark 1.10]{pedersen}. The arguments there, although stated for modules over a ring, translate to our setting as they do not use anything but homological algebra.
\end{proof}
    
 To summarize, we constructed a functor $L^G\colon  R^G_{\fd}(X)\to \mathrm{Perf}(\ZZ[\underline{\pi_1(X)}])$ which sends a space $Y\in R^G_{\fd}(X)$ to the coefficient system $L^G(Y)$ defined by
 \[
    L^G(Y)(G/H) = C_*(\widetilde{Y^H},\widetilde{X^H})
 \]
 which is perfect by \cref{L lands in perfect}.  Note that $L^G$ is evidently functorial and sends $G$-weak equivalences to quasi-isomorphisms and $G$-cofibrations to levelwise inclusions.
 Although this functor does not preserve pushouts along cofibrations, it nonetheless induces a well-defined map on $K$-theory, via a construction due to Thomason \cite[p. 334]{waldhausen:1983}. The proof of the following result is the same as \cite[Sections 7.2--7.3]{Klein-Malkiewich} for the non-equivariant linearization.
 
\begin{proposition}\label{prop:L induces map on fx pts}
     The functor $L^G$ induces a map  \begin{equation}\label{linearization on fixed points}
    K(L^G)\colon A_G(X)^G\to K_G(\ZZ[\underline{\pi_1(X)}])^G
 \end{equation}
 on $K$-theory spectra.
 \end{proposition}\begin{proof}
 Although $L^G$ is not an exact functor (as it does not preserve pushouts), it will still induce a map on $K$-theory spectra provided we use a slightly different model than the $S_\bullet$-construction. In particular, the $S'_\bullet$-construction (defined in \cite{BM:08} and discussed in \cref{subsec: appendix S dot prime}) produces an equivalent $K$-theory spectrum as the $S_\bullet$-construction but is functorial in the larger class of \textit{weakly exact} functors, i.e. functors that preserve the Waldhausen structure up to weak equivalence.
 
In particular, in order to show that $L^G$ induces a map on $K$-theory, it suffices to show that $L^G$ preserves homotopy cocartesian squares. Since pushouts, cofibrations, and weak equivalences in the category of chain complexes over $\ZZ[\underline{\pi_1(X)}]$ are computed levelwise, it suffices to check that the functor $R^{G}_{\fd}(X)\to \Perf_{\ZZ[\pi_1(X^H)]}$ given by
\[
    Y\mapsto C_*(\widetilde{Y^H},\widetilde{X^H})
\]
preserves homotopy cocartesian squares.  Since taking fixed points preserves homotopy cocartesian squares we can reduce further to checking only that the functor $R_{\rm fd}(X)\to \Perf_{\ZZ[\pi_1(X)]}$ given by
\[
    Y\mapsto C_*(\widetilde{Y},\widetilde{X})
\]
preserves homotopy cocartesian squares, which follows from an argument using Mayer-Vietoris and excision (c.f.\ \cite[Section 7.3]{Klein-Malkiewich}).
 \end{proof}
 
Similarly, for any $H\leq G$ we have a functor
 \[
    L^H\colon R^H_{\fd}(X)\to \mathcal{P}^H(\ZZ[\underline{\pi_1(X)}])
 \]
 defined exactly the same way as $L^G$ but treating $X$ as an $H$-space, and this functor induces a map $K(L^H)\colon A_G(X)^H\to K_G(\ZZ[\underline{\pi_1(X)}])^H$.  

In the remainder of this section, we will show that these functors have several desirable properties.  First, we show the functors $L^H$ for various $H\leq G$ are compatible with the categorical transfer and restrictions on $R^{(-)}_{\fd}(X)$ and $\mathcal{P}(\underline{\ZZ[\pi_1X]})$. 
We write $\res^K_H\colon R^K_{\fd}(X)\to R^H_{\fd}(X)$ and $\tr^K_H\colon R^H_{\fd}(X)\to R^K_{\fd}(X)$ for the transfer and restriction of retractive spaces over $X$.  If $Z\in R^K_{\fd}(X)$ and $Y\in R^H_{\fd}(X)$ recall that $\res^K_H(Z)$ is just $Z$ with restricted action.  The transfer of $Y$ is given by 
    \begin{align*}
        tr_H^K(Y)=\mathrm{colim}(K \times_H Y \leftarrow K \times_H X \rightarrow X)
    \end{align*}
    which can be rewritten as $|K/H|$ copies of $Y$ glued together along their respective copy of $X$. This is a $K$-space via the action of $K$ on $X$ together with permutation of the factors of $Y$ according to the action of $K$ on the coset $K/H$ and the residual action of $H$ on $Y$.

    Let $y\in {tr}^K_H(Y)$ be a point which lives in the component corresponding to the coset $\alpha H\in K/H$.  For any $J\leq K$, the point $y$ is in $({tr}_H^K(Y))^J$ if and only if $\alpha H\in (K/H)^J$ and $y$ is fixed by the residual action of $J^{\alpha}=\alpha^{-1}J\alpha\subset H$ on $Y$.  Thus we have an equality
    \begin{equation}\label{equation: fixed points of transfer}
        ({tr}_H^K(Y))^J = \bigcup\limits_{\alpha\in (K/H)^J} (Y^{J^{\alpha}}\cup X^{J^{\alpha}})
    \end{equation}

 \begin{theorem}\label{linearization commutes with tr/res}
The functors $L^H:R_{\fd}^H(X) \to \mathcal{P}^H(\underline{\ZZ}[\pi_1(X)])$ induce a map on $K$-theory $G$-spectra in the equivariant stable homotopy category.
\end{theorem}
\begin{proof}
    We consider $R^{-}_{\fd} = R^{-}_{\fd}(X)$ and $\mathcal{P}^{-} = \mathcal{P}^{-}(\underline{\ZZ}[\pi_1(X)])$ as contravariant functors from $\mathcal{O}_G$ into Waldhausen categories. On morphisms, they send a canonical quotient $q\colon G/H\to G/K$ to the restriction map along the subgroup inclusion $H\leq K$. On isomorphisms $G/H\to G/(gHg^{-1})$ we use the conjugation functors.  It is a straightforward consequence of the definition that the linearization map commutes with the restriction and conjugation functors, and thus this is a natural transformation between these two functors. 
    
    By \cref{Theorem: construction examples and maps}, it suffices to show that for any canonical quotient map $q\colon G/H\to G/K$, the \emph{mate} of the square
\[\begin{tikzcd}
	{\mathcal{P}^H} && {R^H_{\fd}} \\
	{\mathcal{P}^K} && {R^K_{\fd}}
	\arrow["{L^H}"', from=1-3, to=1-1]
	\arrow["{R^K_H}"', swap, from=2-1, to=1-1]
	\arrow["{\res^K_H}", swap, from=2-3, to=1-3]
	\arrow["{L^K}", from=2-3, to=2-1]
\end{tikzcd}\]
is a weak equivalence in $\mathcal{P}^K$, i.e. all the maps comprising the natural transformation are quasi-isomorphisms of chain complexes.  The mate is defined in \cref{subsec: appendix constructing examples}, but for the purposes of this proof it suffices to know that the mate is the natural transformation
\[
    I^K_HL^K\Rightarrow L^H\tr^K_H
\]
given by the diagram
\[\begin{tikzcd}
	{\mathcal{P}^H} & {\mathcal{P}^H} && {R^H_{\fd}} & {R^H_{\fd}} \\
	{\mathcal{P}^K} & {\mathcal{P}^K} && {R^K_{\fd}} & {R^K_{\fd}}
	\arrow["{I^K_H}"', from=1-1, to=2-1]
	\arrow[equal,""{name=0, anchor=center, inner sep=0}, from=1-2, to=1-1]
	\arrow["{L^H}"', from=1-4, to=1-2]
	\arrow[equal,""{name=1, anchor=center, inner sep=0}, from=1-4, to=1-5]
	\arrow["{\tr^K_H}", from=1-5, to=2-5]
	\arrow["{R^K_H}"', from=2-2, to=1-2]
	\arrow[equal,""{name=2, anchor=center, inner sep=0}, from=2-2, to=2-1]
	\arrow["{\res^K_H}", from=2-4, to=1-4]
	\arrow["{L^K}", from=2-4, to=2-2]
	\arrow[equal,""{name=3, anchor=center, inner sep=0}, from=2-5, to=2-4]
	\arrow["\epsilon"', shorten <=6pt, shorten >=9pt, Rightarrow, from=0, to=2]
	\arrow["\eta", shorten <=6pt, shorten >=9pt, Rightarrow, from=1, to=3]
\end{tikzcd}\]
where $\eta$ and $\epsilon$ are the unit and counit of the $\tr^K_H \dashv \res^K_H$ and $I^K_H\dashv R^K_H$ adjunctions, respectively.  

Fix a $Y\in R^H_{\fd}$, and any subgroup $J\leq K$.  Since quasi-isomorphism of chain complexes over a coefficient ring can be checked levelwise, it suffices to show that the map
\[
    I^K_HL^K(Y)(K/J)\Rightarrow L^H\tr^K_H(Y)(K/J)
\]
is a quasi-isomorphism.  According to the diagram defining the mate, this decomposes as the composite
\begin{align*}
    I^K_HL^H(Y)(K/J) & \xrightarrow{I^K_HL^K(\eta_Y)}I^K_HL^H\res^K_H\tr^K_H(Y)(K/J)\\
    & \xrightarrow{=} I^K_HR^K_HL^K\tr^K_H(Y)(K/J)\\
    & \xrightarrow{\epsilon_{L^K\tr^K_H(Y)}}L^K\tr^K_H(Y)(K/J)
\end{align*}

The source is, by \cref{induction on coefficient systems} and the definition of $L^H$,
\[
I^K_HL^H(Y)(K/J) \cong \bigoplus\limits_{\alpha\in (K/H)^J} C_{*}(\widetilde{Y^{J^{\alpha}}},\widetilde{X^{J^{\alpha}}})
\]
and the target is
\[
    L^K\tr^K_H(Y)(K/J)\cong C_*(\widetilde{(\tr^K_H(Y))^J},\widetilde{X^J})\cong C_* \left( \bigcup\limits_{\alpha\in (K/H)^J} \widetilde{Y^{J^{\alpha}}} ,\widetilde{X^{J}}\right).
\]
Using \eqref{equation: fixed points of transfer} we identify the $J$-fixed points of $\eta\colon Y\to \res^K_H\tr^K_H(Y)$ with the inclusion 
\[
    \widetilde{Y^{J^{\alpha}}}\xrightarrow{\alpha = eH} \bigcup\limits_{\alpha\in (K/H)^J} \widetilde{Y^{J^{\alpha}}}
\]
It follows that the map of interest is induced, on each component $\alpha$ in the domain, by the inclusion of $\widetilde{Y^{J^{\alpha}}}\to \bigcup\limits_{\alpha\in (K/H)^J} \widetilde{Y^{J^{\alpha}}}$.  This map is a quasi-isomorphism because relative singular homology takes finite relative unions to direct sums.
\end{proof}



 \begin{definition}
The \emph{equivariant linearization} $L\colon A_G(X)\to K_G(\ZZ[\underline{\pi_1(X)}])$ is the functor defined levelwise by applying $K$-theory to the functors
\[
     L^H\colon R^H_{\fd}(X) \to \mathcal{P}^H(\ZZ[\underline{\pi_1(X)}]).
 \]
 \end{definition}
 
 \begin{remark}
 The map of spectra underlying $L$ is the the classical linearization map $A(X)\to K(\ZZ[\pi_1 (X)])$ as described in \cite[Section 7]{Klein-Malkiewich}.
 \end{remark}

 Now we show how to interpret the linearization after the splittings of \cref{K theory splitting} and the splitting of $A$-theory \eqref{A theory splitting}.  This result allows us to infer several properties of the equivariant linearization from the classical analog.  
 
\begin{theorem}\label{linearization splitting}
The following diagram commutes up to homotopy:\[
\begin{tikzcd}
A_G(X)^G \ar[r, "L^G"] \ar[d, swap, "\simeq"] & K_G(\ZZ[\underline{\pi_1(X)}])^G\ar[d, "\simeq"] \\
\prod\limits_{(H)\leq G} A(X_{hWH}^H) \ar[r, swap, "\prod \ell_H"] & \prod\limits_{(H)\leq G} K(\ZZ[\pi_1(X^H_{hWH})])
\end{tikzcd}
\]
where the maps $\ell_H$ are the ordinary linearization maps for the spaces $X^H_{hWH}$.
\end{theorem}
\begin{proof}
We claim the functor $L^G$ respects the filtrations on the categories $R^G_{\fd}(X)$ and $\mathcal{P}(\ZZ[\underline{\pi_1(X)}])$ which produce the splittings. Indeed, if $Y\in R_{\leq i}(X)$ (meaning that $Y\setminus X$ has no isotropy conjugate to $H_j$ for $j>i$) and $j>i$ we see that $L(Y)(G/H_j)=C_*(\widetilde{Y^{H_j}} , \widetilde{X^{H_j}})\cong 0$ which is precisely what it means for $L^G(Y)$ to be in $\mathrm{Perf}_{\leq{i}}(\ZZ[\underline{\pi_1(X)}])$ as defined in \cref{subsec:perf}. 

From this observation, it follows the map $L^G$ induces after the splitting can be identified with the wedge of the maps 
\[
    L^G_i = L^G|_{R_i(X)} \colon R_i(X)\to \mathrm{Perf}_i(\ZZ[\underline{\pi_1(X)}])
\]
where $R_i(X)$ is the subcategory of retractive spaces $Y$ such that every element in $Y\setminus X$ has isotropy exactly conjugate to $H_i$.
For every $i$ we will construct a natural transformation $\epsilon$ inhabiting the square
\[
    \begin{tikzcd}
        R_i(X) \ar["L^G_i"]{r} \ar[swap, "p"]{d} & \mathcal{P}_i(\ZZ[\underline{\pi_1(X)}]) \ar["q"]{d} \\
        R_{\fd}(X^{H_i}_{hWH_i}) \ar[swap, "\ell_{H_i}"]{r} & \mathrm{Perf}(\ZZ[\pi_1(X^{H_i}_{hWH_i})]) \ar[shorten = 4mm, from = 1-2,to=2-1,Rightarrow, "\epsilon"']
    \end{tikzcd}
\]
where the vertical functors are isomorphisms after $K$-theory.  We show that the component of $\epsilon$ at every $Y\in R_i(X)$ is a quasi-isomorphism; the result then follows immediately. 

The functor $p$ is defined by $p(Y) = Y^{H_i}_{hWH_i}$ and the functor $q$ sends a perfect complex over $\ZZ[\underline{\pi_1(X)}]$ to its evaluation at $G/H_i$.  From the definitions, we compute
\[
    (\ell_{H_i}\circ p) = C_*(\widetilde{Y^{H_i}_{hWH_i}},\widetilde{X^{H_i}_{hWH_i}})
\]
and
\[
    (q\circ L^G_i)(Y) = C_*(\widetilde{Y^{H_i}},\widetilde{X^{H_i}}).
\]
Define $\epsilon_Y$ to be the map induced by the map of pairs
\begin{equation}\label{map of pairs}
    (\widetilde{Y^{H_i}},\widetilde{X^{H_i}})\to (\widetilde{Y^{H_i}_{hWH_i}},\widetilde{X^{H_i}_{hWH_i}})
\end{equation}
which we get by picking any basepoint in $EWH_i$. 

To see this is a quasi-isomorphism consider the commutative diagram:
\[
    \begin{tikzcd}
        \widetilde{X^{H_i}} \ar["\widetilde{f}"]{r} \ar{d} & \widetilde{X^{H_i}_{hWH_i}} \ar{d} & \\
        X^{H_i} \ar[swap, "f"]{r} & X^{H_i}_{hWH_i} \ar{r} & BWH_i 
    \end{tikzcd}
\]
where the bottom row is the fibration \eqref{first sequence} and the vertical maps are the universal covering fibrations.  Since $BWH_i$ has homotopy concentrated in degrees $0$ and $1$, the map $f$ is an isomorphism on all homotopy groups above degree $1$.  Since both vertical maps also have this property we see that the map $\widetilde{f}$ is a weak homotopy equivalence. 

For $Y$, the same argument shows that the map $\widetilde{Y^{H_i}}\to \widetilde{Y^{H_i}_{hWH_i}}$ induces an isomorphism on all $\pi_n$ for $n\geq2$. 
 Unfortunately, $\widetilde{Y^{H_i}}$ and $\widetilde{Y^{H_i}_{hWH_i}}$ might fail to be simply connected and so we need to check that this map also gives an isomorphism on fundamental groups.  By \cite[Corollary 2.2.3]{May-Ponto}, and the fact that coverings induce isomorphisms on $\pi_2$, we see there is a short exact sequence
 \[
    0\to \pi_1(\widetilde{Y^{H_i}})\to \pi_1(Y^{H_i})\xrightarrow{r_*}\pi_1(X^{H_i})\to 0
 \]
 which is right-split giving an isomorphism
 \[
    \pi_1(Y^{H_i})\cong \pi_1(\widetilde{Y^{H_i}})\rtimes \pi_1(X^{H_i})
 \]
 and identifies $\pi_1(\widetilde{Y^{H_i}})\cong \ker(r_*)$.  On the other hand, the same argument shows that 
 \[
    \pi_1(Y^{H_i}_{hWH_i}))\cong \pi_1(\widetilde{Y^{H_i}_{hWH_i}})\rtimes \pi_1(X^{H_i}_{hWH_i})
 \]
 and identifies $\pi_1(\widetilde{Y^{H_i}_{hWH_i}})$ with $\ker(r_{hWH_i})_*$. The split fibration \eqref{first sequence} gives natural isomorphisms
 \[
    \pi_1(\widetilde{Y^{H_i}_{hWH_i}})\cong \pi_1(\widetilde{Y^{H_i}})\rtimes WH_i
 \]
 and so we have 
 \[
    \pi_1(\widetilde{Y^{H_i}_{hWH_i}}) \cong \ker(r_{hWH_i})_*\cong \ker(r_*\rtimes WH_i)\cong \ker(r_*)\cong \pi_1(\widetilde{Y^{H_i}})
 \]
 finishing the proof that the map $\widetilde{Y^{H_i}}\to \widetilde{Y^{H_i}_{hWH_i}}$ is a weak homotopy equivalence.

The analysis above shows that the map \eqref{map of pairs} induces a homology isomorphism on pairs and thus induces quasi-isomorphisms on singular chain complexes.
\end{proof}

We can now leverage facts about non-equivariant linearization to say things about equivariant linearization.  Of particular interest is the following result.
\begin{theorem}\label{linearization is 2 connected}
    The equivariant linearization map \[L\colon A_G(X)\to K_G(\ZZ[\underline{\pi_1(X)}])\] is a $2$-connected map of genuine $G$-spectra.  That is, it induces an isomorphisms on $\pi^H_0$ and $\pi^H_1$ for all subgroups $H\leq G$.
\end{theorem}
\begin{proof}
    This is immediate from \cref{linearization splitting} since the ordinary linearization maps $\ell_H$ are $2$-connected.
\end{proof}

\subsection{Linearization for non-$G$-connected spaces}\label{subsec:non cntd linearization}
In this section, we outline how to extend the construction of the linearization map to $G$-spaces $X$ which are not equivariantly connected. However, we still require that $X^G\neq \varnothing$ and $\abs{\pi_0(X^H)}<\infty$ for all $H\leq G$. The key idea is to replace $\ZZ[\underline{\pi_1(X)}]$ with an equivariant version of the fundamental groupoid \cite[Section 1.1]{CMW} (see also the \textit{fundamental category} of \cite[Definition 8.15]{Luck}).

\begin{definition}\label{defn:eq fundamental gpd}
    If $X$ is a $G$-space, the \textit{equivariant fundamental groupoid} (for $G$ finite) is a category $\Pi_G(X)$ whose objects are $G$-maps $x\colon G/H\to X$. A morphism $(\alpha, \omega)\colon (x\colon G/H\to X)\to (y\colon G/K \to X)$ is a $G$-map $\alpha\colon G/H\to G/K$ along with an equivalence class $\omega$ of paths in $X^H$ from $x$ to $y\circ \alpha$ (where two paths are equivalent if they are homotopic rel endpoints).
\end{definition}

There is a projection map $\pi\colon \Pi_G(X)\to \mathcal O_G$ which sends $x\colon G/H\to X$ to $G/H$ and $(\alpha, \omega)\mapsto \alpha$. The fiber of $\pi$ over $G/H$ is the non-equivariant fundamental groupoid $\Pi(X^H)$ \cite[Remark 1.4]{CMW}. It follows that $\Pi_G(X)$ is equivalent to the Grothendieck construction of the following functor.
 
\begin{definition}
    Let $\Pi(\underline{X})\colon \mathcal{O}_G^{\op}\to \Gpd$ be the coefficient system of groupoids which sends $G/H$ to the non-equivariant fundamental groupoid $\Pi(X^H)$. A map $G/H\to G/K$ in $\mathcal O_G$ induces a functor $\Pi(X^K)\to \Pi(X^H)$ induced by the inclusion $X^K\hookrightarrow X^H$.
\end{definition}

The definition above is more clearly a generalization of $\ZZ[\underline{\pi_1(X)}]$; in particular, when $X$ is $G$-connected, we have $\Pi(X^H)\simeq B(\pi_1(X^H))$ for all $H\leq G$. In the remainder of this section, we will freely pass between $\Pi_G(X)$ and $\Pi(\underline{X})$.

\begin{definition}
    An \textit{equivariant local system on $X$} is a functor $\Pi_G(X)\to \Ab$.
\end{definition}

An equivariant local system on $X$ is equivalently specified by local systems
    \[
        M_H\colon \Pi(X^H)\to \Ab
    \]
    for all $H\leq G$, along with natural transformations 
\[\begin{tikzcd}
	{\Pi(X^H)} && \Ab \\
	\\
	{\Pi(X^K)}
	\arrow[""{name=0, anchor=center, inner sep=0}, "{M_H}", from=1-1, to=1-3]
	\arrow["f"', from=1-1, to=3-1]
	\arrow["{M_K}"', from=3-1, to=1-3]
	\arrow["{\eta_f}"', shorten <=14pt, shorten >=19pt, Rightarrow, from=0, to=3-1]
\end{tikzcd}\]
for every $f\in \mathcal{O}_G^{\op}$ which are suitably compatible. 
Note that if $X$ is $G$-connected, then an equivariant local system is same thing as a $\ZZ[\underline{\pi_1(X)}]$-module. In general, the category of equivariant local systems is an abelian category with enough projectives. We note that may also replace abelian groups with any other abelian category $\cat C$ (such as chain complexes), and we will refer to functors $\Pi_G(X)\to \cat C$ as \textit{equivariant local systems in $\cat C$}.

\begin{lemma}\label{lem: left adj of ev on H for loc sys}
    Let $\cat C$ be an abelian category. For all $H\leq G$, the evaluation functor \[ev_H\colon \Fun(\Pi_G(X), \cat C)\to \Fun(\Pi(X^H), \cat C)\] has a left-adjoint which preserves projective objects.
\end{lemma}\begin{proof}
    The left adjoint is the left Kan extension along the inclusion of the fiber $\Pi(X^H)\hookrightarrow \Pi_G(X)$. The fact that the left adjoint preserves projective objects follows from the fact that its right adjoint $ev_H$ is exact.
\end{proof}

We now follow the methods of the previous section to construct the linearization. For ease of notation we write $\Pi_H(X) := R^G_H(\Pi_G(X)) = \Pi_H(R^G_H X)$ and freely pass between $\Pi_H(X)$ and $R^G_H\Pi(\underline{X})$. As before, we will construct functors\[
L^H \colon R^H_{\fd}(X) \to \Perf(\Pi_H(X))
\] which induce the linearization $L$ on equivariant $K$-theory spectra. It again suffices to describe $L^H$ for $H=G$. The goal is to construct $L^G$ so that \cref{chains are perfect}, \cref{L of finite complex}, and \cref{L lands in perfect} hold when $\ZZ[\underline{\pi_1(X)}]$ is replaced with $\Pi(\underline{X})$.

\begin{definition}
    Given $Y\in R^G_{\fd}(X)$ and $x\in X^H$, define \[
    L^G(Y)(x) = C_*(\widetilde{Y^H_{x}}, \widetilde{X^H_{x}})
    \] where $X^H_{x}$ is the connected component of $X^H$ containing $x$, $\widetilde{X^H_{x}}$ is the universal cover of $X^H_x$ with respect to the basepoint $x$, $Y^H_{x}$ is the preimage of $X_x^H$ under the retraction, and $\widetilde{Y^H_{x}}$ is the pullback of ${Y^H_{x}}$ along $\widetilde{X^H_{x}}\to {X^H_{x}}$.
\end{definition}

We now explain what happens to the Weyl group actions in the disconnected setting, which is similar to the discussion in \cite{andrzejewski:1986}. Let $X_\alpha^H$ be a connected component of $X^H$, so then $NH_\alpha = \{n\in NH\mid nX^H_\alpha = X^H_\alpha\}$ and $WH_\alpha = \{w\in WH\mid wX^H_\alpha = X^H_\alpha\}$ act on $X^H_\alpha$. In particular, if $x\in X^H$ is in the connected component $X^H_\alpha$, then $L^G(Y)(x)$ is a $\ZZ[\pi_1(X^H_\alpha)]_\theta[WH_\alpha]$-module.
The following is the analog of \cref{chains are perfect}.

\begin{proposition}
The assignment $L^G(Y)$ extends to an equivariant local system of chain complexes.
\end{proposition}\begin{proof}
First observe that a morphism $x\to x'$ in $\Pi(X^H)$ induces a map of chain complexes\[
 C_*(\widetilde{Y^H_{x}}, \widetilde{X^H_{x}}) \xrightarrow{\cong} C_*(\widetilde{Y^{H}_{x'}}, \widetilde{X^{H}_{x'}}),
\] so $L^G(Y)$ defines a local system on $\Pi(X^H)$. Now, if $f\colon G/K\to G/H$ is a map in $\mathcal O_G$ and $x\in X^H$, then we observe that there are maps $X^H_{x}\hookrightarrow X^{K^g}_{x} \xrightarrow{\cong} X^{K}_{g^{-1}\cdot x}$ where $f(eK)=gH$. This induces a chain map \[
 C_*(\widetilde{Y^{H}_{x}}, \widetilde{X^{H}_{x}}) \to C_*(\widetilde{Y^{K}_{g^{-1}\cdot x}}, \widetilde{X^{K}_{g^{-1}\cdot x}})
    \] and it is straightforward to check that this defines a natural transformation $\eta_f$ of local systems, and hence $L^G(Y)$ is indeed an equivariant local system.
\end{proof}

Note that the Weyl group action on $X^H$ appears when $K=H$ and the map $f$ in the proof above is an isomorphism in $\mathcal O_G$.
We now show that the linearization takes homotopy finite spaces to perfect systems, meaning that the values are quasi-isomorphic to a chain complex of equivariant local systems which is bounded and levelwise projective.

\begin{lemma}
    If $Y\in R^G_{\rm hf}(X)$, then $L^G(Y)$ is a perfect equivariant local system.
\end{lemma}
\begin{proof}
    As in \cref{L of finite complex}, it suffices to prove the claim for $Y=X\cup e_n$, where $e_n = G/K\times \Delta^n$ for some $K\leq G$. We will show that $L^G(Y)$ satisfies the universal property of the left adjoint of \cref{lem: left adj of ev on H for loc sys} applied to $Y$, i.e.\[
    {\rm Hom}(L^G(Y), M) \cong {\rm Hom}(L^G(Y)(G/H), M(G/H))
    \] where the left side is taken in the category of equivariant local systems of chain complexes and the right side replaces the domain category with $\Pi_G(X^H)$. Recall that $L^G(Y)(G/H)$ is $L^G(Y)$ restricted to $\Pi(X^H)$.

    A natural transformation $\eta\colon L^G(Y)\Rightarrow M$ is specified by a compatible collection of maps of chain complexes\[
    \eta_x\colon C_*(\widetilde{Y^H_x}, \widetilde{X^H_x}) \to M(x)
    \] for $x\in X^H$. Observe that if $H\not\leq_G K$, then $C_*(\widetilde{Y^H_x}, \widetilde{X^H_x})=0$ since $Y=X\cup e_n$. Otherwise, if $H^g\leq K$, then $\eta_x$ may be written as\[
        L^G(Y)(x) \xrightarrow{c_g} L^G(Y)(g^{-1}xg) \to M(g^{-1}xg) \xrightarrow{c_g^{-1}} M(x)
    \] which is to say that $\eta$ is equivalent to the data of a morphism $L^G(Y)(G/H)\Rightarrow M(G/H)$, as claimed.
\end{proof}

Since the proof of \cref{L lands in perfect} is purely formal, it applies in this context as well and so we deduce the desired corollary.

\begin{corollary}
    If $Y\in R^G_{\fd}(X)$, then $L^G(Y)$ is a perfect equivariant local system.
\end{corollary}

Piecing these results together, we obtain a linearization map for $G$-spaces $X$ which are not $G$-connected.

\begin{theorem}
    The functors $\{L^H\}_{H\leq G}$ induce a linearization map on $G$-spectra\[
    L\colon A_G(X) \to K_G(\Pi(\underline{X})).
    \]
\end{theorem}\begin{proof}
    We first claim that the functors $L^H\colon R^H_{\rm fd}(X)\to \Perf(\Pi_H(X))$ induce maps on $K$-theory spectra. This follows from an argument identical to \cref{prop:L induces map on fx pts}, using the $S_\bullet '$-construction and reducing the claim that $L^H$ preserves cocartesian squares to the same claim on each connected component of $X^H$. Second, we observe that essentially the same proof as in \cref{linearization commutes with tr/res} shows that the functors $L^H$ commute with the categorical restriction and transfer maps. Consequently, we obtain a natural transformation which induces a map of genuine $G$-spectra $L\colon A_G(X)\to K(\Pi(\underline{X}))$, as claimed.
\end{proof}

\begin{remark}
    One can also deduce that $L$ is $2$-connected in this case by proving a version of \cref{linearization splitting}. To do this, one first proves that \[
    K(\Pi(\underline{X}))^G \simeq \prod_{(H)\leq G} K(\Pi(X^H_{hWH})),
    \] and the method of proof is essentially the same as \cref{K theory splitting}, where one filters the category of $\Pi(\underline{X})$-modules by isotropy and shows that the evaluation functor ${\rm ev}_{G/H}$ from projective ``$H$ generated'' equivariant local systems to projective local systems on $X^H_{hWH}$ is an exact equivalence of categories (see also \cite[\S 10]{Luck}).
\end{remark}
\subsection{Recovering geometric invariants from equivariant $A$-theory}\label{subsec: geometric invariants}

The (non-equivariant) $K$-theory of the group ring $\ZZ[\pi_1(X)]$ encodes many important geometric invariants, such as the Euler characteristic, Wall finiteness obstruction, and Whitehead torsion. In light of the $2$-connected linearization map $A(X)\to K(\ZZ[\pi_1(X)])$, these invariants lift from $K(\ZZ[\pi_1(X)])$ to $A$-theory. Indeed, $A(X)$ encodes more information about the geometry of $X$ than $K(\ZZ[\pi_1(X)])$, as is evidenced by the stable parametrized $h$-cobordism theorem \cite{waldhausen:1983, WaldhausenJahrenRognes}. 

In the 1970s-1980s, mathematicians constructed equivariant versions of the Euler characteristic \cite{Luck}, Wall finiteness \cite{Baglivo, andrzejewski:1986}, and Whitehead torsion \cite{Luck, illman, araki/kawakubo:scob} which lived in equivariant analogues of the $K$-groups of $\ZZ[\pi_1(X)]$. In this section, we show how to recover these invariants (which were not previously known to come from the homotopy groups of some genuine $G$-spectrum) from the equivariant $A$-theory of Malkiewich--Merling, using the equivariant linearization map constructed in \cref{sec:Equivariant Linearization}. 

We obtain our results by comparing the $K$-theory of coefficient systems from \cref{sec:KT of CS} with constructions of L\"uck. In \cite{Luck}, L\"uck constructs equivariant Wall finiteness obstructions and Whitehead torsion which live in the $K$-theory of a certain \textit{EI categories}, which are defined to be categories in which every endomorphism is an isomorphism (see \cite[Definition 9.2]{Luck}).
Examples of EI categories are the orbit category $\mathcal{O}_G$ as well as its opposite $\mathcal{O}_G^{\op}$.  Another important example for our purposes is the equivariant fundamental groupoid $\Pi_G(X)$ of a $G$-space $X$ (\cref{defn:eq fundamental gpd}).  

\begin{definition}
    A \textit{module} over an EI category $\cat{C}$ is a functor $\cat{C}^{\op}\to \Ab$ and a morphism of modules is a natural transformation; the category of modules is denoted by $\mathrm{Mod}(\cat{C})$. 
\end{definition}

The category $\mathrm{Mod}(\cat{C})$ is abelian with enough projectives.  Given a $\cat{C}$-module $M$ and any subset of elements $S\subset \cup_{a\in \mathrm{ob}\cat{C}}M(a)$, we write $\cat{C}S$ for the smallest submodule of $M$ containing every element in $S$.  We say that $M$ is \textit{finitely generated} if $M = \cat{C}S$ for some finite set $S$. We write $\mathrm{FPMod}(\cat{C})$ for the category of finitely generated projective $\cat{C}$-modules, which is an exact category, hence also a Waldhausen category.

\begin{definition}
    The $K$-theory of an EI category is $K(\mathrm{FPMod}(\cat{C}))$.
\end{definition}

In \cite[Section 14]{Luck}, it is shown that the low-degree homotopy groups of $K(\Pi_G(X))$, and quotients of these groups, are a natural home for the equivariant Euler characteristic, Wall finiteness obstruction, and Whitehead torsion.  

\begin{theorem}[{\cite[Section 14]{Luck}}]
    For a $G$-finitely dominated $G$-space $X$ there exists abelian groups $U^G(X)$, $\mathrm{Wa}^G(X)$, natural in $X$, such that $K_0(\Pi_G(X))\cong U^G(X)\oplus \mathrm{Wa}^G(X)$.  Moreover there is an element $\omega_G(X)\in \mathrm{Wa}^G(X)$, natural in $X$, which is zero if and only if $X$ has the $G$-homotopy type of a finite $G$-CW complex.
\end{theorem}
\begin{theorem}[{\cite[Section 14]{Luck}}]
Let $G$ be a finite $G$-CW complex.  There is a group $\mathrm{Wh}^G(X)$, which is a quotient of $K_1(\Pi_G(X))$, such that for any $G$-homotopy equivalence $f\colon X\to X$ there is an element $\tau(f)\in \mathrm{Wh}^G(X)$ which is zero if and only if $f$ is a simple $G$-homotopy equivalence.
\end{theorem}

We observe here that for $G$-connected $G$-spaces $X$, the category $\mathrm{Mod}(\Pi_G(X))$ is equivalent to the category of modules over the coefficient system $\ZZ[\underline{\pi_1(X)}]$. 

\begin{proposition}
    For a $G$-connected $G$-space $X$ there is an equivalence of abelian categories between $\Mod(\Pi_G(X))$ and the category of modules over the coefficient system $\ZZ[\underline{\pi_1(X)}]$.
\end{proposition}
\begin{corollary}
    For $X$ a $G$-connected $G$-space there is an equivalence of spectra
    \[
        K_G(\ZZ[\underline{\pi_1(X)}])^G\simeq K(\mathrm{FPMod}(\Pi_G(X))).
    \]
\end{corollary}

The $2$-connectedness of our linearization map (\cref{linearization is 2 connected}) allows us to then lift L\"uck's invariants to the equivariant $A$-theory of Malkiewich--Merling. The class of the Wall finiteness obstruction can be made explicit by identifying $\pi_0^G(A_G(X))$ with the ``universal additive invariant'' of retractive $G$-spaces over $X$, in the sense of \cite[Theorem 6.9]{Luck}.

\begin{corollary}\label{corollary: wall}
    Let $X$ be a $G$-finitely dominated $G$-connected $G$-space.  There is a splitting $\pi^G_0(A_G(X))\cong U^G(X)\oplus \mathrm{Wa}^G(X)$.  Under this isomorphism, the class $\omega_G(X)$ is the $\mathrm{Wa}^G(X)$ is component of the class of the retractive space $[X\amalg X]\in \pi^G_0(A_G(X))$.  The component of $[X\amalg X]$ in $U^G(X)\cong \Omega(X)$ is the equivariant Euler characteristic of $X$.
\end{corollary}
\begin{remark}
    We note that the group $\mathrm{Wa}^G(X)$ can be identified with the the reduced $K$-theory, from \cref{subsec: reduced K theory}, of the coefficient ring $K_G(\ZZ[\underline{\pi_1(X)}])$.
\end{remark}

We can similarly lift the Whitehead torsion to the homotopy groups of $A_G(X)$.  Recall that given any weak equivalence $f\colon a\to b$ in a Waldhausen category $\cat{C}$ there is an associated element $[f]\in K_1(\cat{C})$.  
\begin{corollary}\label{corollary: whitehead}
    Let $X$ be a finite $G$-CW complex which is $G$-connected and let $f\colon X\to X$ be a $G$-homotopy equivalence. There is a surjection $q\colon \pi_1^G(A_G(X))\to \mathrm{Wh}^G(X)$, and $q([f]) = \tau_G(f)$.
\end{corollary}

We end with a discussion of how Whitehead torsion is related to Malkewich--Merling's work on an equivariant version of the stable parametrized $h$-cobordism theorem. In \cite{malkiewich/merling:2022}, Malkiewich--Merling show that when $M$ is a pointed smooth $G$-manifold, there is a cofiber sequence of $G$-spectra\[
{\bf H}_G(M)\to \Sigma^{\infty}_G M \xrightarrow{\alpha_M} A_G(M)
\] where ${\bf H}_G(M)$ is a $G$-spectrum whose fixed points can be identified ${\bf H}_G(M)^G\simeq \prod_{(H)\leq G} \mathcal H^{\infty}(M^H_{hWH})$ in terms of spaces of (non-equivariant) $h$-cobordisms. Ongoing work of Malkiewich--Merling, joint with Goodwillie--Igusa \cite{GIMM:23}, seeks to describe ${\bf H}_G(M)^G$ as a space of stable $G$-equivariant $h$-cobordisms, although connecting equivariant Whitehead torsion to equivariant cobordisms in the form of an $h$-cobordism theorem is more subtle than in the non-equivariant setting  \cite{araki/kawakubo:scob, steinberger/west:hcob}.

\begin{definition}
    Let ${\rm Wh}^a_G(X)$ be the \textit{equivariant algebraic Whitehead spectrum}, defined abstractly as the cofiber\[
    \Sigma^\infty_G X_+ \to K_G(\ZZ[\underline{\pi_1(X)}]) \to {\rm Wh}_G^a(X)
    \] where the first map is defined as the composition $\Sigma^\infty_G X_+ \xrightarrow{\alpha} A_G(X) \xrightarrow{L} K_G(\ZZ[\underline{\pi_1(X)}])$.  When $G=e$ is the trivial group we omit the $G$ and just write $\mathrm{Wh}^a(X)$.  
\end{definition}
\begin{remark}
    Non-equivariantly, the composite $L\circ \alpha$ can be understood on $\pi_1$ as sending a class $g\in \pi_1^s(X)$ to the class of the $1\times 1$ matrix $[g]$ in $K_1(\ZZ[\pi_1(X)])$. Since Whitehead torsion lives the in the quotient of $K_1(\ZZ[\pi_1(X)])$ by such elements we recognize the algebraic Whitehead group the natural home for Whitehead torsion.  The observations that follow amount to the fact that this same story carries through in the equivariant setting.
\end{remark}

We denote the cofiber of $\Sigma^{\infty}_G M \xrightarrow{\alpha_M} A_G(M)$ by $\mathrm{Wh}_G(X)$ and call this the equivariant \textit{geometric Whitehead spectrum} of $X$.  If $X$ is a compact smooth $G$-manifold, the work of Malkiewich--Merling \cite{malkiewich/merling:2022} identifies \[
{\rm Wh}_G(X)^K \simeq \prod_{(H)\leq K} {\rm Wh}(X^H_{hWH})
\] where ${\rm Wh}(X^H_{hWH})$ is the Whitehead spectrum of $X^H_{hWH}$.

\begin{theorem}
    For every $K\leq G$, there is a commutative diagram of spectra\[\begin{tikzcd}
    {\rm Wh}_G(X)^K \ar[r, "\simeq"] \ar[d] & \prod_{(H)\leq K} {\rm Wh}(X^H_{hWH}) \ar[d] \\
    {\rm Wh}_G^a(X)^K \ar[r, swap, "\simeq"] & \prod_{(H)\leq K} {\rm Wh}^a(X^H_{hWH})
    \end{tikzcd}
    \] where the vertical maps are $2$-connected.
\end{theorem}\begin{proof}
    There is a commutative diagram \[
\begin{tikzcd}
    \Sigma^{\infty}_G X_+ \ar[r, "\alpha"] \ar[d, equal] & A_G(X)\ar[r] \ar[d, "L"] & {\rm Wh}_G(X) \ar[d] \\
   \Sigma^{\infty}_G X_+ \ar[r] & K_G(\ZZ[\underline{\pi_1(X)}]) \ar[r] & {\rm Wh}_G^a(X)
\end{tikzcd}
\] and, since $L$ is $2$-connected, the induced map ${\rm Wh}_G(X)\to {\rm Wh}_G^a(X)$ is also $2$-connected. Moreover, since $\alpha$ and $L$ are compatible with the splittings on fixed points, we also obtain a splitting\[
{\rm Wh}_G^a(X)^K \simeq \prod_{(H)\leq K} {\rm Wh}^a(X^H_{hW_GH})
\] which is compatible with the analogous splitting on ${\rm Wh}_G(X)^K$. 
\end{proof}

Comparing this argument to L\"uck's, we see that applying $\pi_1$ to the result above recovers the isomorphism from \cite[Theorem 14.16]{Luck}. 

\appendix
\section{Proof of \cref{MMProp4.6}}\label{sec:tech pf app}

This appendix is devoted to the proof \cref{MMProp4.6}. We assume that all of our $G$-sets come equipped with a choice of total ordering, but we do not require morphisms of $G$-sets to respect this ordering.

\begin{lemma}[{\cite[Definition 4.3]{malkiewich/merling:2016}}] \label{associative pullbacks}
    The category of totally ordered $G$-sets admits choices of coproducts, products, and pullbacks which are associative.
\end{lemma}

We now recall the construction, due to \cite{GuillouMay:2011} and \cite{malkiewich/merling:2016}, of a $2$-category $\mathbb{B}^G$ which serves as a domain for our categorical Mackey functors. The objects of $\mathbb{B}^G$ are the subgroups of $G$.
For any two subgroups $K,H\leq G$, let $S_{H,K}$ denote the category of finite $G$-sets over $G/H\times G/ K$. The objects of $S_{H,K}$ can be thought of as spans of $G$-equivariant morphisms
\[
     G/H \xleftarrow{r} A \xrightarrow{t} G/K .
\]
The object $A$ in the span above is called \textit{the middle} $G$-set of the span.  A morphism of spans consists of a $G$-map between the middle $G$-sets such that the obvious diagram commutes.  This category of spans admits a coproduct, given by the coproduct of middle $G$-sets. We set the morphism category to be $\mathbb{B}^G(H,K) = S_{H,K}$.

We need to show that $\mathbb{B}^G$ is actually a strict $2$-category.  Given two spans $[G/H \xleftarrow{f} A \xrightarrow{g} G/K]$ and $[G/K \xleftarrow{h} B\xrightarrow{k} G/L]$, we ``compose'' them by taking a pullback in $G$-sets
\begin{equation}\label{burnsideComposition}
			\begin{tikzcd}
				& & A\times_{G/K} B \ar[bend right, "f\circ \pi_1"']{ddll} \ar["\pi_1"]{dl}\ar[bend left,"k\circ \pi_2"]{ddrr} \ar["\pi_2"']{dr} & & \\
				& A \ar["g"']{dr} \ar["f"]{dl} & & B \ar["k"']{dr} \ar["h"]{dl} & \\
				G/H & & G/K & & G/L	
			\end{tikzcd}.
    \end{equation}
For $L,K,H\leq G$, this composition gives a pairing
\begin{equation}\label{unbased pairing}
    \ast \colon S_{H,K}\times S_{K,L}\to S_{H,L}
\end{equation}
by the composition of spans as in \eqref{burnsideComposition} with pullbacks chosen as in \cref{associative pullbacks}. Note that the $\ast$ map is associative on the nose with these choices of pullback.  The unit in $S_{H,H}$ should be the identity span on $G/H$, though as noted in \cite[Remark 4.1]{malkiewich/merling:2016}  these elements only serve as right units for $\ast$. To address this issue, we insert an additional object $1_H$ into the category $S_{H,H}$ which is isomorphic to the identity span on $G/H$ and defines pullbacks so that this element is an honest unit. With these choices, $\mathbb{B}^G$ is a strict $2$-category.

In \cref{MMProp4.6}, we claim that a spectral Mackey functor can be constructed from a strict $2$-functor $F\colon \mathbb{B}^G\to \mathrm{Wald}$.  This proposition makes use of a result of Malkiewich--Merling \cite[Proposition 4.6]{malkiewich/merling:2016}, which gives a concrete list of data needed to build spectral Mackey functors.

Making use of Malkiewich--Merling's result requires us to introduce a slight variation on the morphism categories $S_{H,K}$ in order to make them Waldhausen categories. For $H,K\leq G$, let $S_{H,K}^+$ be the category of finite retractive $G$-sets over $G/H\times G/K$. The main difference between $S_{H,K}$ and $S_{H,K}^+$ is that the latter category has a zero object. In particular, $S_{H,K}^+$ is a Waldhausen category with weak equivalences given by isomorphism and cofibrations given by injections.

Malkiewich--Merling define a horizontal composition 
\[
    \bullet\colon S_{H,K}^+\times S_{K,L}^+\to S_{H,L}^+
\]
by the rule
\[
    X\bullet Y = (X_{\circ}\ast Y_{\circ})_+
\]
where $\ast$ is the pairing \eqref{unbased pairing}.

We have maps between the objects of these two categories,\[
\begin{tikzcd}
    {\rm Ob} S_{H,K} \ar[r, bend right=15, swap, "(-)_+"] & {\rm Ob} S_{H,K}^+ \ar[l, bend right=15, swap, "(-)_{\circ}"]
\end{tikzcd}
\] where $Y_+ = Y\amalg (G/H\times G/K)$ for $Y\in S_{H,K}$ and $X_{\circ} = (X\setminus \mathrm{Im}(s),r)$ for $(X,r,s)\in S_{H,K}^+$. Although $(-)_+$ extends to a functor, $(-)_{\circ}$ is just a function on objects -- it cannot be extended to morphisms.
\begin{remark}\label{remark: map from SHK to K theory of SHK plus}
    We record the following observation for later reference.  Since $S_{H,K}^+$ has split cofibrations, in the sense of \cite[Section 1.8]{waldhausen:1983}, there is a group completion map $|S_{H,K}^+|\to \Omega|w_{\bullet}S_{\bullet}(S_{H,K})^+| = \Omega^{\infty}K(S_{H,K}^+)$.  This yields a map
    \[
        |S_{H,K}|\xrightarrow{(-)_+} |S_{H,K}^+|\to \Omega|w_{\bullet}S_{\bullet}(S_{H,K}^+)|
    \]
    where the second map is group completion.  
\end{remark}

In \cite{malkiewich/merling:2016}, Malkiewich--Merling work with a spectrally-enriched category $\mathcal{G}\mathcal{B}_{\mathrm{Wald}}$ which has subgroups $H\leq G$ as objects and morphism spectra given by $K(S_{H,K}^+)$.  A theorem of Guillou--May \cite{GuillouMay:2011}, together with some technical results of Bohmann-Osorno \cite{BohmannOsorno2020}, allows Malkiewich--Merling to construct genuine $G$-spectra from the data of spectrally enriched functors $\mathcal{G}\mathcal{B}_{\mathrm{Wald}}\to \mathrm{Sp}$.  For more details see \cite[Theorem 4.5]{malkiewich/merling:2016} and the preceding discussion. Their result, stated below for reference, provides sufficient categorical data to produce spectrally-enriched functor $\mathcal{GB}_{\mathrm{Wald}}\to\mathrm{Sp}$ and thus a genuine $G$-spectrum.

\begin{proposition}[{\cite[Proposition 4.6]{malkiewich/merling:2016}}]\label{Appendix B 1-7}
    The following data determines a spectral Mackey functor:
    \begin{enumerate}
        \item for every $H\leq G$, a Waldhausen category $R^H$,
        \item for every $S\in S_{H,K}^+$, an exact functor $R^{S}\colon R^H\to R^{K}$,
        \item for every morphism $f\colon S\to S'$ in $S_{H,K}^+$, there is a natural transformation $R^f\colon R^S\Rightarrow R^{S'}$,
        \item for every object $A\in R^H$ the rules $S\mapsto R^S(A)$ and $f\mapsto R^f$ define a functor $S_{H,K}^+\to R^K$,
        \item there are isomorphisms $R^{(\emptyset_+)}(A) \cong 0_K$ and $R^{(S_{\circ}\amalg T_{\circ})_+}(A)\cong R^{S}(A)\vee R^T(A)$ where $\vee$ is the coproduct in $R^K$,
        \item the functor $R^{(1_{H})_+}$ is the identity on $R^H$ for all $H$,
        \item for every $S\in S_{H,K}^+$ and $T\in S_{H,L}^+$ we have an equality of functors
        \[
            R^{S\bullet T} = R^TR^{S},
        \]
        and if $f\colon S\to S'$ and $g\colon T\to T'$ are morphisms in $S_{H,K}^+$ and $S_{K,L}^+$ respectively then the diagram
        \begin{equation}\label{1-7 coherence diagram}
            \begin{tikzcd}[row sep =large, column sep =large]
                R^TR^S \ar[r,"R^T\cdot R^f"] \ar[d, swap, "R^g\cdot R^S"] & R^TR^{S'} \ar[d,"R^g\cdot R^{S'}"]\\
                R^{T'}R^S \ar[r, swap, "R^{T'}\cdot R^f"] & R^{T'}R^{S'}
            \end{tikzcd}
        \end{equation}
        commutes in the category of functors from $R^H$ to $R^L$.
    \end{enumerate}
\end{proposition}


To prove \cref{MMProp4.6}, we need to show that (1)--(7) above can be obtained from a strict $2$-functor $F\colon  \mathbb{B}^G\to \mathrm{Wald}$ satisfying the condition 
 \begin{itemize}
    \item[($\star$)] for any $A\in F(H)$ and any $S,T\in S_{H,K}$ we have $F(S\amalg T)(A)\cong F(S)(A)\vee F(T)(A)$ and $F(\emptyset)(A) \cong 0_K$ where $\vee$ is the coproduct in $F(K)$ and $0_K\in F(K)$ is the zero object. 
    \end{itemize}
which is essentially the same as (5) above.  

Let us say a bit more about condition $(\star)$. The inclusions $S,T\to S\coprod T$ in $S_{H,K}$ give natural transformations $F(S)\vee F(T)\Rightarrow F(S\coprod T)$ where the coproduct on the right is given by taking coproduct objectwise. The first half of condition $(\star)$ is that this particular map is an isomorphism.  Note that if $F$ is a strict $2$-functor that satisfies $(\star)$, then for any coproduct $A\amalg B \amalg C$ the two canonical choices of isomorphism $F(A)\vee F(B)\vee F(C)\cong F(A\amalg B\amalg C)$ are the same. 

Because Waldhausen categories have a zero object, there is always a \textit{collapse} map $F(S)(A)\vee F(T)(A)\to F(S)$ given by the coproduct of the identity and the zero map.  The collapse maps assemble into a natural transformation $F(S)\vee F(T)\to F(S)$ and thus, by $(\star)$, a transformation $F(S\amalg T)\to F(S)$. It follows from the discussion in the previous paragraph that the projection $F(A\amalg B\amalg C)\to F(A)$ and the composite $F(A\amalg B\amalg C)\to F(A\amalg B)\to F(A)$ are the same natural transformation; this fact is used in the (4) below.

\begin{proof}[Proof of \cref{MMProp4.6}]
    We address (1)--(7) of \cref{Appendix B 1-7} in order:
    \begin{enumerate}
        \item Define $R^H =F(H)$.
        \item For $S\in S_{H,K}^+$, define $R^S = F(S_{\circ})$.
        \item For $f\colon S\to T$ in $S_{H,K}^+$, let $U^f = f^{-1}(G/H\times G/K)$ so there is decomposition
        \[  
            S_{\circ}  = U^f_{\circ}\amalg (S\setminus U^f)
        \]
        and $f$ restricts to a map $\overline{f}\colon (S\setminus U^f)\to T_{\circ}$.  Define $R^{f}\colon R^S\to R^{T}$ to be the composite
        \[
            R^{S} = F(S_{\circ}) \cong F(U^f_{\circ})\vee F(S\setminus U^f) \to F(S\setminus U^f)\xrightarrow{F(\overline{f})} F(T_{\circ}) = R^{T}
        \]
        where the unlabeled arrow is the projection and the unlabeled isomorphism is provided by $(\star)$.
        \item We need to show that for $f\colon S\to T$ and $g\colon T\to W$ that there is an equality $R^{g\circ f} = R^g\circ R^f$.  Note that we have $U^f\subset U^{(g\circ f)}$.  We write $V = U^{(g\circ f)}\setminus U^{f}$ so that $S\setminus U^{f} = V\amalg S\setminus U^{(g\circ f)}$ Consider the diagram
            \[
            \begin{tikzcd}
                F(S_{\circ})\ar[r] \ar[ddd,"\cong"']  & F(S\setminus U^f) \ar[r,"F(\overline{f})"] \ar[d,"\cong"]& F(T_{\circ})  \ar[d,"\cong"] \\
                &  F(V)\vee F(S\setminus U^{(g\circ f)})  \ar[r,"F(\overline{f})"]\ar[d] & F(U^g_{\circ})\vee F(T\setminus U^g) \ar[d] \\
               &  F(S\setminus U^{(g\circ f)}) \ar[r,"F(\overline{f})"] \ar[d,equal]& F(T\setminus U^g) \ar[d,"F(\overline{g})"]\\
                F(U^{(g\circ f)}_{\circ})\vee F(S\setminus U^{(g\circ f)}) \ar[r]& F(S\setminus U^{(g\circ f)})  \ar[r,"F(\overline{g\circ f})"]  & F(W_{\circ})
            \end{tikzcd}
            \]
            where every unlabeled arrow is a collapse.  The left rectangle commutes because both ways of going around represent collapse maps associated to the same coproduct decomposition of $S_{0}$; see the discussion preceding this proof.   The top-right square commutes because $F$ is a strict $2$-functor. The middle-right square commutes because the collapse maps are natural.  The bottom-right square commutes because on $S\setminus U^{(g\circ f)}$ we have $\overline{g}\circ \overline{f} = \overline{g\circ f}$ and $F$ is a strict $2$-functor.  Since the top-right composite is $R^g\circ R^f$ and the left-bottom composite is $R^{(g\circ f)}$ we are done.
        \item This is immediate from $(\star)$.
        \item This follows from the fact that $F$ is a strict $2$-functor and $R^{(1_{H})_+} = F((1^+_H)_{\circ}) = F(1_{H})$.
        \item For $S\in S_{H,K}^+$ and $T\in S_{H,L}^+$ we have
        \[
            R^{S\bullet T} = F((S\bullet T)_{\circ}) = F(S_{\circ}\ast T_{\circ}) =F(T_{\circ})F(S_{\circ}) = R^TR^S
        \]
        proving the first part. That the diagram \eqref{1-7 coherence diagram} commutes follows from $F$ being a strict $2$-functor and coherence for pasting diagrams in the $2$-category of categories.
    \end{enumerate}
\end{proof}

\section{Constructing examples of spectral Mackey functors and morphisms}

The previous appendix lays out a $2$-categorical framework for constructing spectral Mackey functors via $K$-theory.  In this appendix we address some outstanding technical problems in order to facilitate the construction of examples and, most importantly for the body of the paper, morphisms.  The first issue we address is how one actually constructs a morphism of spectral Mackey functors from this framework.  At the most basic level this is straightforward: given two pseudofunctors $F,F'\colon \mathbb{B}^G\to \Wald$ which satisfy $(\star)$, one can check that any pseudonatural transformation $\alpha\colon F\Rightarrow F'$ which is comprised of exact functors will (after some strictification) induce a map on the associated spectral Mackey functors.  In the main body of the paper we would like to apply this to a natural transformation which is constructed out of the linearization maps.

There are two issues with this approach.  First, the linearization maps are not exact, since they do not preserve coproducts.  They are, however, \textit{weakly exact}, in the sense that they preserve coproducts up to weak equivalence.  The fix for this issue is to use a construction of algebraic $K$-theory which is functorial in weakly exact functors.  The construction we use is the $S_{\bullet}'$-construction developed by Blumberg--Mandell in \cite{BM:08}. The first part of this appendix recalls this construction, and discusses the multifunctoriality of this construction which is necessarily for it to be used to construct a genuine $G$-spectra as in \cite{malkiewich/merling:2016}.

The second issue with using linearization to construct a map of spectral Mackey functors is that it is not actually a pseudonatural transformation.  Rather, it is a lax transformation, meaning that the structure $2$-cells are not all invertible.  In particular, it only commutes with the transfer morphisms up to a quasi-isomorphism.  Morally, this is not a problem; indeed, Waldhausen observes in \cite[\S 1.3]{waldhausen:1983} that such a natural transformation will produce homotopy coherent data after applying $K$-theory.
Thus one expects that after applying $K$-theory there is a well defined linearization map, at least in the homotopy category.  On the other hand, making this argument precise requires some care, and occupies the second part of this appendix, \cref{subsec: appendix making maps}. 

In the last part of the appendix, \cref{subsec: appendix constructing examples}, we address the problem of how one actually constructs such a lax transformation.  One can, in principle, do this entirely by hand, however this is a rather arduous task.  Indeed, even the construction of a pseudofunctor $F\colon \mathbb{B}^G\to \Wald$ is quite involved. 
We lay out a more systematic approach to this problem, showing in \cref{Theorem: construction examples and maps} that actually much of what we need exists for entirely formal reasons.  Explicitly, we show that given sufficiently nice pseudofunctors $F,F'\colon \mathcal{O}_G^{\op}\to \Wald$ that there is a canonical way to extend these to pseudofunctors $\overline{F},\overline{F'}\colon \mathbb{B}^G\to \Wald$. Moreover, given a sufficiently nice transformation $\alpha\colon F\Rightarrow F'$ we obtain a lax transformation $\overline{\alpha}\colon \overline{F}\Rightarrow \overline{F'}$ which is good enough to induce a map on $K$-theory $G$-spectra.  We believe \cref{Theorem: construction examples and maps} constitutes the absolute minimum that one needs to check to construct examples and morphisms, although we do not make this precise. We expect this result to have applications beyond the present work. 

The remainder of this appendix is structured as follows.  In \cref{subsec: appendix S dot prime} we discuss the $S_{\bullet}'$-construction of Blumberg--Mandell. In \cref{subsec: appendix making maps} we explain how lax natural transformations can be used to construct maps of spectral Mackey functors obtained from $K$-theory.  In \cref{subsec: appendix constructing examples} we explain how to construct examples of pseudofunctors $F\colon \mathbb{B}^G\to \Wald$ as well as lax natural transformation between such pseudofunctors.  \cref{subsec: appendix constructing examples}, especially \cref{Theorem: construction examples and maps}, contains the essential information which is used in the main body of the paper.  We note that the promised proof of \cref{spectral Mackey functor} is found in \cref{corollary: proof of proposition from main body}.

\subsection{The $S_{\bullet}'$ construction}\label{subsec: appendix S dot prime}
In \cite[\S 2]{BM:08}, Blumberg--Mandell define a different model for Waldhausen $K$-theory using an $S'_\bullet$-construction, which essentially replaces the pushout squares in the $S_\bullet$-construction with more general \textit{weak pushout} squares. The $S'_\bullet$-construction is defined for Waldhausen categories with \textit{functorial factorization of weak cofibrations} (FFWC) which are also saturated, i.e. whose weak equivalences satisfy the 2-of-3 property. 

\begin{definition}[{\cite[Definition 2.2]{BM:08}}]
    A morphism $f\colon A\to B$ in a Waldhausen category $\cat C$ is a \textit{weak cofibration} if there is a zig-zag of weak equivalences (in the arrow category of $\cat C$) from $f$ to a cofibration. A Waldhausen category has \textit{functorial factorization of weak cofibrations} (FFWC) if every weak cofibration can be factored (in the arrow category) as a cofibration followed by a weak equivalence.
\end{definition}

The condition FFWC is a very mild assumption, and every Waldhausen category discussed in this paper has FFWC, as remarked below. Note also that every Waldhausen category in this paper is saturated.

\begin{example}
    If the cofibrations are a Waldhausen category $\cat{C}$ are monics then every weak cofibration is an actual cofibration and thus $\cat{C}$ has FFWC trivially.  Alternatively, observe that if $\cat{C}$ is a full subcategory of a model category $\cat{D}$ and the cofibrations and weak equivalences of $\cat{C}$ are those of $\cat{D}$ then the factorization axioms for a model category imply that $\cat{C}$ has FFWC.  Every Waldhausen category in this paper falls into one of these two categories, thus they all have FFWC.
\end{example}

\begin{definition}[cf. {\cite[Def. 2.7]{BM:08}}]\label{defn:S'}
Let $S'_n\cat C\subseteq \Fun(\mathrm{ar}[n], \cat C)$ be the full subcategory on those functors $A$ such that:\begin{itemize}
        \item[(i)] the initial map $*\cof A_{ii}$ is a weak equivalence for all $i\in [n]$, 
        \item[(ii)] the map $A_{ij}\cof A_{ik}$ is a weak cofibration for all $i\leq j\leq k$,
        \item[(iii)] for every $i\leq j\leq k$, the square\[
        \begin{tikzcd}
            A_{ij} \ar[r] \ar[d] & A_{ik} \ar[d] \\
            A_{jj}\ar[r] & A_{jk}
        \end{tikzcd}
        \] is a weak pushout, meaning it is weakly equivalent (via a zig-zag of morphisms of diagrams) to a pushout square where one of the parallel sets of arrows are cofibrations. 
    \end{itemize}
    Just as with the $S_\bullet$-construction, the categories $S'_n\cat C$ assemble into a simplicial Waldhausen category, and we let $wS'_\bullet \cat C$ denote the simplicial subcategory of weak equivalences; see \cite[Definition 2.7]{BM:08} for details. 
\end{definition}

There is an evident simplicial inclusion $S_\bullet \cat C\to S'_\bullet \cat C$, and if $\cat C$ is a saturated Waldhausen category with FFWC, then the inclusion $wS_\bullet \cat C\to wS'_\bullet \cat C$ is a weak equivalence {\cite[Theorem 2.9]{BM:08}}. The $S_{\bullet}'$-construction can be iterated, just like the $S_{\bullet}$-construction, and the resulting symmetric spectrum $n\mapsto \Omega\lvert N(w_{\bullet}S_{\bullet}^{(n)})\rvert$ is an $\Omega$-spectrum which is stably equivalent to the usual $K$-theory spectrum.

A benefit of the $S_\bullet'$-construction, particularly for our purposes, is that it is functorial in a larger class of functors than exact ones. 

\begin{definition}
    A functor $F\colon \cat C\to \cat D$ of Waldhausen categories is \textit{weakly exact} if it preserves the zero object, cofibrations, and weak equivalences, and whenever \[
    \begin{tikzcd}
        A \ar[r, >->] \ar[d] & B \ar[d] \\
        C \ar[r, >->] & D
    \end{tikzcd}
    \] is a pushout square in $\cat C$, then the induced map $F(C)\cup_{F(A)} F(B) \to F(D)$ is a weak equivalence. 
\end{definition}

Observe that such a functor will preserve weak pushouts (and hence induce a map after $S_\bullet'$), since whenever there is a zig-zag\[
        \begin{tikzcd}
            A_{ij} \ar[r] \ar[d] & A_{ik} \ar[d] \\
            A_{jj}\ar[r] & A_{jk}
        \end{tikzcd} \xleftrightarrow{\sim} \begin{tikzcd}
        A \ar[r, >->] \ar[d] & B \ar[d] \\
        C \ar[r, >->] & D
    \end{tikzcd}
        \] of weak equivalences of diagrams in $\cat C$, so that the latter is a pushout, then there is a zig-zag of weak equivalences \[
        \begin{tikzcd}
            F(A_{ij}) \ar[r] \ar[d] & F(A_{ik}) \ar[d] \\
           F( A_{jj})\ar[r] & F(A_{jk})
        \end{tikzcd} \xleftrightarrow{\sim}  \begin{tikzcd}
        F(A) \ar[r, >->] \ar[d] & F(B) \ar[d] \\
        F(C) \ar[r, >->] & F(D)
    \end{tikzcd} \xleftarrow{\sim} \begin{tikzcd}
        F(A) \ar[r, >->] \ar[d] & F(B) \ar[d] \\
        F(C) \ar[r, >->] & F(B)\cup_{F(A)} F(C)
    \end{tikzcd}
        \] of diagrams in $\cat D$.

\begin{example}
    For a $G$-space $X$, the linearization functor $L^G\colon R^G_{\fd}(X) \to \mathcal P^G(\underline{\ZZ}[\pi_1(X)])$ is weakly exact; this is precisely the content of \cref{prop:L induces map on fx pts}.  
\end{example}

Recall from \cite[\S 1.3]{waldhausen:1983} that a \textit{weak equivalence} of exact functors is a natural transformation $\eta\colon F\Rightarrow F'$ so that every component of $\eta$ is a weak equivalence in the target Waldhausen category; such a notion also makes sense for weakly exact functors. As in \cite[Proposition 1.3.1]{waldhausen:1983}, such natural transformations induce homotopies upon taking $K$-theory.

\begin{proposition}
    A weak equivalence of weakly exact functors induces a natural homotopy between $wS'_\bullet F$ and $wS'_\bullet F'$.
\end{proposition}

We will make use of this proposition in \cref{subsec: appendix making maps}. Before we do so, we need to ensure that spectral Mackey functors may be constructed using the $S_\bullet'$-construction equally well as they are with the $S_\bullet$-construction.
The translation from Mackey functors in Waldhausen categories to spectral Mackey functors via $K$-theory is achieved through the multifunctoriality of the $S_{\bullet}$-construction, and we now explain how the same arguments can be made using $S_{\bullet}'$. For more detailed definitions and the relevant arguments for $S_\bullet$, we refer the reader to \cite{ElmendorfMandell} and \cite{Zak}. 

\begin{definition}
    Let $M_1,\dots, M_n$ and $N$ be saturated Waldhausen categories with FWCC.  A functor
    \[
        F\colon M_1\times \dots \times M_n\to N
    \]
    is \textit{weakly $n$-exact} if \begin{enumerate}
        \item $F(M_1,\dots,M_n)$ is the zero object if any $M_i$ is the zero object,
        \item $F$ is weakly exact in each variable separately,
        \item Given cofibrations $M_{10}\cof M_{11}, \dots, M_{n0}\cof M_{n1}$, the functor \[C(i_1, \dots, i_n) = F(M_{1i_1}, \dots, M_{ni_n})\colon [1]\times \dots \times [1]\to \cat N\]
        is cubically cofibrant in the sense of \cite[Definition 2.1]{BM:2011}. 
    \end{enumerate}
    The cubical cofibrancy condition in this case amounts to checking that for every $1\leq j<k \leq n$, we have a square\[
    \begin{tikzcd}
        F(M_{10}, \dots, M_{j0}, \dots, M_{k0}, \dots, M_{n0}) \ar[r] \ar[d] & F(M_{10}, \dots, M_{j1}, \dots, M_{k0}, \dots, M_{n0})\ar[d] \\
        F(M_{10}, \dots, M_{j0}, \dots, M_{k1}, \dots, M_{n0}) \ar[r] & F(M_{10}, \dots, M_{j1}, \dots, M_{k1}, \dots, M_{n0}) 
    \end{tikzcd}
    \] and we ask for the map from the pushout of the span to the lower-right entry of the diagram to be a cofibration. 
\end{definition}
Hence $F$ is weakly $n$-exact functor if it satisfies the conditions of being $n$-exact, except that it is only weakly exact in every variable. It is shown in \cite{Zak} that the collection of Waldhausen categories and $n$-exact functors forms a closed symmetric multicategory, and the same arguments work to show that Waldhausen categories and weakly $n$-exact functors form a closed symmetric multicategory and that the assignment $M\mapsto K'(M):=\Loop \abs{wS'_\bullet M}$ assembles into a symmetric multifunctor $K'\colon \Wald'\to \Sp$.

Let $\Wald^{ex}\subset \Wald^{wex}$ denote the sub-multicategory of Waldhausen categories and multiexact functors.

\begin{proposition}\label{proposition: multinatural transformation}
    There is multinatural transformation $\Phi$ filling the triangle of multifunctors
\[\begin{tikzcd}
	\Wald^{ex} && {\Wald^{wex}} \\
	& \Sp
	\arrow[""{name=0, anchor=center, inner sep=0}, "\subset", hook, from=1-1, to=1-3]
	\arrow["K"', from=1-1, to=2-2]
	\arrow["{K'}", from=1-3, to=2-2]
	\arrow["\phi", shorten <=6pt, shorten >=5pt, Rightarrow, from=2-2, to=0]
\end{tikzcd}.\]
Moreover, $\phi_M\colon K(M)\to K'(M)$ is the stable equivalence induced by \cite[Theorem 2.9]{BM:08}
\end{proposition}
\begin{proof}
    Let $C_1,\dots, C_n$ and $D$ be saturated Waldhausen categories with FWCC.  To check that $\phi$ is a multinatural transformation (see e.g. \cite[Definition 4.5]{BohmannOsorno2020}), we need to check that the diagram
\[\begin{tikzcd}
	{\Wald^{\ex}(C_1,\dots,C_n;D)} & {\Sp(K'(C_1),\dots,K'(C_n);K'(D))} \\
	{\Sp(K(C_1),\dots,K(C_n);K(D))} & {\Sp(K(C_1),\dots,K(C_n);K'(D))}
	\arrow["{K'}", from=1-1, to=1-2]
	\arrow["K", from=1-1, to=2-1]
	\arrow["{(\phi_{C_1},\dots,\phi_{C_n})^*}", from=1-2, to=2-2]
	\arrow["{(\phi_{D})_*}", from=2-1, to=2-2]
\end{tikzcd}\]
commutes. Here, $\Sp$ denotes the multicategory of symmetric spectra and so by definition a multimorphism in $\Sp(X_1,\dots,X_n;Y)$ consists of a map of symmetric spectra $X_1\wedge\dots\wedge X_n\to Y$. Given a weakly exact functor, $F\colon C_1\times\dots\times C_n\to D$, we need to check that the diagram of symmetric spectra
\[\begin{tikzcd}
	{K(C_1)\wedge\dots\wedge K(C_n)} && {K(D)} \\
	{K'(C_1)\wedge\dots\wedge K'(C_n)} && {K'(D)}
	\arrow["{K(F)}", from=1-1, to=1-3]
	\arrow["{\phi_{C_1}\wedge\dots\wedge\phi_{C_n}}"', from=1-1, to=2-1]
	\arrow["{\phi_D}", from=1-3, to=2-3]
	\arrow["{K'(F)}", from=2-1, to=2-3]
\end{tikzcd}\]
commutes.  But this follows from the observation that the natural transformation $\phi$ is induced by an inclusion of simplicial categories $S_{\bullet}(C)\to S_{\bullet}'(C)$ for any Waldhausen category $C$.
\end{proof}

Let $\mathcal{GB}_{\Wald}' = K'_{\bullet}(\mathbb{B}^G_+)$ be the $\Sp$-enriched category with objects subgroups $H\leq G$ and morphism spectra $\mathcal{GB}_{\Wald}'= K'(S_{H,K})$.  The multinatural transformation $\phi$ induces a map of $\Sp$-categories $\phi_*\colon \mathcal{GB}_{\Wald}\to \mathcal{GB}_{\Wald}'$ which is a stable equivalence on mapping spectra.

\begin{corollary}[{\cite[Proposition 2.4]{MGM:EnrichedModelCats}}]
    The map $\phi_*$ induces is a Quillen equivalence between the categories of $\Sp$-enriched functors $\Fun_{\Sp}(\mathcal{GB}_{\Wald},\Sp)$ and $\Fun_{\Sp}(\mathcal{GB}_{\Wald}',\Sp)$.  
\end{corollary}

Thus we may use either $K$ or $K'$ equally well to produce spectral Mackey functors.  In particular, Malkiewich and Merling's construction of equivariant $A$-theory can be achieved equally well with the $S_{\bullet}'$ construction in place of the $S_{\bullet}$-construction. 

\subsection{Morphisms of spectral Mackey functors}\label{subsec: appendix making maps}

This section addresses the construction of morphisms between spectral Mackey functors obtained from pseudofunctors $F\colon \mathbb{B}^G\to \Wald$ via \cref{MMProp4.6}.  Let $\Wald$ denote the $2$-category of Waldhausen categories with $\Wald_2(\cat C, \cat D)$ the category of exact functors and natural isomorphisms. 
The linearization functors $L = \{L^H\}_{H\leq G}$ will not assemble into a \textit{pseudo}-natural transformation, essentially because linearization is only weakly exact. For instance, $L$ will not preserve coproducts on the nose as in general there is only a quasi-isomorphism between $C_*(Y_1\cup_X Y_2, X)$ and $C_*(Y_1,X)\oplus C_*(Y_2, X)$. However, we claim that $L$ forms what might be called a ``weakly natural transformation,'' in the sense that it is a lax natural transformation whose $2$-cell components are weak equivalences. 

More precisely, let $\Wald^{wex}$ denote the $2$-category with the same objects as $\Wald$, but with $\Wald^{wex}(\cat{C},\cat{D})$ the category of weakly exact functors and \textit{natural weak equivalences} of functors, which (as in \cite[Section 1.3]{waldhausen:1983}) are natural transformations $\eta\colon F\Rightarrow F'$ such that $\eta_x\colon Fx\xrightarrow{\sim} F'x$ is a weak equivalence for all objects $x$. We can also consider $\Wald^{wex}$ as a simplicially enriched category by doing a change of enrichment along the nerve; we denote the resulting simplicial category by $\Wald^{wex}_{\Delta}$. Note that we have an inclusion of simplicial categories $\Wald_{\Delta}\subseteq \Wald^{wex}_{\Delta}$.

\begin{proposition}[{cf. \cite[Proposition 1.3.1]{waldhausen:1983}}]\label{proposition: K theory is simplically enriched}
    The $S_{\bullet}'$ construction induces a map of simplicial categories
    \[K_{\Delta}\colon \cat W^{wex}_{\Delta} \to \Top,\] where $\Top(X,Y) = \Sing_*(\Map(X,Y))$.
\end{proposition}\begin{proof}
   Waldhausen observes that a natural weak equivalence $\eta\colon F\Rightarrow F'$ induces a natural homotopy $K(\eta)$ between $K(F)$ and $K(F')$.  The same observation holds for the $S_{\bullet}'$ construction in place of the $S_{\bullet}$-construction.
\end{proof}

\begin{definition}
    Let $\cat C$ be a $2$-category. A \textit{weakly natural transformation} $\eta$ between two pseudofunctors $F, F'\colon \cat C\to \cat W^{wex}$ is a lax natural transformation.
\end{definition}

Explicitly, this means that for each $1$-cell $f\colon x\to y$ in $\cat C$, the naturality diagram for $f$ commutes ``up to weak equivalence,'' i.e. the $2$-cell $\alpha_f\colon G(f)\circ \alpha_x \Rightarrow \alpha_y \circ F(f)$ is a weak equivalence of exact functors.  Our goal is to prove the following result.

\begin{theorem}\label{Theorem: weak transformations give maps}
    A weakly natural transformation between pseudofunctors $\mathbb{B}^G\to \cat W^{wex}$ induces a map of the associated genuine $G$-spectra in the equivariant stable homotopy category.
\end{theorem}

This proof is most easily effectuated using the language of $\infty$-categories (by which we mean quasi-categories).  We briefly recall the $\infty$-categorical setup of spectral Mackey functor due to Barwick \cite{BarwickOne}. We write $\Span(\Set^G)_{(2,1)}$ for the bicategory whose objects are finite $G$-sets, morphisms are spans of finite $G$-sets, and $2$-cells are \emph{isomorphisms} of spans.  There is an evident inclusion $\Span(\Set^G)_{(2,1)}\to \Span(\Set^G)$ where the target has the same objects and $1$-cells but additional $2$-cells. Since $\Span(\Set^G)_{(2,1)}$ is a $(2,1)$-category we can obtain an $\infty$-category by applying the \emph{Duskin nerve} functor $N^D_{*}\colon \Cat_{2}\to \mathrm{sSet}$.  We summarize the relevant properties of the Duskin nerve.
\begin{theorem}[{\cite{Duskin,BullejosFaroBlanco}}]\label{theorem: Duskin facts}
    The Duskin nerve functor has the following properties:
    \begin{itemize}
        \item the objects and morphisms in in $N^D_*(\cat{C})$ are those of $\cat{C}$,
        \item if $\cat{C}$ is a $(2,1)$-category then $N^D_*(\cat{C})$ is a quasi-category.  Moreover, $N^D_*(\cat{C})$ is equivalent to the quasi-category obtained by first changing the enrichment of $\cat{C}$ to Kan complexes using the ordinary nerve and then applying the homotopy coherent nerve,
        \item the Duskin nerve carries lax transformations to simplicial homotopies.
    \end{itemize}
\end{theorem}
\begin{definition}
    The effective Burnside category for a finite group $G$ is the Duskin nerve $\mathcal{A}_{\mathrm{eff}}^G:= N^D_*(\Span(\Set^G)_{(2,1)})$.  
\end{definition}
\begin{remark}\label{remark: local group completion of the effective burnside category}
    For a $(2,1)$-category $\cat{C}$ let us write $N_{\Delta}(\cat{C})$ for the simplicial category with morphism simplicial sets given by the nerve of the morphism categories. Observe that the morphism sets in $N_{\Delta}(\Span(\Set^G)_{(2,1)})(X,Y)$ are precisely $N_{\bullet}(S_{X,Y}^{\cong})$ where $S_{X,Y}$ are as in \cref{sec:tech pf app} and the superscript $\cong$ denotes taking internal groupoid.
    
     For a spectral category $\cat{D}$, let us write $\Omega^{\infty}\cat{D}$ for the simplicial category with morphism simplicial sets given by applying $\mathrm{Sing}\ \Omega^{\infty}$ to the mapping spectra.  Observe that the morphism simplicial sets $\Omega^{\infty}(\mathcal{GB}_{\Wald})(X,Y)$ are precisely $\Sing\Omega^{\infty}K(S_{H,K^+})$.  Finally, note that taking adjuncts of the maps constructed in \cref{remark: map from SHK to K theory of SHK plus} yields a simplicial functor
     \[
        \Span_{(2,1)}(\Set^G)_{\Delta}\to\Omega^{\infty}(\mathcal{GB}_{\Wald}).
     \]
     Since both of these simplicial categories are locally Kan complexes, applying the simplicial nerve $N^{\mathrm{ch}}$ yields a functor of $\infty$-categories
     \begin{equation}\label{equation: local group completion}
        \mathcal{A}^G_{\mathrm{eff}}\to N^{\mathrm{ch}}\Omega^{\infty}(\mathcal{GB}_{\Wald})
     \end{equation}
     where the the source is identified with the effective Burnside category by the second bullet point of \cref{theorem: Duskin facts}. This map is the identity on objects and on morphism spaces is given by group completion. This map is a model for the ``local group completion'' functor described in \cite[3.8]{BarwickOne}.
\end{remark}

The effective Burnside category has products given by disjoint union of finite $G$-sets. If $\cat{C}$ is any additive $\infty$-category then we write $\mathrm{Mack}_{\infty}(\cat{C})$ for the $\infty$-category of product preserving functors $\mathcal{A}^G_{\mathrm{eff}}\to \cat{C}$. When $\cat{C} = \Sp$ is the $\infty$-category of spectra we will call an object in $\mathrm{Mack}_{\infty}(\Sp)$ a spectral Mackey functor.

\begin{theorem}[{\cite{BarwickOne}, \cite[Appendix A]{CMNN}}]
    There is an equivalence of $\infty$-categories between $\mathrm{Mack}_{\infty}(\Sp)$ and $\Sp^G$, the $\infty$-category of genuine $G$-spectra.
\end{theorem}

\begin{lemma}
    Any pseudofunctor $F\colon \mathbb{B}^G\to \Wald^{\mathrm{wex}}$ which satisfies condition $(\star)$ of \cref{MMProp4.6} determines a map of simplicial sets
    \[
        F_{\infty}\colon \mathcal{A}^G_{\mathrm{eff}}\to N^D_*(\Wald^{wex})
    \]
    which, on the objects $G/H$ and morphisms $[G/H\leftarrow A\rightarrow G/K]$, is given precisely by the data of $F$.
\end{lemma}
\begin{proof}
    We show in \cref{lem:BG extend to Span FinG} below that $F$ determines, an functor $\overline{F}\colon \Span(\Set^G)\to \Wald^{wex}$ which agrees with $F$ on the objects $G/H$ and morphisms between them.   Restricting $\overline{F}$ along the inclusion $\Span(\Set^G)_{(2,1)}\subset \Span(\Set^G)$ and then applying the Duskin nerve yields the result.
\end{proof}

We note that $N^D_*(\Wald^{wex})$ is not an $\infty$-category because it has non-invertible $2$-simplices.  On the other hand, notice that if we apply the simplicial (homotopy coherent) nerve to the simplicial functor of \cref{proposition: K theory is simplically enriched} we obtain a map of simplicial sets
\[
N_{\Delta}(K_{\Delta})\colon N_{\Delta}(\Wald^{wex}_{\Delta}) \to N_{\Delta}(\Top).
\]
By \cref{theorem: Duskin facts} we have $N_{\Delta}( \Wald^{wex}_{\Delta}) \simeq N_*^D(\Wald^{wex})$, and $N_{\Delta}(\Top)\simeq \cat S$ is the $\infty$-category of spaces. 

\begin{lemma}
    The map $N_{\Delta}(K_{\Delta})$ factors through $\Sp^{\geq 0} \simeq \Alg^{gp}_{\mathbb{E}_\infty}(\cat S)\subseteq \cat S$.
\end{lemma}\begin{proof}
      To see $K$-theory factors through $\Alg^{gp}_{\mathbb{E}_\infty}(\cat S)$, it suffices to know the natural homotopy $K(\eta)$ between $K(F)$ and $K(F')$ is a homotopy through maps of infinite loop spaces. This follows from the explicit delooping of $K$-theory via the iterated $S_\bullet'$-construction \cite[Proposition 1.5.3]{waldhausen:1983} and \cite{BM:08}, as $K(\eta)$ induces a homotopy on each iterate and hence defines a homotopy between maps of sequential spectra.
\end{proof}
\begin{definition}
    If $F\colon \mathbb{B}^G\to \Wald^{wex}$ is a pseudofunctor which satisfies $(\star)$ we write $K_{\infty}(F)\in \mathrm{Mack}_{\infty}(\Sp)$ for the spectral Mackey functor given by the composite
    \[
        \mathcal{A}^G_{\mathrm{eff}}\xrightarrow{F_{\infty}} N^D_*(\Wald^{wex})\xrightarrow{N_{\Delta}(K_{\Delta})} \Sp^{\geq 0}\hookrightarrow \Sp.
    \]
    Note that, by definition, $K_{\infty}(F)(G/H) = K(F(G/H))$, and the induced maps on transfers and restrictions are given by the maps obtained by applying $K$-theory to $F$.
\end{definition}

\begin{corollary}
    If $R,P\colon \mathbb{B}^G\to \Wald^{wex}$ are two pseudofunctors which satisfy condition $(\star)$ and $L\colon R\Rightarrow P$ is a weakly natural transformation then there is natural transformation of $\infty$-functors $K_{\infty}(L)\colon K_{\infty}(R)\Rightarrow K_{\infty}(P)$ which, on any orbit $G/H$, is given by the component morphism
    \[
        K(L_{G/H})\colon K(R(G/H))\to K(P(G/H)).
    \]
\end{corollary}
\begin{proof}
    This follows from the last bullet point of \cref{theorem: Duskin facts} and the definition of $K_{\infty}(F)$.
\end{proof}
To prove \cref{Theorem: weak transformations give maps}, all that remains is to identify the objects $K_{\infty}(F)$ and $K_G(F)$ in the equivariant stable homotopy category, for any pseudofunctor $F\colon \mathbb{B}^G\to \Wald^{wex}$ satisfying $(\star)$. We emphasize that this is not surprising: the two spectral Mackey functors have the same value at each finite $G$-set and the same definition of transfer and restriction.  On the other hand, the two objects $K_{\infty}(F)$ and $K_G(F)$ live in different categories of spectral Mackey functors, so a direct comparison requires some care. 

Given a spectral functor $H\colon \mathcal{GB}_{\mathrm{Wald}}\to \Sp_{\Sp}$, we may apply the functors $\Sing$ and $\Omega^{\infty}$ to the mapping spectra (as in \cref{remark: local group completion of the effective burnside category}) to obtain a simplicial functor $\Omega^{\infty}(H)\colon \Omega^{\infty}\mathcal{GB}_{\mathrm{Wald}}\to \Omega^{\infty}(\Sp)$.  Taking simplicial nerves yields an $\infty$-functor
\[
    N^{\mathrm{ch}}\Omega^{\infty}(H)\colon N^{\mathrm{ch}}\Omega^{\infty}\mathcal{GB}_{\mathrm{Wald}}\to N^{\mathrm{ch}}\Omega^{\infty}(\Sp)\simeq \Sp_{\infty}
\]
where the last object is the $\infty$-category of spectra, and by precomposing with the map \eqref{equation: local group completion} we obtain an $\infty$-categorical spectral Mackey functor $ \Phi(H)\colon \mathcal{A}^G_{\mathrm{eff}}\to \Sp_{\infty}.$ The following result completes the proof of \cref{Theorem: weak transformations give maps}.

\begin{theorem}
    There is an identification of objects $\Phi(K(F)) = K_{\infty}(F)$ in $\Mack_{\infty}(\Sp)$.
\end{theorem}
\begin{proof}
    Unwinding the definitions, one sees that $\Phi(H)$ agrees with $H$ on objects. We can also describe what happens on morphisms explicitly: for any morphism $f\colon X\to Y$ in $\mathcal{A}^G_{\mathrm{eff}}$ let $\overline{f}$ denote its image in $\Omega^{\infty}\mathcal{GB}_{\Wald}(X,Y)$ under the map \eqref{equation: local group completion}.  The based map $S^0\to \Omega^{\infty}\mathcal{GB}_{\Wald}(X,Y)$ adjoints to a map of spectra $\Sigma_{f}\colon \mathbb{S}\to \mathcal{GB}_{\Wald}(X,Y)$ and $\Phi(H)(f)$ is given by the composite
\[
  H(X) = \mathbb{S}\wedge H(X)\xrightarrow{\Sigma_f\wedge 1} \mathcal{GB}_{\Wald}(X,Y)\wedge H(X)\to H(Y)
\]
where the unlabeled map comes from the fact that $H$ is a spectral functor.  Observe that when $H = K(F)$ then it's value on any such $f$ is exactly given by applying $K$-theory to the exact functor $F(f)\colon F(X)\to F(Y)$.  Thus when $H = K(F)$ we see that $\Phi(K(F))$ and $K_{\infty}(F)$ are identically the same on objects and 1-simplicies.  The higher simplices in the Duskin nerve $\mathcal{A}^G_{\eff}$ record composable strings of morphisms in $\Span_{(2,1)}(\Set)$ and 2-cells relating their compositions, and both $\Phi(K(F))$ and $K_{\infty}(F)$ evaluate the 2-functor $F$ on the $1$- and $2$-cells and then apply multifunctorial $K$-theory.
\end{proof}

\subsection{Constructing examples}\label{subsec: appendix constructing examples}

All of the examples of pseudofunctors $\mathbb{B}^G\to \Wald$ (or, more generally, $\Wald^{wex}$) in this paper are essentially constructed by producing a $G$-category of Waldhausen categories, i.e. a functor $\mathcal{O}_G^{\op}\to \Wald$ which specifies a Waldhausen category for every $G/H$ along with restriction functors between them. In this section, we will justify this method and also show that weakly natural transformations of pseudofunctors $\mathbb{B}^G\to \Wald$ can similarly be constructed via pseudonatural transformations of pseudofunctors $\mathcal{O}_G^{\op}\to \Wald^{wex}$ that satisfy some checkable conditions. The upshot being that spectral Mackey functors, and maps between them, can be constructed using the simpler domain category $\mathcal{O}_G^{\op}$.

To pass from $\mathbb{B}^G$ to $\mathcal{O}_G^{\op}$, we pass through an intermediary bicategory $\Span(\mathbb{B}^G)$ and make use of some technical results of Dawson--Pare--Pronk \cite{DawsonParePronk}. In particular, we will exploit the fact that $\mathbb{B}^G$ is a full subcategory of the span bicategory $\Span(\Set_G)$ on the category of finite $G$-sets (as defined in \cite[\S 1]{DawsonParePronk}).

\begin{lemma}\label{lem:BG extend to Span FinG}
Restricting along the inclusion $\mathbb{B}^G\subseteq \Span(\Set_G)$ 
\[
\Fun_{\ps}^{lax}(\Span(\Set_G), \Wald^{wex}) \to \Fun_{\ps}^{lax}(\mathbb{B}^G, \Wald^{wex})
\] sends product-preserving pseudofunctors (i.e. $F$ such that $F(X\amalg Y) \cong F(X)\times F(Y)$ for all finite $G$-sets $X$ and $Y$) to pseudofunctors satisfying  $(\star)$ of \cref{MMProp4.6}. This functor is essentially surjective.
\end{lemma}
\begin{proof}
The functor is essentially surjective because $\Span(\Set^G)$ is generated under products by the objects in $\mathbb{B}^G$. The only thing left to prove is that when $F$ preserves products its restriction to $\mathbb{B}^G$ satisfies $(\star)$. Recall that the category $\Span(\Set^G)$ has finite products and coproducts which are both given by disjoint unions of finite $G$-sets. 

For any map $f\colon x\to y$ of finite $G$-sets we write 
    \[
        \rho_f = [y\xleftarrow{f} x\xrightarrow{=} x] \quad \mathrm{and} \quad \tau_{f} = [x \xleftarrow{=} x\xrightarrow{f} y]
    \]
    for the restriction and transfer along $f$, respectively.
     Suppose we are given two spans
    \[
      S = [G/H\xleftarrow{a} S\xrightarrow{b} G/K]\quad \textrm{and}\quad T= [G/H\xleftarrow{c} T\xrightarrow{d} G/K]  
    \]
    in $\mathbb{B}^G$.  In $\Span(\Set^G)$ we can factor the coproduct of these spans as
    \[
        [G/H\xleftarrow{a\amalg c} S\amalg T\xrightarrow{b\amalg d} G/K] =  \tau_{b\amalg c}\circ \rho_{a\amalg c}
    \]
 The map
       \[
         \rho_{a,c}=[G/H\xleftarrow{a\amalg c} S\amalg T\xrightarrow{=} S\amalg T]
        \]
        is the product of the restrictions along $a$ and $c$, and so if we assume that $F$ is product-preserving then $F(\rho_{a,c})$ must be the functor
        \[
            F(G/H)\xrightarrow{F(\rho_a)\times F(\rho_c)} F(S)\times F(T). 
        \]
        On the other hand, the map $b\amalg d\colon S\amalg T\to G/K$ factors as the composite
        \[
            S\amalg T\xrightarrow{\tau_b\amalg \tau_d} G/K\amalg G/K\xrightarrow{\nabla} G/K
        \]
        where $\nabla$ is the fold map.  Thus $\tau_{b\amalg d} = \tau_{\nabla}\circ (\tau_{b}\times \tau_d)$ and we can rewrite $F(S\amalg T)$ as the composite
        \[
             F(G/H)\xrightarrow{F(\rho_a)\times F(\rho_c)} F(S)\times F(T)\xrightarrow{F(\tau_b)\times F(\tau_{c})} F(G/K)^2\xrightarrow{F(\tau_{\nabla})}F(G/K).
        \]
        Thus, for any $a\in F(G/H)$, there is an isomorphism
        \[
            F(S\amalg T)(a)\cong F(\tau_{\nabla})(F(S)(a),F(T)(a))
        \]
        and we have established $(\star)$ as soon as we see that $F(\tau_{\nabla})$ is a coproduct on $F(G/K)$.  This follows from the observation that $\tau_{\nabla}$ is a left adjoint of $\rho_{\nabla}$ in $\Span(\Set^G)$ (c.f. \cite[p. 63]{DawsonParePronk}).  Since $\rho_{\nabla}$ is the diagonal, and $F$ preserves products, we see that $F(\rho_{\nabla})$ is the diagonal functor $F(G/K)\to F(G/K)^2$.  As pseudofunctors always preserve adjunctions in $2$-categories, it follows  that $F(\tau_{\nabla})$ is the left adjoint of the diagonal, i.e. it is a coproduct.
\end{proof}

The upshot is that one way to produce pseudofunctors $\mathbb{B}^G\to \Wald^{wex}$ is to  instead produce pseudofunctors $\Span(\Set_G)\to \Wald^{wex}$. While this may seem like a more difficult task \textit{a priori}, it is made easier by the universal property of $\Span(\Set_G)$ as shown in \cite{DawsonParePronk}. 
The statements of their results require some preliminary definitions, which we unpack only in our case of interest. 

\begin{definition}
    A weakly exact functor $f\colon \cat C\to \cat D$ in $\Wald^{wex}$ is a \textit{left adjoint} if there is a weakly exact functor $f^*\colon \cat D\to \cat C$ and natural weak equivalences\[
    \eta\colon \id_{\cat C}\Rightarrow f^*f \text{ and } \varepsilon\colon ff^*\Rightarrow \id_{\cat D}
    \] satisfying the usual triangle identities. The 1-cell $f^*$ is called a \textit{right adjoint} of $f$, and is unique up to natural isomorphism. 
\end{definition}

\begin{definition}
    Suppose there is a square in $\Wald^{wex}$

\begin{equation}\label{eq: mate square}
\begin{tikzcd}
	\cat A & \cat B \\
	\cat C & \cat D
	\arrow[""{name=0, anchor=center, inner sep=0}, "f", from=1-1, to=1-2]
	\arrow["h"', from=1-1, to=2-1]
	\arrow["k", from=1-2, to=2-2]
	\arrow[""{name=1, anchor=center, inner sep=0}, "g"', from=2-1, to=2-2]
	\arrow["\alpha"', shorten <=6pt, shorten >=6pt, Rightarrow, from=0, to=1]
\end{tikzcd}\end{equation}
inhabited by a natural weak equivalence $\alpha\colon kf\Rightarrow gh$.  If both $h$ and $k$ are left adjoints, then the \emph{mate} of the square \eqref{eq: mate square} is the square
\[
\begin{tikzcd}
	\cat A & \cat B \\
	\cat C & \cat D
	\arrow[""{name=0, anchor=center, inner sep=0}, "f", from=1-1, to=1-2]
	\arrow["{h^*}", from=2-1, to=1-1]
	\arrow[""{name=1, anchor=center, inner sep=0}, "g"', from=2-1, to=2-2]
	\arrow["{k^*}"', from=2-2, to=1-2]
	\arrow["\beta"', shorten <=6pt, shorten >=6pt, Rightarrow, from=0, to=1]
\end{tikzcd}
=
\begin{tikzcd}
	\cat A & \cat A & \cat B & \cat B \\
	\cat C & \cat C & \cat D & \cat D
	\arrow[""{name=0, anchor=center, inner sep=0}, "{=}", from=1-1, to=1-2]
	\arrow[""{name=1, anchor=center, inner sep=0}, "f", from=1-2, to=1-3]
	\arrow["h"', from=1-2, to=2-2]
	\arrow[""{name=2, anchor=center, inner sep=0}, "{=}", from=1-3, to=1-4]
	\arrow["k"', from=1-3, to=2-3]
	\arrow["{h^*}", from=2-1, to=1-1]
	\arrow[""{name=3, anchor=center, inner sep=0}, "{=}"', from=2-1, to=2-2]
	\arrow[""{name=4, anchor=center, inner sep=0}, "g"', from=2-2, to=2-3]
	\arrow[""{name=5, anchor=center, inner sep=0}, "{=}"', from=2-3, to=2-4]
	\arrow["{k^*}"', from=2-4, to=1-4]
	\arrow["\epsilon"', shorten <=6pt, shorten >=6pt, Rightarrow, from=0, to=3]
	\arrow["\alpha"', shorten <=6pt, shorten >=6pt, Rightarrow, from=1, to=4]
	\arrow["\eta"', shorten <=6pt, shorten >=6pt, Rightarrow, from=2, to=5]
\end{tikzcd}\]
where $\epsilon$ and $\eta$ are the counit and unit of the adjunctions $h\dashv h^*$ and $k\dashv k^*$, respectively. We will occasionally use the notation $\beta = \mathrm{mate}(\alpha)$.
\end{definition}
In formulas, the natural transformation $\beta$ is given by
\[
    fh^*\xRightarrow{\eta\cdot fh^*} k^*kfh^*\xRightarrow{k^*\cdot \alpha\cdot h^*} k^*ghh^*\xRightarrow{k^*g\cdot\epsilon} k^*g.
\]
The mate construction does not, in general, preserve the property of being an isomorphism.  That is, if $\alpha$ is an invertible $2$-cell there is no guarantee that $\mathrm{mate}(\alpha)$ needs to be invertible. For our purposes, we only require the mate to be a weak equivalence.

\begin{definition}
    An invertible $2$-cell $\alpha$ is \emph{weakly sinister} if $\mathrm{mate}(\alpha)$ is defined and is a natural weak equivalence.
\end{definition}

\begin{definition}
 A pseudofunctor $F\colon \Set_G\to \Wald^{wex}$ is \textit{left (resp. right) sinister} if it takes all morphisms in $\Set_G$ to left (resp. right) adjoints. If $F,F'\colon \cat{C}\to \mathcal{B}$ are both left or right sinister, then a pseudonatural transformation $\alpha\colon F\Rightarrow F'$ is \textit{weakly sinister} if for all $f\in \cat{C}$ the invertible $2$-cell $\alpha_f$ is weakly sinister.
\end{definition}

We need one more technical definition, which is the (contravariant) Beck condition. Given a right sinister pseudofunctor $F\colon \Set_G^{\op} \to \Wald^{wex}$ and any pullback square
    \begin{equation}\label{equation: pullback}
        \begin{tikzcd}
            b \ar[r,"f"] \ar[d,swap,"k"] & a \ar[d,"h"]\\
            d\ar[r,swap,"g"] & c
        \end{tikzcd}
    \end{equation}
    in $\Set_G$, there is a square
\begin{equation}\label{equation:Beck square}
\begin{tikzcd}
	Fa & Fb \\
	Fc & Fd
	\arrow[""{name=0, anchor=center, inner sep=0}, "{F(f)}", from=1-1, to=1-2]
	\arrow["{F(h)}", from=2-1, to=1-1]
	\arrow[""{name=1, anchor=center, inner sep=0}, "{F(g)}"', from=2-1, to=2-2]
	\arrow["{F(k)}"', from=2-2, to=1-2]
	\arrow["\mu", shorten <=6pt, shorten >=6pt, Rightarrow, 2tail reversed, from=0, to=1]
\end{tikzcd}
\end{equation}
where $\mu$ is the invertible structure $2$-cell of the pseudofunctor $F$. Since $F$ is right sinister, both vertical maps are right adjoints, with left adjoints $F(k)_!$ and $F(h)_!$. 

\begin{definition}
    A right sinister pseudofunctor $F\colon \Set_G^{\op}\to \Wald^{wex}$ satisfies the \emph{weak Beck condition} if for any pullback square \eqref{equation: pullback} in $\Set_G$ the mate of the corresponding $2$-cell \eqref{equation:Beck square} is weakly sinister.
\end{definition}

Generally speaking, this condition is easiest to remember as the presence of an isomorphism
\[
    F(g)\circ F(k)_!\cong F(h)_!\circ F(f).
\]In fact, the Beck condition is equivalent to $F$ having a ``double coset formula,'' in a way that we will make precise in \cref{corollary: double coset iso implies Beck}. First, we state the result \cite[Proposition 1.10]{DawsonParePronk}, specialized to our situation.

\begin{corollary}\label{corollary: DPP for restrictions}
    The inclusion $j\colon \Set_G^{\op}\to \Span(\Set_G)$ induces an equivalence of categories
    \[
    \mathrm{Fun}^{\rm lax}_{\mathrm{ps}}(\Span(\Set_G),\Wald^{wex})\simeq \mathrm{Fun}^{\rm rsin}_{\rm Beck}(\Set_G^{\op},\Wald^{wex})\] 
    where the category on the right consists of right-sinister pseudofunctors satisfying the Beck condition and weakly sinister pseudonatural transformations.
\end{corollary}

\begin{remark}
The statement above does not precisely match the one in \cite{DawsonParePronk}, as they state the result using the covariant inclusion $i\colon \Set_G\to \Span(\Set_G)$. We note that replacing $i$ with $j$ has the effect of replacing op-lax transformations with lax ones.
\end{remark}

This equivalence of categories sends a pseudofunctor $R\colon\Span(\mathrm{Set}^G) \to \Wald^{wex}$ to the pseudofunctor $Rj\colon \Set_G^{\op}\to \Wald^{wex}$ which agrees with $R$ on objects and sends $X\to Y$ to the associated restriction. A weakly natural transformation of pseudofunctors is sent to a pseudonatural transformation.

\begin{remark}\label{remark: DPP with product preserving}
The equivalence of \cref{corollary: DPP for restrictions} will send product-preserving functors to product-preserving functors. 
This follows from the fact that the inclusion $\Set_G^{\op}\to \Span(\Set_G)$ preserves products and that $F$ and $\tilde{F}$ agree on objects and restriction maps.
\end{remark}

Returning to the problem at hand, producing a weakly natural transformation between two product preserving pseudofunctors $R,P\colon \Span(\Set)^G\to \Wald^{wex}$ has been reduced to producing a pseudonatural transformation $j^*(R)\Rightarrow j^*(P)$ which is weakly sinister.  The next lemma provides an efficient means of checking when such a pseudonatural transformation is weakly sinister.


\begin{lemma}\label{lemma: ps natural transformation lemma}
Suppose that $R,P\colon \Set_G^{\op}\to \Wald^{wex}$ are two pseudofunctors which send disjoint unions to products of categories. A pseudonatural transformation $L\colon R\Rightarrow P$ is weakly sinister if for all canonical quotients $q\colon G/H\to G/K$ in $\Set_G$, the square
\[\begin{tikzcd}
	{R(G/H)} & {P(G/H)} \\
	{R(G/K)} & {P(G/K)}
	\arrow[""{name=0, anchor=center, inner sep=0}, "{L_{G/H}}", from=1-1, to=1-2]
	\arrow["{R(q)}", from=2-1, to=1-1]
	\arrow[""{name=1, anchor=center, inner sep=0}, "{L_{G/K}}"', from=2-1, to=2-2]
	\arrow["{P(q)}"', from=2-2, to=1-2]
	\arrow["{L_q}"', shorten <=6pt, shorten >=6pt, Rightarrow, from=0, to=1]
\end{tikzcd}\]
satisfies the contravariant Beck condition; that is, if $\mathrm{mate}(L_q)$ is a natural weak equivalence.
\end{lemma}
\begin{proof}
Given the assumptions, we want to show that for every $f\colon X\to Y$ in $\Set_G$, $\mathrm{mate}(L_f)$ is a natural weak equivalence. We note that this holds when $f$ is an isomorphism, since the functors $R$ and $P$ must send isomorphisms to categorical equivalences and squares with vertical equivalences always satisfy Beck condition (in particular their mates are isomorphisms). Thus the property of $\mathrm{mate}(L_f)$ being a natural weak equivalence is closed under pre- or post-composition with isomorphisms in $\Set_G$.
    
Recall that every morphism in $\Set_G$ is equivalent to a coproduct of maps of the form
    \[
        \coprod\limits_{i=1}^n G/H_i \xrightarrow{\amalg t_i}\coprod\limits_{i=1}^n G/H\xrightarrow{\nabla} G/H
    \]
    where $H_i\leq H$ for all $i$, $q_i\colon G/H_i\to G/H$ is the canonical quotient, and $\nabla$ is the fold map.  Since $\Wald^{wex}$ is a strict $2$-category, the mate of a composite of natural transformations is the composite of the mates, so the property of $\mathrm{mate}(L_f)$ being weakly invertible is preserved under composition. Condition (1) in the lemma statement is precisely that each $\mathrm{mate}(L_{q_i})$ is a natural weak equivalence, and the fact that $R$ and $P$ send disjoint unions to products implies that that this property is preserved under coproducts. Hence $\mathrm{mate}(L_{\amalg t_i})$ is a natural weak equivalence as well. It remains to show that $\mathrm{mate}(L_\nabla)$ is a natural weak equivalence. 

    For brevity, we write $X=G/H$ and carefully check the binary case $\nabla\colon X\amalg X\to X$. The non-binary case is essentially the same but with more cumbersome notation.  The condition that $R$ and $P$ are product-preserving implies that the composite
    \[
        R(X)\xrightarrow{R(\nabla)} R(X\amalg X)\cong R(X)\times R(X)
    \]
    is the diagonal $\Delta$ on $R(X)$.  Thus its left adjoint is the coproduct functor $\amalg\colon R(X)\to R(X)\times R(X)$. Thus we need to show that the natural transformation defined by the pasting diagram
\[\begin{tikzcd}
	{R(X)\times R(X)} & {R(X)\times R(X)} && {P(X)\times P(X)} \\
	& {R(X)} && {P(X)} & {P(X)}
	\arrow[""{name=0, anchor=center, inner sep=0}, "{=}", from=1-1, to=1-2]
	\arrow["\amalg"', from=1-1, to=2-2]
	\arrow["{L_X\times L_X}", from=1-2, to=1-4]
	\arrow["\oplus", from=1-4, to=2-5]
	\arrow["\Delta"', from=2-2, to=1-2]
	\arrow["{L_X}"', from=2-2, to=2-4]
	\arrow["\Delta", from=2-4, to=1-4]
	\arrow[""{name=1, anchor=center, inner sep=0}, "{=}"', from=2-4, to=2-5]
	\arrow["\eta", shorten <=8pt, shorten >=8pt, Rightarrow, from=0, to=2-2]
	\arrow["\epsilon"', shorten <=7pt, shorten >=7pt, Rightarrow, from=1-4, to=1]
\end{tikzcd}\]
is a natural weak equivalence.

Let us recall the description of the unit $\eta$ and counit $\epsilon$ of the $\amalg\dashv \Delta$ adjunction, which are the same in any category $\cat{C}$ with coproducts.  Suppose that $x,y\in \cat{C}$, then the unit
\[
    (x,y)\xrightarrow{\eta_{x,y}}(x\amalg y,x\amalg y)
\]
is the map $(i_c\colon x\to x\amalg y)\amalg (i_t\colon y\to x\amalg y)$.  The counit
\[
    x\amalg x\xrightarrow{\epsilon_x} x
\]
is the fold map.  For any $(x,y)\in R(X)^2$ the component of the natural transformation we are interested in is given by the composite
\[
    L_X(x)\amalg L_X(y) \xrightarrow{L_X(i_x)\amalg L_X(i_x)} L_X(x\amalg y)\amalg L_X(x\amalg y) \xrightarrow{\nabla} L(x\amalg y)
\]
which is the canonical assembly map.  Since $L_X$ is weakly exact, it preserves coproducts up to weak equivalence which is precisely the claim that this map is a weak equivalence.
\end{proof}


Finally, we may further reduce from $\Set_G$ to $\mathcal{O}_G$. For functors, we have the following observation.

\begin{lemma}\label{lemma: coefficient system to presheaf on SetG}
    Let $i\colon \mathcal{O}_G^{\op}\to (\Set_G)^{\mathrm{op}}$ denote the inclusion.  Given any pseudofunctor $R\colon \mathcal{O}_G^{\op}\to \Wald^{wex}$, there is a product preserving pseudofunctor $S\colon (\Set_G)^{\op}\to \Wald^{wex}$ such that $i^*(S)$ is naturally isomorphic to $R$.
\end{lemma}
\begin{proof}
    This follows from the fact that $\Set_G$ is the finite coproduct completion of $\mathcal{O}_G$, and therefore $\Set_G^{\op}$ is the finite product completion of $\mathcal{O}_G^{\op}$.
\end{proof}  

To construct appropriate morphisms between these functors, our goal is to translate the content of \cref{lemma: ps natural transformation lemma} to consider functors out of $\mathcal{O}_G^{\op}$ rather than $\Set_G^{\op}$. This essentially amounts to checking a ``push-pull isomorphism,'' or double coset formula, in the following way:
Consider a pullback in $\Set_G$,
\[\begin{tikzcd}
	P & {G/H} \\
	{G/K} & {G/L}
	\arrow["{\pi_2}", from=1-1, to=1-2]
	\arrow["{\pi_1}"', from=1-1, to=2-1]
	\arrow["q", from=1-2, to=2-2]
	\arrow["p"', from=2-1, to=2-2]
\end{tikzcd}\]
where $H,K\leq L$ and $p$ and $q$ are the canonical quotients. The $G$-set $P$ can be identified explicitly as
\begin{align*}
    P& \cong G/K\times_{G/L} G/H \cong G\times_{L}(L/K\times L/H)\\
    & \cong G\times_{L}\left(\coprod\limits_{\gamma\in H\backslash L/K} L/(H^{\gamma}\cap K)\right)\\
    & \cong \coprod\limits_{\gamma\in H\backslash L/K} G/(H^{\gamma}\cap K)
\end{align*}
where the second bijection uses the fact that induction preserves pullbacks, the third is the Mackey double coset formula.  With this identification, the map $\pi_1$ is the composite
\[
    \coprod\limits_{\gamma\in H\backslash L/K} G/(H^{\gamma}\cap K)\xrightarrow{\amalg q_\gamma} \coprod\limits_{\gamma\in H\backslash L/K} G/K\xrightarrow{\nabla} G/K
\]
where the maps $q_{\gamma}\colon G/(H^{\gamma}\cap K)\to G/K$ are the canonical quotients.  The map $\pi_2$ is identified with
\[
    \coprod\limits_{\gamma\in H\backslash L/K} G/(H^{\gamma}\cap K)\xrightarrow{c_{\gamma^{-1}}} \coprod\limits_{\gamma\in H\backslash L/K} G/(H\cap K^{\gamma^{-1}})\xrightarrow{\amalg p_\gamma} \coprod\limits_{\gamma\in H\backslash L/K} G/K\xrightarrow{\nabla} G/K
\]
where $p_{\gamma}\colon G/(H\cap K^{\gamma^{-1}})\to G/H$ is the canonical quotient. Thus the original pullback square may be rewritten as \[\begin{tikzcd}[column sep =large]
	\coprod_{\gamma\in H\backslash L/K} G/(H^{\gamma}\cap K) && {G/H} \\
	{G/K} && {G/L}
	\arrow["{\nabla\circ (\amalg p_\gamma)\circ (\amalg c_{\gamma^{-1}})}", from=1-1, to=1-3]
	\arrow["{\nabla\circ (\amalg q_\gamma)}"', from=1-1, to=2-1]
	\arrow["q", from=1-3, to=2-3]
	\arrow["p"', from=2-1, to=2-3]
\end{tikzcd}.\] If $F$ is product preserving, then $F$ will send fold maps to diagonals, hence the left adjoints, which exist because $F$ is right-sinister, must be coproducts. That is, the mate of the square obtained by applying $F$ to the pullback above is \[\begin{tikzcd}[column sep =large]
	\prod_{\gamma\in H\backslash L/K} F(G/(H^{\gamma}\cap K)) && F({G/H}) \\
	F({G/K}) && F({G/L})
	\arrow["{\Pi_{\gamma} F(p_\gamma)\circ F(c_{\gamma^{-1}})}", from=1-1, to=1-3]
	\arrow["{\Pi_{\gamma} F(q_\gamma)_!}"', swap, to=1-1, from=2-1]
	\arrow["q_!", swap, to=1-3, from=2-3]
	\arrow["p"', from=2-1, to=2-3]
\end{tikzcd}.\]

\begin{definition}\label{def:double coset iso}
    We say a right-sinister functor $F\colon \mathcal{O}_G^{\op}\to \Wald^{wex}$ has a \emph{double coset isomorphism} if for every pullback square as above, the natural transformation
\begin{equation}\label{equation: general double coset isomorphism}
    F(p)\circ F(q)_!\Rightarrow \bigoplus_{\gamma\in H\backslash L/K} F(q_{\gamma})_!F(p_{\gamma})F(c_{\gamma^{-1}})  
\end{equation}
is a natural isomorphism, which is the assertion that the functors $F(q)$ and and their left adjoints satisfy the Mackey double coset formula. 
\end{definition}

The purpose of this definition is to obtain the following corollary.

\begin{corollary}\label{corollary: double coset iso implies Beck}
    If $F\colon \mathcal{O}_G^{\op}\to \cat{W}$ is a right-sinister functor with a double coset isomorphism then its product preserving extension $\widetilde{F}\colon(\Set^G)^{\op}\to \cat{W}$ satisfies the Beck condition.
\end{corollary}\begin{proof}
    Let $F\colon \mathcal{O}_G^{\op}\to \cat{W}$ be a right-sinister pseudofunctor and let $\widetilde{F}\colon (\Set^G)^{\op}\to \cat{W}$ be the product preserving functor determined by $F$.  Then $\widetilde{F}$ is certainly right-sinister, and we are interested in identifying conditions on $F$ which guarantee that $\widetilde{F}$ satisfies the Beck condition.
    Up to isomorphism of diagrams, every pullback in $\Set_G$ is a union of pullbacks of the form
\[\begin{tikzcd}
	P & {G/H} \\
	{G/K} & {G/L}
	\arrow["{\pi_2}", from=1-1, to=1-2]
	\arrow["{\pi_1}"', from=1-1, to=2-1]
	\arrow["q", from=1-2, to=2-2]
	\arrow["p"', from=2-1, to=2-2]
\end{tikzcd}\]
where $H,K\leq L$ and $p$ and $q$ are the canonical quotients. Unpacking the mate construction, we see that the Beck condition is the assertion that the natural map in \eqref{equation: general double coset isomorphism} is an isomorphism.
\end{proof}

The next theorem gathers the preceding discussion together in one place for easy reference.

\begin{theorem}\label{Theorem: construction examples and maps}
    Suppose that $R\colon \mathcal{O}_G^{\op}\to \Wald^{wex}$ is a right-sinister pseudofunctor with a double coset formula. Then \begin{enumerate}
        \item[(1)] The assignment $G/H\mapsto R(G/H)$ extends to a pseudofunctor $\overline{R}\colon \mathbb{B}^G\to \Wald^{wex}$ which satisfies $(\star)$. The restriction functors associated to $H\leq K$ are given by $R(q)$ where $q\colon G/H\to G/K$ is the canonical quotient, and the transfers are given by the left adjoints of the restrictions. 
    \end{enumerate}
    Moreover, if $L\colon R\Rightarrow P$ is a pseudonatural transformation between two such pseudofunctors, then
    \begin{enumerate}
        \item[(2)] There exists an extension of $L$ to a lax natural transformation $\overline{L}\colon \overline{R}\Rightarrow \overline{P}$,
        \item[(3)] If for all $H\leq K\leq G$ the mate of the square
\[\begin{tikzcd}
	{R(G/H)} & {P(G/H)} \\
	{R(G/K)} & {P(G/K)}
	\arrow[""{name=0, anchor=center, inner sep=0}, "{L_{G/H}}", from=1-1, to=1-2]
	\arrow["{R(q)}", from=2-1, to=1-1]
	\arrow[""{name=1, anchor=center, inner sep=0}, "{L_{G/K}}"', from=2-1, to=2-2]
	\arrow["{P(q)}"', from=2-2, to=1-2]
	\arrow["{L_q}"', shorten <=6pt, shorten >=6pt, Rightarrow, from=0, to=1]
\end{tikzcd}\]
is weakly invertible, then $\overline{L}$ is a weakly natural transformation, so by \cref{Theorem: weak transformations give maps} it induces a map of $G$-spectra after $K$-theory.
    \end{enumerate}
\end{theorem}
\begin{proof}
    Claims (1) and (2) follow from applying, in order, \cref{lemma: coefficient system to presheaf on SetG}, \cref{corollary: double coset iso implies Beck}, \cref{corollary: DPP for restrictions}, and \cref{lem:BG extend to Span FinG}.  The third claim is \cref{lemma: ps natural transformation lemma} and \cref{corollary: DPP for restrictions}.
\end{proof}
\begin{example}
    The assignment $G/H\mapsto \Coeff^H$ defines a strict $2$-functor, with quotient maps $G/H\to G/K$ being sent to the restriction functors $R^K_H$ of \cref{induction on coefficient systems}, and conjugations being sent to the conjugation functors of \cref{definition: conjugation functors}.  This functor has a double coset isomorphism, given by \cref{Ind Res double coset}.
\end{example}
\begin{example}\label{example: perf and proj have double coset isomorphisms}
    We can extend the previous example in a few ways:\begin{itemize}
        \item[(1)]For any coefficient ring $\underline{S}$, the assignment $G/H\mapsto \Mod_{R^G_H(\underline{S})}$ is a strict $2$-functor $\mathcal{O}_G^{\op}\to \Wald^{wex}$ with the same structure maps as the last example.  Since the left adjoints are also the same as the last example, by \cref{induced functors on projectives}, the double coset isomorphism follows for free.
        \item[(2)] The assignment $G/H\mapsto \mathrm{Proj}_{R^G_H(\underline{S})}$, the category of finitely generated projective $R^G_H(\underline{S})$-modules is a strict $2$ functor with double coset isomorphisms for the same reasons.
        \item[(3)] Similarly, we may consider the assignment $G/H\mapsto \mathrm{Perf}_{R^G_H(\underline{S})}$, the category of perfect $R^G_H(\underline{S})$-complexes.  By \cref{lemma: adjunction on perfect complexes}, induction and restriction, applied to chain complexes pointwise, give an adjunction on this category.  Since the unit and counit are computed at each chain group separately, we obtain the double coset isomorphism once again.
    \end{itemize}   
\end{example}

An immediate consequence of the theorem and \cref{example: perf and proj have double coset isomorphisms} is the following corollary, which proves \cref{spectral Mackey functor}.
\begin{corollary}\label{corollary: proof of proposition from main body}
    Let $\underline{S}$ denote a $G$-coefficient ring and let $\mathcal{P}^G_{\underline{S}}$ denote either the category of finitely generated projective $R^G_H(\underline{S})$-modules or the category of perfect $R^G_H(\underline{S})$-complexes. The assignment $G/H\mapsto \mathcal{P}^H$ extends to a pseudofunctor $\mathbb{B}^G\to \Wald^{wex}$, satisfying condition $(\star)$, where the restrictions are restriction of Mackey functors, transfers are induction of Mackey functors, and conjugations are conjugation of Mackey functors.  
\end{corollary}

\bibliographystyle{alpha}
\bibliography{bibliography}
\end{document}